\pdfoutput=1

\documentclass{article}%
\usepackage{amsmath}
\usepackage{amsfonts}
\usepackage{amssymb}
\usepackage{graphicx}
\usepackage{lscape}%
\providecommand{\U}[1]{\protect\rule{.1in}{.1in}}
\pdfoutput=1
\newtheorem{theorem}{Theorem}

\newtheorem{lemma}[theorem]{Lemma}

\newtheorem{remark}[theorem]{Remark}

\newenvironment{proof}[1][Proof]{\noindent\textbf{#1.} }{\ \rule{0.5em}{0.5em}}
\begin{document}

\title{\textbf{Empirical phi-divergence test statistics for testing simple and
composite null hypotheses}}
\author{N. Balakrishnan$^{1}$, N. Mart\'{\i}n$^{2}$ and L. Pardo$^{3}$
\and $^{1}${\small Department of Mathematics and Statistics, McMaster University,
Hamilton, Canada}
\and $^{2}${\small Department of Statistics, Carlos III University of Madrid,
Getafe (Madrid), Spain}
\and $^{3}${\small Department of Statistics and O.R., Complutense University of
Madrid, Madrid, Spain}}
\date{}
\maketitle

\begin{abstract}
The main purpose of this paper is to introduce first a new family of empirical
test statistics for testing a simple null hypothesis when the vector of
parameters of interest are defined through a specific set of unbiased
estimating functions. This family of test statistics is based on a distance
between two probability vectors, with the first probability vector obtained by
maximizing the empirical likelihood on the vector of parameters, and the
second vector defined from the fixed vector of parameters under the simple
null hypothesis. The distance considered for this purpose is the
phi-divergence measure. The asymptotic distribution is then derived for this
family of test statistics. The proposed methodology is illustrated through the
well-known data of Newcomb's measurements on the passage time for light. A
simulation study is carried out to compare its performance with respect to the
empirical likelihood ratio test when confidence intervals are constructed
based on the respective statistics for small sample sizes. The results suggest
that the \textquotedblleft empirical modified likelihood ratio test
statistic\textquotedblright\ provides a competitive alternative to the
empirical likelihood ratio test statistic, and is also more robust than the
empirical likelihood ratio test statistic in the presence of contamination in
the data. Finally, we propose empirical phi-divergence test statistics for
testing a composite null hypothesis and present some asymptotic as well as
simulation results to study the performance of these test procedures.

\end{abstract}

\bigskip\bigskip



\section{Introduction\label{sec1}}

Empirical likelihood (EL) is a powerful and currently widely used approach for
developing efficient nonparametric statistics. The EL was introduced by Owen
(1988) and since then many papers have appeared on this topic making varied
contributions to different inferential problems. In this approach, the
parameters are usually defined as functionals of the unknown population distribution.

The first purpose of this paper is to introduce a new family of empirical test
statistics as an alternative to the likelihood ratio test statistic proposed
by Qin and Lawless (1994) for testing a simple null hypothesis. As an
extension of the empirical likelihood ratio test of Qin and Lawless (1995), a
new family of empirical test statistics is also considered here for composite
hypothsis. This new family of empirical test statistics is based on divergence measures.

Consider $k$-variate i.i.d. random vectors $\mathbf{X}_{1},...,\mathbf{X}_{n}$
with unknown distribution function $F$, and a $p$-dimensional parameter,
$\boldsymbol{\theta\in\Theta}$, associated with $F$ having finite mean and
non-singular covariance matrix. We assume that all the information about
$\boldsymbol{\theta}$ and $F$ is available in the form of $r\geq p$
functionally independent unbiased estimating functions, through the functions
$g_{j}(\boldsymbol{X},\boldsymbol{\theta})$, $j=1,2,...,r$, such that
$E_{F}\left[  g_{j}(\boldsymbol{X},\boldsymbol{\theta})\right]  =0$. In vector
notation, we have $\boldsymbol{g}(\boldsymbol{X},\boldsymbol{\theta})=\left(
g_{1}(\boldsymbol{X},\boldsymbol{\theta}),...,g_{r}(\boldsymbol{X}%
,\boldsymbol{\theta}\right)  )^{T}$, such that%
\begin{equation}
E_{F}\left[  \boldsymbol{g}(\boldsymbol{X},\boldsymbol{\theta})\right]
=\boldsymbol{0}_{r}. \label{estF}%
\end{equation}
We shall assume that for each realization $\boldsymbol{\boldsymbol{x}}$ of
$\boldsymbol{\boldsymbol{X}}$, $\boldsymbol{g}:$ $\mathbb{R}^{p}%
\mathbb{\rightarrow R}^{r}$ is a vector-valued function and the $r\times p$
matrix
\begin{equation}
\boldsymbol{G}_{\boldsymbol{x}}\left(  \boldsymbol{\theta}\right)
=\frac{\partial\boldsymbol{g}(\boldsymbol{\boldsymbol{x}},\boldsymbol{\theta
})}{\partial\boldsymbol{\theta}^{T}} \label{G}%
\end{equation}
exists. This formulation is as in Qin and Lawless (1994), but a little
different from that of Owen (1988, 1990). The essential difference is that
Owen considered $r=p$ instead of $r\geq p$. For example, by taking into
account that $E_{F}\left[  X\right]  =\theta$, we can adopt steps similar to
the method of moments to obtain estimators through the estimating function
$g_{1}(X,\theta)=X-\theta$ of a univariate distribution. In addition, if we
assume $E_{F}\left[  X^{2}\right]  =2\theta^{2}+1$, the other estimating
function is $g_{2}(X,\theta)=X^{2}-2\theta^{2}-1$, and in this case we have
$r=2>p=k=1$.

Let $\boldsymbol{x}_{1},...,\boldsymbol{x}_{n}$\ be a realization of
$\boldsymbol{X}_{1},...,\boldsymbol{X}_{n}$. The empirical likelihood function
is then given by
\[
\mathcal{L}_{n}(F)=\prod_{i=1}^{n}dF\left(  \boldsymbol{x}_{i}\right)
=\prod_{i=1}^{n}p_{i},
\]
where $p_{i}=dF\left(  \boldsymbol{x}_{i}\right)  =P(\boldsymbol{X}%
=\boldsymbol{x}_{i})$. Only distributions with an atom of probability at each
$\boldsymbol{x}_{i}$ have non-zero likelihood, and without consideration of
estimating functions,\ the empirical likelihood function $\mathcal{L}_{n}(F)$
is seen to be maximized, at $\boldsymbol{X}_{1}=\boldsymbol{x}_{1}%
,...,\boldsymbol{X}_{n}=\boldsymbol{x}_{n}$, by the empirical distribution
function%
\[
F_{n}\left(  \boldsymbol{x}\right)  =\sum\limits_{i=1}^{n}u_{i}%
I(\boldsymbol{X_{i}<x}),
\]
which is associated with the $n$-dimensional discrete uniform distribution%
\[
\boldsymbol{u}=(u_{1},...,u_{n})^{T}=(\tfrac{1}{n},\overset{\overset{n}{\smile
}}{\cdots},\tfrac{1}{n})^{T}.
\]
Now, let%
\[
F_{n,\boldsymbol{\theta}}\left(  \boldsymbol{x}\right)  =\sum\limits_{i=1}%
^{n}p_{i}\left(  \boldsymbol{\theta}\right)  I(\boldsymbol{X_{i}<x})\text{,}%
\]
be an empirical distribution function associated with the probability vector%
\begin{equation}
\boldsymbol{p}(\boldsymbol{\theta})=(p_{1}(\boldsymbol{\theta}),...,p_{n}%
(\boldsymbol{\theta}))^{T},\quad p_{i}\left(  \boldsymbol{\theta}\right)
\geq0,\quad\sum\limits_{i=1}^{n}p_{i}\left(  \boldsymbol{\theta}\right)  =1,
\label{cond}%
\end{equation}
and $\ell_{E,n}(\boldsymbol{\theta})$ be the kernel of the empirical
log-likelihood function, $\sum_{i=1}^{n}\log p_{i}\left(  \boldsymbol{\theta
}\right)  $. If we are interested in maximizing $\ell_{E,n}(\boldsymbol{\theta
})$ subject to the restrictions defined by the estimating functions based on
$F_{n,\boldsymbol{\theta}}\left(  \boldsymbol{x}\right)  $ given by
\begin{equation}
\sum\limits_{i=1}^{n}p_{i}\left(  \boldsymbol{\theta}\right)  \boldsymbol{g}%
(\boldsymbol{X}_{i},\boldsymbol{\theta})=\boldsymbol{0}_{r}, \label{res}%
\end{equation}
we obtain, by applying the Lagrange multipliers method,
\begin{equation}
p_{i}\left(  \boldsymbol{\theta}\right)  =\frac{1}{n}\frac{1}{1+\boldsymbol{t}%
(\boldsymbol{\theta})^{T}\boldsymbol{g}(\boldsymbol{X}_{i},\boldsymbol{\theta
})},\text{ }i=1,...,n, \label{empF2}%
\end{equation}
where $\boldsymbol{t}(\boldsymbol{\theta})$ is an $r$-dimensional vector to be
determined by solving the non-linear system of $r$ equations,%
\begin{equation}
\frac{1}{n}\sum\limits_{i=1}^{n}\frac{1}{1+\boldsymbol{t}(\boldsymbol{\theta
})^{T}\boldsymbol{g}(\boldsymbol{x}_{i},\boldsymbol{\theta})}\boldsymbol{g}%
(\boldsymbol{X}_{i},\boldsymbol{\theta})=\boldsymbol{0}_{r}, \label{ec}%
\end{equation}
subject to (\ref{cond}) and (\ref{empF2}). Thus, the kernel of the empirical
log-likelihood function is%
\begin{equation}
\ell_{E,n}(\boldsymbol{\theta})=-\sum\limits_{i=1}^{n}\log\left[
1+\boldsymbol{t}(\boldsymbol{\theta})^{T}\boldsymbol{g}(\boldsymbol{X}%
_{i},\boldsymbol{\theta})\right]  . \label{ELF}%
\end{equation}

One of the important results\ of Qin and Lawless (1994) is that the empirical
likelihood ratio test statistic for testing%
\begin{equation}
H_{0}\text{: }\boldsymbol{\theta}=\boldsymbol{\theta}_{0}\text{ vs. }%
H_{1}\text{: }\boldsymbol{\theta\neq\theta}_{0} \label{H}%
\end{equation}
is given by
\begin{equation}
L_{E,n}(\widehat{\boldsymbol{\theta}}_{E,n},\boldsymbol{\theta}_{0}%
)=2\ell_{E,n}(\widehat{\boldsymbol{\theta}}_{E,n})-2\ell_{E,n}%
(\boldsymbol{\theta}_{0}), \label{1.1}%
\end{equation}
where $\widehat{\boldsymbol{\theta}}_{E,n}$ is the empirical maximum
likelihood estimator of the parameter $\boldsymbol{\theta}$ obtained by
maximizing $\ell_{E,n}(\boldsymbol{\theta})$ in (\ref{ELF}). In particular, if
$r=p$, it can be seen that $\boldsymbol{t}(\widehat{\boldsymbol{\theta}}%
_{E,n})=\boldsymbol{0}_{r}$, $\boldsymbol{p}(\widehat{\boldsymbol{\theta}%
}_{E,n})=\boldsymbol{u}$ and $\widehat{\boldsymbol{\theta}}_{E,n}$ is the
solution of the system of $r$ equations $\boldsymbol{\bar{g}}_{n}%
(\boldsymbol{\theta})=\boldsymbol{0}_{r}$, where%
\begin{equation}
\boldsymbol{\bar{g}}_{n}(\boldsymbol{\theta})=\frac{1}{n}\sum\limits_{i=1}%
^{n}\boldsymbol{g}(\boldsymbol{X}_{i},\boldsymbol{\theta}), \label{g_bar}%
\end{equation}
subject to (\ref{cond}) and (\ref{empF2}). Furthermore, if $r=p$, $\ell
_{E,n}(\widehat{\boldsymbol{\theta}}_{E,n})=0$ and $L_{E,n}%
(\widehat{\boldsymbol{\theta}}_{E,n},\boldsymbol{\theta}_{0})=-2\ell
_{E,n}(\boldsymbol{\theta}_{0})$. The asymptotic properties of
$\widehat{\boldsymbol{\theta}}_{E,n}$, when $r\geq p$, were studied by Qin and
Lawless (1994). In fact, the asymptotic distribution of $L_{E,n}%
(\widehat{\boldsymbol{\theta}}_{E,n},\boldsymbol{\theta}_{0})$ in (\ref{1.1})
is chi-square with $p$ degrees of freedom.

Here, we propose a new family of test statistics for testing the hypotheses in
(\ref{H}), based on $\phi$-divergence measures, and then derive their
asymptotic distribution. This new family of empirical test statistics is
referred to hereafter as \textquotedblleft empirical $\phi$-divergence test
statistics\textquotedblright. In Section \ref{sec2}, the asymptotic null
distribution of the empirical $\phi$-divergence test statistics is derived.
Then, two power approximations of the empirical $\phi$-divergence test
statistics are presented in Section \ref{sec3}. An illustrative example is
presented in Section \ref{Example}. In Section \ref{Simulation}, a Monte Carlo
simulation study is carried out to compare its performance with respect to the
empirical likelihood ratio test when confidence intervals are constructed
based on the respective statistics for small simple sizes. The results show
that the empirical $\phi$-divergence test statistic is competitive in terms of
power when compared to the empirical likelihood ratio test statistic, and
moreover is more robust than the empirical likelihood ratio test statistic in
the presence of contamination in the data. Next, in Section \ref{secComp}, we
propose empirical phi-divergence test statistics for testing a composite null
hypothesis and present some asymptotic as well as simulation results to
evaluate the performance of these test procedures. Finally, in Section
\ref{LastSec}, we make some concluding remarks.

\section{New family of empirical phi-divergence test statistics\label{sec2a}}

The Kullback-Leibler divergence measure between the probability vectors,
$\boldsymbol{u}$ and $\boldsymbol{p}(\boldsymbol{\theta})$, is given by
\[
\sum\limits_{i=1}^{n}u_{i}\log\frac{u_{i}}{p_{i}(\boldsymbol{\theta})}%
=\sum\limits_{i=1}^{n}\frac{1}{n}\log\frac{\frac{1}{n}}{\frac{1}{n}\frac
{1}{1+\boldsymbol{t}(\boldsymbol{\theta})^{T}\boldsymbol{g}(\boldsymbol{x}%
_{i},\boldsymbol{\theta})}}=\frac{1}{n}\sum\limits_{i=1}^{n}\log\left(
1+\boldsymbol{t}(\boldsymbol{\theta})^{T}\boldsymbol{g}(\boldsymbol{x}%
_{i},\boldsymbol{\theta})\right)  .
\]
In the sequel, we shall denote it by $D_{\mathrm{Kullback}}(F_{n}%
,F_{n,\boldsymbol{\theta}})$, since the above expression is the distance, in
the sense of the\ Kullback-Leibler divergence, between the distributions
functions $F_{n}$ and $F_{n,\boldsymbol{\theta}}$. Using this notation, it is
clear that the empirical likelihood ratio test statistic in (\ref{1.1}) can be
expressed as
\[
L_{E,n}(\widehat{\boldsymbol{\theta}}_{E,n},\boldsymbol{\theta}_{0})=2n\left(
D_{\mathrm{Kullback}}\left(  F_{n},F_{n,\boldsymbol{\theta}_{0}}\right)
-D_{\mathrm{Kullback}}\left(  F_{n},F_{n,\widehat{\boldsymbol{\theta}}_{E,n}%
}\right)  \right)  ,
\]
or equivalently
\[
L_{E,n}\left(  \widehat{\boldsymbol{\theta}}_{E,n},\boldsymbol{\theta}%
_{0}\right)  =\frac{2n}{\phi^{\prime\prime}(1)}\left(  D_{\phi}(F_{n}%
,F_{n,\boldsymbol{\theta}_{0}})-D_{\phi}(F_{n}%
,F_{n,\widehat{\boldsymbol{\theta}}_{E,n}})\right)  ,
\]
where
\begin{equation}
D_{\phi}\left(  F_{n},F_{n,\boldsymbol{\theta}}\right)  =\frac{1}{n}%
\sum\limits_{i=1}^{n}\frac{1}{1+\boldsymbol{t}(\boldsymbol{\theta}%
)^{T}\boldsymbol{g}(\boldsymbol{x}_{i},\boldsymbol{\theta})}\phi\left(
1+\boldsymbol{t}(\boldsymbol{\theta})^{T}\boldsymbol{g}(\boldsymbol{x}%
_{i},\boldsymbol{\theta})\right)  , \label{eqDiv}%
\end{equation}
with $\phi\left(  x\right)  =x\log x-x+1$, $x\in%
\mathbb{R}
$.

Now, we shall denote by $\Phi^{\ast}$ the class of all convex functions
$\phi:\mathbb{R}^{+}\longrightarrow\mathbb{R}$ such that at $x=1$,
$\phi\left(  1\right)  =0$, $\phi^{\prime\prime}\left(  1\right)  >0$, and at
$x=0$, $0\phi\left(  0/0\right)  =0$ and $0\phi\left(  p/0\right)
=p\lim_{u\rightarrow\infty}\frac{\phi\left(  u\right)  }{u}$. Instead of
$\phi\left(  x\right)  =x\log x-x+1$, if we consider a function $\phi$
belonging to $\Phi^{\ast}$, we obtain a new family of test statistics for
testing (\ref{H}) given by
\begin{equation}
T_{n}^{\phi}(\widehat{\boldsymbol{\theta}}_{E,n},\boldsymbol{\theta}%
_{0})=\frac{2n}{\phi^{\prime\prime}(1)}\left(  D_{\phi}(F_{n}%
,F_{n,\boldsymbol{\theta}_{0}})-D_{\phi}(F_{n}%
,F_{n,\widehat{\boldsymbol{\theta}}_{E,n}})\right)  . \label{F1}%
\end{equation}
We will refer to this family of test statistics as empirical $\phi$-divergence
test statistics.

For every $\phi\in\Phi^{\ast}$ that is differentiable at $x=1$, the function
$\psi\left(  x\right)  \equiv\phi(x)-\phi^{\prime}\left(  1\right)  \left(
x-1\right)  $ also belongs to $\Phi^{\ast}$. Then, we have
\[
T_{n}^{\phi}(\widehat{\boldsymbol{\theta}}_{E,n},\boldsymbol{\theta}%
_{0})=T_{n}^{\psi}(\widehat{\boldsymbol{\theta}}_{E,n},\boldsymbol{\theta}%
_{0})
\]
and $\psi$ has the additional property that $\psi^{\prime}\left(  1\right)
=0.$ Since the two divergence measures are equivalent, we can consider the set
$\Phi^{\ast}$ to be equivalent to the set $\Phi=\Phi^{\ast}\cap\left\{
\phi:\phi^{\prime}\left(  1\right)  =0\right\}  $. In what follows, we shall
assume that $\phi\in\Phi$.

The statistics in (\ref{F1}) have been considered recently by Broniatowski and
Keziou (2012) to give some empirical test statistics, generalizing an
important subfamily of test statistics introduced by Baggerly (1998). We
present more details in Section \ref{Simulation} in this regard.

The main purpose of this paper is to present a new family of test statistics
for testing the hypotheses in (\ref{H}) based on the $\phi$-divergence measure
between $F_{n,\boldsymbol{\theta}_{0}}$ and $F_{n,\widehat{\boldsymbol{\theta
}}_{E,n}}$, namely, $D_{\phi}\left(  F_{n,\widehat{\boldsymbol{\theta}}_{E,n}%
},F_{n,\boldsymbol{\theta}_{0}}\right)  $. We shall consider the empirical
family of $\phi$-divergence test statistics given by
\begin{equation}
S_{n}^{\phi}(\widehat{\boldsymbol{\theta}}_{E,n},\boldsymbol{\theta}%
_{0})=\frac{2n}{\phi^{\prime\prime}(1)}D_{\phi}\left(
F_{n,\widehat{\boldsymbol{\theta}}_{E,n}},F_{n,\boldsymbol{\theta}_{0}%
}\right)  =\frac{2}{\phi^{\prime\prime}(1)}\sum\limits_{i=1}^{n}\frac
{1}{1+\boldsymbol{t}(\boldsymbol{\theta}_{0})^{T}\boldsymbol{g}(\boldsymbol{x}%
_{i},\boldsymbol{\theta}_{0})}\phi\left(  \frac{1+\boldsymbol{t}%
(\boldsymbol{\theta}_{0})^{T}\boldsymbol{g}(\boldsymbol{x}_{i}%
,\boldsymbol{\theta}_{0})}{1+\boldsymbol{t}(\widehat{\boldsymbol{\theta}%
}_{E,n})^{T}\boldsymbol{g}(\boldsymbol{x}_{i},\widehat{\boldsymbol{\theta}%
}_{E,n})}\right)  , \label{F2}%
\end{equation}
where $\phi$ is a function satisfying the same conditions as function $\phi$
used to construct $T_{n}^{\phi}(\widehat{\boldsymbol{\theta}}_{E,n}%
,\boldsymbol{\theta}_{0})$. Observe that (\ref{F1}) and (\ref{F2}) are
equivalent only when $r=p$. It is well-known that the family of test
statistics based on $\phi$-divergence has some nice and optimal properties for
different inferential problems, and especially in relation to robustness; see
Pardo (2006) and Basu et al. (2011).

\section{Asymptotic null distribution \label{sec2}}

In this section, we derive the asymptotic distribution of $S_{n}^{\phi
}(\widehat{\boldsymbol{\theta}}_{E,n},\boldsymbol{\theta}_{0})$. For this
purpose, the asymptotic distribution of the maximum empirical likelihood
estimator of the parameter $\boldsymbol{\theta}$, $\widehat{\boldsymbol{\theta
}}_{E,n},$ as well as the asymptotic distribution of $\widehat{\boldsymbol{t}%
}=\boldsymbol{t}(\widehat{\boldsymbol{\theta}}_{E,n})$ are important$.$ These
asymptotic distributions are given in Qin and Lawless (1994), for example.
Under some regularity assumptions (see Lemma 1 and Theorem 1 of Qin and
Lawless (1994)), we have%
\begin{equation}
\sqrt{n}%
\begin{pmatrix}
\widehat{\boldsymbol{\theta}}_{E,n}-\boldsymbol{\theta}_{0}\\
\boldsymbol{t}(\widehat{\boldsymbol{\theta}}_{E,n})
\end{pmatrix}
\overset{\mathcal{L}}{\underset{n\rightarrow\infty}{\longrightarrow}%
}\mathcal{N}\left(  \boldsymbol{0}_{p+r},%
\begin{pmatrix}
\boldsymbol{V}\left(  \boldsymbol{\theta}_{0}\right)  & \boldsymbol{0}%
_{p\times r}\\
\boldsymbol{0}_{r\times p} & \boldsymbol{R}\left(  \boldsymbol{\theta}%
_{0}\right)
\end{pmatrix}
\right)  , \label{dist}%
\end{equation}
where $\overset{\mathcal{L}}{\underset{n\rightarrow\infty}{\longrightarrow}}%
$\ denotes convergence in law and%
\begin{align}
\boldsymbol{V}\left(  \boldsymbol{\theta}_{0}\right)   &  =\left(
\boldsymbol{S}_{12}\left(  \boldsymbol{\theta}_{0}\right)  ^{T}\boldsymbol{S}%
_{11}^{-1}\left(  \boldsymbol{\theta}_{0}\right)  \boldsymbol{S}_{12}\left(
\boldsymbol{\theta}_{0}\right)  \right)  ^{-1},\label{V}\\
\boldsymbol{R}\left(  \boldsymbol{\theta}_{0}\right)   &  =\boldsymbol{S}%
_{11}^{-1}\left(  \boldsymbol{\theta}_{0}\right)  -\boldsymbol{S}_{11}%
^{-1}\left(  \boldsymbol{\theta}_{0}\right)  \boldsymbol{S}_{12}\left(
\boldsymbol{\theta}_{0}\right)  \boldsymbol{V}\left(  \boldsymbol{\theta}%
_{0}\right)  \boldsymbol{S}_{21}\left(  \boldsymbol{\theta}_{0}\right)
\boldsymbol{S}_{11}^{-1}\left(  \boldsymbol{\theta}_{0}\right)  , \label{V2}%
\end{align}
with%
\begin{align}
\boldsymbol{S}_{11}\left(  \boldsymbol{\theta}_{0}\right)   &  =E_{F}\left[
\boldsymbol{g}\left(  \boldsymbol{X},\boldsymbol{\theta}_{0}\right)
\boldsymbol{g}\left(  \boldsymbol{X},\boldsymbol{\theta}_{0}\right)
^{T}\right]  ,\label{S11}\\
\boldsymbol{S}_{12}\left(  \boldsymbol{\theta}_{0}\right)   &  =E_{F}\left[
\boldsymbol{G}_{\boldsymbol{X}}(\boldsymbol{\theta}_{0})\right]
,\quad\boldsymbol{S}_{21}\left(  \boldsymbol{\theta}_{0}\right)
=\boldsymbol{S}_{12}\left(  \boldsymbol{\theta}_{0}\right)  ^{T}. \label{S12}%
\end{align}
This result is derived from:\newline a)
\begin{equation}
\sqrt{n}(\widehat{\boldsymbol{\theta}}_{E,n}-\boldsymbol{\theta}%
_{0})=\boldsymbol{V\left(  \boldsymbol{\theta}_{0}\right)  S}_{12}\left(
\boldsymbol{\theta}_{0}\right)  ^{T}\boldsymbol{S}_{11}^{-1}\left(
\boldsymbol{\theta}_{0}\right)  \sqrt{n}\boldsymbol{\bar{g}}_{n}%
(\boldsymbol{\theta}_{0})+o_{p}(\boldsymbol{1}_{p}), \label{thetaHat}%
\end{equation}
where $\boldsymbol{\bar{g}}_{n}\left(  \boldsymbol{\theta}\right)  $ is given
by (\ref{g_bar}). It is clear from the Central Limit Theorem that
\[
\sqrt{n}\boldsymbol{\bar{g}}_{n}(\boldsymbol{\theta}_{0})\overset{\mathcal{L}%
}{\underset{n\rightarrow\infty}{\longrightarrow}}\mathcal{N}(\boldsymbol{0}%
_{r},\boldsymbol{S}_{11}\left(  \boldsymbol{\theta}_{0}\right)  ),
\]
and so $\sqrt{n}(\widehat{\boldsymbol{\theta}}_{E,n}-\boldsymbol{\theta}%
_{0})\overset{\mathcal{L}}{\underset{n\rightarrow\infty}{\longrightarrow}%
}\mathcal{N}(\boldsymbol{0}_{p},\boldsymbol{V}\left(  \boldsymbol{\theta}%
_{0}\right)  )$;\newline b) $\sqrt{n}\boldsymbol{t}%
(\widehat{\boldsymbol{\theta}}_{E,n})=-\boldsymbol{R}\left(  \boldsymbol{\theta
}_{0}\right)  \sqrt{n}\boldsymbol{\bar{g}}_{n}(\boldsymbol{\theta}_{0}%
)+o_{p}(\boldsymbol{1}_{r})$, where $\boldsymbol{R}\left(  \boldsymbol{\theta
}_{0}\right)  $\ is given by (\ref{V2}) and so $\sqrt{n}\boldsymbol{t}%
(\widehat{\boldsymbol{\theta}}_{E,n})\overset{\mathcal{L}%
}{\underset{n\rightarrow\infty}{\longrightarrow}}\mathcal{N}(\boldsymbol{0}%
_{r},\boldsymbol{R}\left(  \boldsymbol{\theta}_{0}\right)  )$.\newline In
addition, $\widehat{\boldsymbol{\theta}}_{E,n}$ and $\boldsymbol{t}%
(\widehat{\boldsymbol{\theta}}_{E,n})$ are asymptotically uncorrelated.

\begin{lemma}
\label{Lem1}The influence function of the empirical maximum likelihood
estimator of parameter $\boldsymbol{\theta}$, $\widehat{\boldsymbol{\theta}%
}_{E,n}$, is given by%
\begin{equation}
\mathcal{IF}(\boldsymbol{x},\widehat{\boldsymbol{\theta}}_{E,n}%
,F_{n,\boldsymbol{\theta}_{0}})=\boldsymbol{V\left(  \boldsymbol{\theta}%
_{0}\right)  S}_{12}\left(  \boldsymbol{\theta}_{0}\right)  ^{T}%
\boldsymbol{S}_{11}^{-1}\left(  \boldsymbol{\theta}_{0}\right)  \boldsymbol{g}%
(\boldsymbol{x},\boldsymbol{\theta}_{0}). \label{IF}%
\end{equation}

\begin{proof}
It follows from the expression given in (\ref{thetaHat}) and taking into
account the definition of the influence function given in formula (20.1) of
van der Vaart (2000, page 292).
\end{proof}
\end{lemma}

\begin{remark}
The empirical maximum likelihood estimator of parameter $\boldsymbol{\theta}$,
$\widehat{\boldsymbol{\theta}}_{E,n}$, is obtained maximizing the kernel of
empirical log-likelihood function given in (\ref{ELF}) or equivalently
minimizing the function $-\frac{1}{n}%
{\textstyle\sum\nolimits_{i=1}^{n}}
\log\left(  np_{i}\left(  \boldsymbol{\theta}\right)  \right)  $, subject to
the restrictions given in (\ref{res}). This expression can be written as the
$\phi$-divergence measure between the probability vectors $\boldsymbol{u}$ and
$\boldsymbol{p}\left(  \boldsymbol{\theta}\right)  $, i.e. $D_{\phi}\left(
F_{n},F_{n,\boldsymbol{\theta}}\right)  =%
{\textstyle\sum\nolimits_{i=1}^{n}}
p_{i}\left(  \boldsymbol{\theta}\right)  \phi\left(  \frac{1}{np_{i}\left(
\boldsymbol{\theta}\right)  }\right)  $ with $\phi\left(  x\right)  =x\log
x-x+1$. Therefore,%
\[
\widehat{\boldsymbol{\theta}}_{E,n}=\arg\min_{\theta}D_{\phi}\left(
F_{n},F_{n,\boldsymbol{\theta}}\right)  ,\text{\quad}\phi(x)=x\log x-x+1,
\]
subject to the restrictions given in (\ref{res}). If\ we consider a general
function $\phi\in\Phi$, defined in Section \ref{sec2a}, instead of considering
$\phi\left(  x\right)  =x\log x-x+1$, then we can define the empirical minimum
$\phi$-divergence estimator by
\[
\widehat{\boldsymbol{\theta}}_{E\phi,n}=\arg\min_{\boldsymbol{\theta}}D_{\phi
}\left(  F_{n},F_{n,\boldsymbol{\theta}}\right)  ,\text{\quad}\phi\in\Phi,
\]
subject to the restrictions given in (\ref{res}). The empirical exponential
tilting estimator (ET), considered for instance in Schennach (2007), is
defined by $\widehat{\boldsymbol{\theta}}_{ET,n}=\arg\min_{\boldsymbol{\theta
}}%
{\textstyle\sum\nolimits_{i=1}^{n}}
p_{i}\left(  \boldsymbol{\theta}\right)  \log\left(  np_{i}\left(
\boldsymbol{\theta}\right)  \right)  $, subject to the restrictions given in
(\ref{res}). The ET is another member of this family of estimators since
$\widehat{\boldsymbol{\theta}}_{ET,n}=\widehat{\boldsymbol{\theta}}_{E\phi,n}$
with $\phi(x)=-\log x+x-1$. The asymptotic properties of
$\widehat{\boldsymbol{\theta}}_{E\phi,n}$ and $\widehat{\boldsymbol{\theta}%
}_{E,n}$ are the same (for more details, see Ragusa (2011), Broniatowski and
Keziou (2012) and Schennach (2007)). The Fisher consistence of
$\widehat{\boldsymbol{\theta}}_{E\phi,n}$ was established in Ragusa (2011).
Therefore, all the asymptotic results obtained for the test statistics
considered in this paper are valid replacing $\widehat{\boldsymbol{\theta}%
}_{E,n}$ by $\widehat{\boldsymbol{\theta}}_{E\phi,n}$. The expression given in
(\ref{IF}), for the influence function of the empirical maximum likelihood
estimator of parameter $\boldsymbol{\theta}$, can be found for the empirical
minimum $\phi$-divergence estimators, $\widehat{\boldsymbol{\theta}}_{E\phi
,n}$, in Proposition 2.3 of Toma (2013). Hence, all the estimators based on
$\phi$-divergence measures, independently of the $\phi$ function, share the
same influence function.
\end{remark}

Let $\left\Vert \cdot\right\Vert $ denote any vector or matrix norm. We shall
assume the following regularity conditions:

\begin{enumerate}
\item[i)] $\boldsymbol{S}_{11}\left(  \boldsymbol{\theta}_{0}\right)  $ in
(\ref{S11}) is positive definite, and for $\boldsymbol{S}_{12}\left(
\boldsymbol{\theta}_{0}\right)  $ in (\ref{S12}), $\mathrm{rank}%
(\boldsymbol{S}_{12}\left(  \boldsymbol{\theta}_{0}\right)  )=p$.

\item[ii)] There exists a neighbourhood of $\boldsymbol{\theta}_{0}$, in which
$\left\Vert \boldsymbol{g}\left(  \boldsymbol{x},\boldsymbol{\theta}\right)
\right\Vert ^{3}$ is bounded by some integrable function.

\item[iii)] There exists a neighbourhood of $\boldsymbol{\theta}_{0}$, in
which $\boldsymbol{G}_{\boldsymbol{x}}(\boldsymbol{\theta})$, given in
(\ref{G}), is continuous and $\left\Vert \boldsymbol{G}_{\boldsymbol{x}%
}(\boldsymbol{\theta})\right\Vert $ is bounded by some integrable function.

\item[iv)] There exists a neighbourhood of $\boldsymbol{\theta}_{0}$ in which
$\frac{\partial\boldsymbol{G}_{\boldsymbol{x}}(\boldsymbol{\theta})}%
{\partial\boldsymbol{\theta}}$ is continuous and $\left\Vert \frac
{\partial\boldsymbol{G}_{\boldsymbol{x}}(\boldsymbol{\theta})}{\partial
\boldsymbol{\theta}}\right\Vert $ is bounded by some integrable function.
\end{enumerate}

The asymptotic distribution of the empirical $\phi$-divergence test
statistics, $S_{n}^{\phi}(\widehat{\boldsymbol{\theta}}_{E,n}%
,\boldsymbol{\theta}_{0})$, is given in the following theorem.

\begin{theorem}
\label{Th1}Under $H_{0}$ in (\ref{H}) and the assumptions i)-iv) above, we
have%
\[
S_{n}^{\phi}(\widehat{\boldsymbol{\theta}}_{E,n},\boldsymbol{\theta}%
_{0})\overset{\mathcal{L}}{\underset{n\rightarrow\infty}{\longrightarrow}}%
\chi_{p}^{2}.
\]

\end{theorem}

\begin{proof}
Based on the Taylor expansions of (\ref{F2}), as well as the asymptotic
properties of $\widehat{\boldsymbol{\theta}}_{E,n}$, given in Qin and Lawless
(1994), it holds%
\begin{equation}
S_{n}^{\phi}(\widehat{\boldsymbol{\theta}}_{E,n},\boldsymbol{\theta}%
_{0})=\boldsymbol{U}_{n}(\widehat{\boldsymbol{\theta}}_{E,n}%
,\boldsymbol{\theta}_{0})^{T}\boldsymbol{U}_{n}(\widehat{\boldsymbol{\theta}%
}_{E,n},\boldsymbol{\theta}_{0})+o_{p}(1), \label{TaylorS}%
\end{equation}
with $\boldsymbol{U}_{n}(\widehat{\boldsymbol{\theta}}_{E,n}%
,\boldsymbol{\theta}_{0})=\sqrt{n}\boldsymbol{V}^{-1/2}\left(
\boldsymbol{\theta}_{0}\right)  (\widehat{\boldsymbol{\theta}}_{E,n}%
-\boldsymbol{\theta}_{0})\overset{\mathcal{L}}{\underset{n\rightarrow
\infty}{\longrightarrow}}\mathcal{N}(\boldsymbol{0}_{p},\boldsymbol{I}_{p})$,
according to (\ref{dist}).
\end{proof}

\begin{remark}
There are some measures of divergence which can not be expressed as a $\phi
$-divergence measure such as the divergence measures of Bhattacharya (1943),
R\'{e}nyi (1961), and Sharma and Mittal (1977). However, such measures can be
written in the form%
\[
D_{\phi}^{h}\left(  F_{n,\theta},F_{n,\boldsymbol{\theta}_{0}}\right)
=h\left(  D_{\phi}\left(  F_{n,\theta},F_{n,\boldsymbol{\theta}_{0}}\right)
\right)  ,
\]
where $h$ is a differentiable increasing function mapping from $\left[
0,\infty\right)  $ onto $\left[  0,\infty\right)  $, with $h(0)=0$,
$h^{\prime}(0)>0$, and $\phi\in\Phi$. In Table \ref{t1}, these divergence
measures are presented, along with the corresponding expressions of $h$ and
$\phi$.\newline%
\begin{table}[htbp] \tabcolsep0.8pt  \centering
$%
\begin{tabular}
[c]{ccccc}\hline
Divergence & \hspace*{0.5cm} & $h\left(  x\right)  $ & \hspace*{0.5cm} &
$\phi\left(  x\right)  $\\\hline
\multicolumn{1}{l}{R\'{e}nyi} &  & \multicolumn{1}{l}{$\frac{1}{a\left(
a-1\right)  }\log\left(  a\left(  a-1\right)  x+1\right)  ,\quad a\neq0,1$} &
& \multicolumn{1}{l}{$\frac{x^{a}-a\left(  x-1\right)  -1}{a\left(
a-1\right)  },\quad a\neq0,1$}\\
\multicolumn{1}{l}{Sharma-Mittal} &  & \multicolumn{1}{l}{$\frac{1}%
{b-1}\left\{  [1+a\left(  a-1\right)  x]^{\frac{b-1}{a-1}}-1\right\}  ,\quad
b,a\neq1$} &  & \multicolumn{1}{l}{$\frac{x^{a}-a\left(  x-1\right)
-1}{a\left(  a-1\right)  },\quad a\neq0,1$}\\
\multicolumn{1}{l}{Battacharya} &  & \multicolumn{1}{l}{$-\log\left(
-x+1\right)  $} &  & \multicolumn{1}{l}{$-x^{1/2}+\frac{1}{2}\left(
x+1\right)  $}\\\hline
\end{tabular}
\ \ \ \ \ \ \ \ \ \ \ \ \ \ \ $%
\caption{Some specific $(h,\phi)$-divergence measures.\label{t1}}%
\end{table}%
\newline In the case of R\'{e}nyi's divergence, we have
\begin{equation}
D_{\text{\textrm{R\'{e}nyi}}}^{a}\left(  F_{n,\theta},F_{n,\boldsymbol{\theta
}_{0}}\right)  =\frac{1}{a\left(  a-1\right)  }\log\sum\limits_{i=1}^{n}%
\frac{\left(  1+\boldsymbol{t}(\boldsymbol{\theta}_{0})^{T}\boldsymbol{g}%
(\boldsymbol{x}_{i},\boldsymbol{\theta}_{0})\right)  ^{a-1}}{n\left(
1+\boldsymbol{t}(\boldsymbol{\theta})^{T}\boldsymbol{g}(\boldsymbol{x}%
_{i},\boldsymbol{\theta})\right)  ^{a}}\text{, if }a\neq0,1, \label{renyi}%
\end{equation}
and its \textquotedblleft limit\textquotedblright\ cases corresponding to
$a=0$ and $a=1$ as
\begin{align*}
D_{\text{\textrm{R\'{e}nyi}}}^{0}\left(  F_{n,\theta},F_{n,\boldsymbol{\theta
}_{0}}\right)   &  =\lim_{a\rightarrow0}D_{\text{\textrm{R\'{e}nyi}}}\left(
F_{n,\theta},F_{n,\boldsymbol{\theta}_{0}}\right)  =D_{\mathrm{Kullback}%
}\left(  F_{n,\boldsymbol{\theta}_{0}},F_{n,\theta}\right) \\
&  =\sum\limits_{i=1}^{n}\frac{1}{n(1+\boldsymbol{t}(\boldsymbol{\theta}%
)^{T}\boldsymbol{g}(\boldsymbol{x}_{i},\boldsymbol{\theta}))}\log\left(
\frac{1+\boldsymbol{t}(\boldsymbol{\theta}_{0})^{T}\boldsymbol{g}%
(\boldsymbol{x}_{i},\boldsymbol{\theta}_{0})}{1+\boldsymbol{t}%
(\boldsymbol{\theta})^{T}\boldsymbol{g}(\boldsymbol{x}_{i},\boldsymbol{\theta
})}\right)
\end{align*}
and
\begin{align*}
D_{\text{\textrm{R\'{e}nyi}}}^{1}\left(  F_{n,\theta},F_{n,\boldsymbol{\theta
}_{0}}\right)   &  =\lim_{a\rightarrow1}D_{\text{\textrm{R\'{e}nyi}}}\left(
F_{n,\theta},F_{n,\boldsymbol{\theta}_{0}}\right)  =D_{\mathrm{Kullback}%
}\left(  F_{n,\theta},F_{n,\boldsymbol{\theta}_{0}}\right) \\
&  =\sum\limits_{i=1}^{n}\frac{1}{n(1+\boldsymbol{t}(\boldsymbol{\theta}%
_{0})^{T}\boldsymbol{g}(\boldsymbol{x}_{i},\boldsymbol{\theta}_{0}))}%
\log\left(  \frac{1+\boldsymbol{t}(\boldsymbol{\theta})^{T}\boldsymbol{g}%
(\boldsymbol{x}_{i},\boldsymbol{\theta})}{1+\boldsymbol{t}(\boldsymbol{\theta
}_{0})^{T}\boldsymbol{g}(\boldsymbol{x}_{i},\boldsymbol{\theta}_{0})}\right)
,
\end{align*}
respectively.

The $(h,\phi)$-divergence measures were introduced in Men\'{e}ndez et al.
(1995) and some associated asymptotic results were established in Men\'{e}ndez
et al. (1997).
\end{remark}

\begin{theorem}
\label{th3}Under the assumptions of Theorem \ref{Th1}, the asymptotic null
distribution of the family of empirical test statistics
\[
S_{n}^{\phi,h}(\widehat{\boldsymbol{\theta}}_{E,n},\boldsymbol{\theta}%
_{0})=\frac{2n}{\phi^{\prime\prime}(1)h^{\prime}(0)}h\left(  D_{\phi
}(F_{n,\widehat{\boldsymbol{\theta}}_{E,n}},F_{n,\boldsymbol{\theta}_{0}%
})\right)
\]
is chi-squared with $p$ degrees of freedom.
\end{theorem}

\begin{remark}
\label{Remark}Note that we can also consider the family of empirical
$(h,\phi)$-divergence test statistics given by%
\[
T_{n}^{\phi,h}(\widehat{\boldsymbol{\theta}}_{E,n},\boldsymbol{\theta}%
_{0})=\frac{2n}{\phi^{\prime\prime}(1)h^{\prime}(0)}\left(  h\left(  D_{\phi
}(F_{n},F_{n,\boldsymbol{\theta}_{0}})\right)  -h\left(  D_{\phi}%
(F_{n},F_{_{\widehat{\boldsymbol{\theta}}_{E,n}}})\right)  \right)  ,
\]
and their asymptotic distribution is chi-squared with $p$ degrees of freedom
as well. Test-statistics $T_{n}^{\phi,h}(\widehat{\boldsymbol{\theta}}%
_{E,n},\boldsymbol{\theta}_{0})$ and $S_{n}^{\phi,h}%
(\widehat{\boldsymbol{\theta}}_{E,n},\boldsymbol{\theta}_{0})$, include as
particular cases $T_{n}^{\phi}(\widehat{\boldsymbol{\theta}}_{E,n}%
,\boldsymbol{\theta}_{0})$ and $S_{n}^{\phi}(\widehat{\boldsymbol{\theta}%
}_{E,n},\boldsymbol{\theta}_{0})$, taking $h(x)=x$. In the same way done for
$S_{n}^{\phi}(\widehat{\boldsymbol{\theta}}_{E,n},\boldsymbol{\theta}_{0})$ in
(\ref{TaylorS}), it may be concluded\ that all of them have a Taylor expansion
of second order which does not depend on $\phi$ or $h$,%
\[
T_{n}^{\phi,h}(\widehat{\boldsymbol{\theta}}_{E,n},\boldsymbol{\theta}%
_{0})=\boldsymbol{U}_{n}(\widehat{\boldsymbol{\theta}}_{E,n}%
,\boldsymbol{\theta}_{0})^{T}\boldsymbol{U}_{n}(\widehat{\boldsymbol{\theta}%
}_{E,n},\boldsymbol{\theta}_{0})+o_{p}(1),\quad\text{and}\quad S_{n}^{\phi
,h}(\widehat{\boldsymbol{\theta}}_{E,n},\boldsymbol{\theta}_{0}%
)=\boldsymbol{U}_{n}(\widehat{\boldsymbol{\theta}}_{E,n},\boldsymbol{\theta
}_{0})^{T}\boldsymbol{U}_{n}(\widehat{\boldsymbol{\theta}}_{E,n}%
,\boldsymbol{\theta}_{0})+o_{p}(1),
\]
with $\boldsymbol{U}_{n}(\widehat{\boldsymbol{\theta}}_{E,n}%
,\boldsymbol{\theta}_{0})=\sqrt{n}\boldsymbol{V}^{-1/2}\left(
\boldsymbol{\theta}_{0}\right)  (\widehat{\boldsymbol{\theta}}_{E,n}%
-\boldsymbol{\theta}_{0})\overset{\mathcal{L}}{\underset{n\rightarrow
\infty}{\longrightarrow}}\mathcal{N}(\boldsymbol{0}_{r},\boldsymbol{I}_{p})$.
With regard to their influence function, a second order influence function
have to be considered since the first order one vanishes (see van der Vaart
(2000, Section 20.1.1)), and its expression is given by%
\begin{align}
&  \mathcal{IF}_{2}\left(  \boldsymbol{x},T_{n}^{\phi,h}%
(\widehat{\boldsymbol{\theta}}_{E,n},\boldsymbol{\theta}_{0}%
),F_{n,\boldsymbol{\theta}_{0}}\right)  =\mathcal{IF}_{2}\left(
\boldsymbol{x},S_{n}^{\phi,h}(\widehat{\boldsymbol{\theta}}_{E,n}%
,\boldsymbol{\theta}_{0}),F_{n,\boldsymbol{\theta}_{0}}\right)  =\mathcal{IF}%
_{2}\left(  \boldsymbol{x},\left\Vert \boldsymbol{U}_{n}%
(\widehat{\boldsymbol{\theta}}_{E,n},\boldsymbol{\theta}_{0})\right\Vert
^{2},F_{n,\boldsymbol{\theta}_{0}}\right)  \label{IF2}\\
&  =2\left\Vert \boldsymbol{V}^{-1/2}\left(  \boldsymbol{\theta}_{0}\right)
\mathcal{IF}(\boldsymbol{x},\widehat{\boldsymbol{\theta}}_{E,n}%
,F_{n,\boldsymbol{\theta}_{0}})\right\Vert ^{2}=2\boldsymbol{g}(\boldsymbol{x}%
,\boldsymbol{\theta}_{0})^{T}\boldsymbol{S}_{11}^{-1}\left(
\boldsymbol{\theta}_{0}\right)  \boldsymbol{S}_{12}\left(  \boldsymbol{\theta
}_{0}\right)  \boldsymbol{V\left(  \boldsymbol{\theta}_{0}\right)  S}%
_{12}\left(  \boldsymbol{\theta}_{0}\right)  ^{T}\boldsymbol{S}_{11}%
^{-1}\left(  \boldsymbol{\theta}_{0}\right)  \boldsymbol{g}(\boldsymbol{x}%
,\boldsymbol{\theta}_{0}).\nonumber
\end{align}
The first and second equality come from the previous second order Taylor
expansion, the third one from Heritier and Ronchetti (1994) and the last one
from Lemma \ref{Lem1}. It suggests that under the simple null hypothesis, all
members of both family of test-statistics have the same infinitesimal
robustness. This is not a strange result, since in Lemma 1 of Toma (2009) a
similar conclusion has been reached based on power divergence measures for
parametric models.
\end{remark}

\section{Two approximations of power\label{sec3}}

Based on the null distribution presented in Theorem \ref{Th1}, we reject the
null hypothesis in (\ref{H}) in favour of the alternative hypothesis, if
$S_{n}^{\phi}(\widehat{\boldsymbol{\theta}}_{E,n},\boldsymbol{\theta}%
_{0})>\chi_{p,\alpha}^{2}$, where $\chi_{p,\alpha}^{2}$ is the $\left(
1-\alpha\right)  $-th quantile of the chi-squared distribution with $p$
degrees of freedom.

In most cases, the power function of this test procedure can not be derived
explicitly. In the following theorem, we present an asymptotic result, which
provides an approximation of the power of the test.

\begin{theorem}
\label{th5}Let $\boldsymbol{\theta}^{\ast}$ be the true parameter value with
$\boldsymbol{\theta}^{\ast}\neq\boldsymbol{\theta}_{0}$. Under the assumptions
of Theorem \ref{Th1} and $E_{F}\left[  \sup_{\boldsymbol{\theta\in\Theta}%
}\left\Vert \boldsymbol{g}\left(  \boldsymbol{X},\boldsymbol{\theta}\right)
\right\Vert ^{2}\right]  <\infty$, we have%
\[
\sqrt{n}\left(  D_{\phi}(F_{n,\widehat{\boldsymbol{\theta}}_{E,n}%
},F_{n,\boldsymbol{\theta}_{0}})-D_{\phi}(F_{n,\boldsymbol{\theta}^{\ast}%
},F_{n,\boldsymbol{\theta}_{0}})\right)  \overset{\mathcal{L}%
}{\underset{n\rightarrow\infty}{\longrightarrow}}\mathcal{N}\left(
0,\sigma_{\phi}^{2}\left(  \boldsymbol{\theta}^{\ast}%
,\boldsymbol{\boldsymbol{\theta}_{0}}\right)  \right)  ,
\]
where
\[
\sigma_{\phi}^{2}\left(  \boldsymbol{\theta}^{\ast}%
,\boldsymbol{\boldsymbol{\theta}_{0}}\right)  =\boldsymbol{\tau}_{\phi}\left(
\boldsymbol{\theta}^{\ast},\boldsymbol{\boldsymbol{\theta}_{0}}\right)
^{T}\boldsymbol{V\left(  \boldsymbol{\theta}_{0}\right)  \tau}_{\phi}\left(
\boldsymbol{\theta}^{\ast},\boldsymbol{\boldsymbol{\theta}_{0}}\right)
\]
with
\begin{equation}
\boldsymbol{\tau}_{\phi}(\boldsymbol{\theta}^{\ast}%
,\boldsymbol{\boldsymbol{\theta}_{0}})=\left(  \tau_{1}^{\phi}\left(
\boldsymbol{\theta}^{\ast},\boldsymbol{\boldsymbol{\theta}_{0}}\right)
,...,\tau_{p}^{\phi}\left(  \boldsymbol{\theta}^{\ast}%
,\boldsymbol{\boldsymbol{\theta}_{0}}\right)  \right)  ^{T}=\left.
\frac{\partial}{\partial\boldsymbol{\theta}}D_{\phi}(F_{n,\boldsymbol{\theta}%
},F_{n,\boldsymbol{\theta}_{0}})\right\vert _{\boldsymbol{\theta
}=\boldsymbol{\theta}^{\ast}}, \label{tao}%
\end{equation}
and the matrix $\boldsymbol{V\left(  \boldsymbol{\theta}_{0}\right)  }$\ as
defined in (\ref{V}).
\end{theorem}

\begin{proof}
A first-order Taylor expansion gives%
\[
D_{\phi}(F_{n,\widehat{\boldsymbol{\theta}}_{E,n}},F_{n,\boldsymbol{\theta
}_{0}})=D_{\phi}(F_{n,\boldsymbol{\theta}^{\ast}},F_{n,\boldsymbol{\theta}%
_{0}})+\boldsymbol{\tau}_{\phi}\left(  \boldsymbol{\theta}^{\ast
},\boldsymbol{\boldsymbol{\theta}_{0}}\right)  ^{T}%
(\widehat{\boldsymbol{\theta}}_{E,n}-\boldsymbol{\theta}^{\ast}%
)+o(||\widehat{\boldsymbol{\theta}}_{E,n}-\boldsymbol{\theta}^{\ast}||).
\]
But,
\[
\sqrt{n}(\widehat{\boldsymbol{\theta}}_{E,n}-\boldsymbol{\theta}^{\ast
})\overset{\mathcal{L}}{\underset{n\rightarrow\infty}{\longrightarrow}%
}\mathcal{N}(\boldsymbol{0}_{p},\boldsymbol{V}(\boldsymbol{\boldsymbol{\theta
}_{0}})),
\]
and so $\sqrt{n}\,o(||\widehat{\boldsymbol{\theta}}_{E,n}-\boldsymbol{\theta
}^{\ast}||)=o_{p}(\boldsymbol{1}_{r})$. Thus, the random variables
\[
\sqrt{n}\left(  D_{\phi}(F_{n,\widehat{\boldsymbol{\theta}}_{E,n}%
},F_{n,\boldsymbol{\theta}_{0}})-D_{\phi}(F_{n,\boldsymbol{\theta}^{\ast}%
},F_{n,\boldsymbol{\theta}_{0}})\right)  \text{ and }\boldsymbol{\tau}_{\phi
}(\boldsymbol{\theta}^{\ast},\boldsymbol{\boldsymbol{\theta}_{0}})\sqrt
{n}(\widehat{\boldsymbol{\theta}}_{E,n}-\boldsymbol{\theta}^{\ast})
\]
have the same asymptotic distribution, and hence the desired result.
\end{proof}

\begin{remark}
\label{RemPostth5}It is straightforward to extend Theorem \ref{th5} for
$(h,\phi)$-divergence measures $D_{\phi}^{h}\left(  F_{n,\theta}%
,F_{n,\boldsymbol{\theta}_{0}}\right)  =$\linebreak$h\left(  D_{\phi}\left(
F_{n,\theta},F_{n,\boldsymbol{\theta}_{0}}\right)  \right)  $, as follows.
Since $h(x)-h(x_{0})=h^{\prime}(x_{0})(x-x_{0})+o(x-x_{0})$, we have
\begin{align*}
&  \sqrt{n}\left(  D_{\phi}^{h}(F_{n,\widehat{\boldsymbol{\theta}}_{E,n}%
},F_{n,\boldsymbol{\theta}_{0}})-D_{\phi}^{h}(F_{n,\boldsymbol{\theta}^{\ast}%
},F_{n,\boldsymbol{\theta}_{0}})\right) \\
&  =\frac{\partial}{\partial x}\left.  h(x)\right\vert _{x=D_{\phi}%
^{h}(F_{n,\boldsymbol{\theta}^{\ast}},F_{n,\boldsymbol{\theta}_{0}})}(D_{\phi
}((F_{n,\widehat{\boldsymbol{\theta}}_{E,n}},F_{n,\boldsymbol{\theta}_{0}%
})-D_{\phi}(F_{n,\boldsymbol{\theta}^{\ast}},F_{n,\boldsymbol{\theta}_{0}%
}))+o_{p}(1),
\end{align*}
and so%
\[
\sqrt{n}\left(  D_{\phi}^{h}(F_{n,\widehat{\boldsymbol{\theta}}_{E,n}%
},F_{n,\boldsymbol{\theta}_{0}})-D_{\phi}^{h}(F_{n,\boldsymbol{\theta}^{\ast}%
},F_{n,\boldsymbol{\theta}_{0}})\right)  \overset{\mathcal{L}%
}{\underset{n\rightarrow\infty}{\longrightarrow}}\mathcal{N}(0,\sigma_{h,\phi
}^{2}\left(  \boldsymbol{\theta}^{\ast},\boldsymbol{\boldsymbol{\theta}_{0}%
}\right)  ),
\]
where%
\[
\sigma_{h,\phi}^{2}\left(  \boldsymbol{\theta}^{\ast}%
,\boldsymbol{\boldsymbol{\theta}_{0}}\right)  =\kappa_{h,\phi}^{2}\left(
\boldsymbol{\theta}^{\ast},\boldsymbol{\boldsymbol{\theta}_{0}}\right)
\sigma_{\phi}^{2}\left(  \boldsymbol{\theta}^{\ast}%
,\boldsymbol{\boldsymbol{\theta}_{0}}\right)  ,
\]
$\kappa_{h,\phi}\left(  \boldsymbol{\theta}^{\ast}%
,\boldsymbol{\boldsymbol{\theta}_{0}}\right)  =\frac{\partial}{\partial
x}\left.  h(x)\right\vert _{x=D_{\phi}^{h}\left(  F_{n,\boldsymbol{\theta
}^{\ast}},F_{n,\boldsymbol{\theta}_{0}}\right)  }$ and $\sigma_{\phi}%
^{2}\left(  \boldsymbol{\theta}^{\ast},\boldsymbol{\boldsymbol{\theta}_{0}%
}\right)  $\ as defined in Theorem \ref{th5}.
\end{remark}

\begin{remark}
\label{beta1}From Theorem \ref{th5}, we can present a first approximation of
the power function, at $\boldsymbol{\theta}^{\ast}\neq\boldsymbol{\theta}_{0}%
$, of the test based on $\phi$-divergence measure, as%
\begin{equation}
\beta_{n,\phi}^{1}(\boldsymbol{\theta}^{\ast},\boldsymbol{\boldsymbol{\theta
}_{0}})=1-\Phi\left(  \frac{\sqrt{n}}{\sigma_{\phi}\left(  \boldsymbol{\theta
}^{\ast},\boldsymbol{\boldsymbol{\theta}_{0}}\right)  }\left(  \frac
{\phi^{\prime\prime}(1)\chi_{p,\alpha}^{2}}{2n}-D_{\phi}%
(F_{n,\boldsymbol{\theta}^{\ast}},F_{n,\boldsymbol{\theta}_{0}})\right)
\right)  , \label{beta1A}%
\end{equation}
with $\Phi(x)$ being the standard normal distribution function.\newline If
some alternative $\boldsymbol{\theta}^{\ast}\neq\boldsymbol{\theta}_{0}$ is
the true parameter, then the probability of rejecting $\boldsymbol{\theta}%
_{0}$ with the rejection rule $S_{n}^{\phi}(\widehat{\boldsymbol{\theta}%
}_{E,n},\boldsymbol{\theta}_{0})>\chi_{p,\alpha}^{2}$, for fixed significance
level $\alpha$, tends to one as $n\rightarrow\infty$. Thus, the test is
consistent in the sense of Fraser (1957). Similarly, from Theorem \ref{th5},
we obtain a first approximation of the power function, at $\boldsymbol{\theta
}^{\ast}\neq\boldsymbol{\theta}_{0}$, of the test based on $(h,\phi
)$-divergence measure as
\begin{equation}
\beta_{n,h,\phi}^{1}(\boldsymbol{\theta}^{\ast},\boldsymbol{\boldsymbol{\theta
}_{0}})=1-\Phi\left(  \frac{\sqrt{n}}{\sigma_{h,\phi}\left(
\boldsymbol{\theta}^{\ast},\boldsymbol{\boldsymbol{\theta}_{0}}\right)
}\left(  \frac{\phi^{\prime\prime}(1)\chi_{p,\alpha}^{2}}{2n}-D_{\phi}%
^{h}(F_{n,\boldsymbol{\theta}^{\ast}},F_{n,\boldsymbol{\theta}_{0}})\right)
\right)  . \label{beta1B}%
\end{equation}

\end{remark}

To produce some less trivial asymptotic powers that are not all equal to $1$,
we can use a Pitman-type local analysis, as developed by Le Cam (1960), by
confining attention to $n^{1/2}$-neighborhoods of the true parameter values. A
key tool to get the asymptotic distribution of the statistic $S_{n}^{\phi
}(\widehat{\boldsymbol{\theta}}_{E,n},\boldsymbol{\theta}_{0})$ under such a
contiguous hypothesis is Le Cam's third lemma, as presented in H\'{a}jek and
Sid\'{a}k (1967). Instead of relying on these results, we present in the
following theorem a proof which is easy and direct to follow. This proof is
based on the results of Morales and Pardo (2001). Specifically, we consider
the power at contiguous alternatives of the form
\begin{equation}
H_{1,n}:\boldsymbol{\theta}_{n}=\boldsymbol{\theta}_{0}+n^{-1/2}%
\boldsymbol{f}, \label{cont}%
\end{equation}
where $\boldsymbol{f}$ is a fixed vector in $\mathbb{R}^{p}$ such that
$\boldsymbol{\theta}_{n}\in\Theta\subset\mathbb{R}^{p}$.

\begin{theorem}
\label{Th1B}Under the assumptions of Theorem \ref{Th1} and $H_{1,n}$ in
(\ref{cont}), we have the asymptotic distribution of the empirical $\phi
$-divergence test statistic $S_{n}^{\phi}(\widehat{\boldsymbol{\theta}}%
_{E,n},\boldsymbol{\theta}_{0})$ to be non-central chi-squared with $p$
degrees of freedom and non-centrality parameter%
\[
S_{n}^{\phi}(\widehat{\boldsymbol{\theta}}_{E,n},\boldsymbol{\theta}%
_{0})\overset{\mathcal{L}}{\underset{n\rightarrow\infty}{\longrightarrow}}%
\chi_{p}^{2}(\delta\boldsymbol{\left(  \boldsymbol{\theta}_{0}\right)  }),
\]
where
\begin{equation}
\delta\boldsymbol{\left(  \boldsymbol{\theta}_{0}\right)  }=\boldsymbol{f}%
^{T}\boldsymbol{V}^{-1}\boldsymbol{\left(  \boldsymbol{\theta}_{0}\right)  f}.
\label{ncp}%
\end{equation}

\begin{proof}
We can write%
\[
\sqrt{n}(\widehat{\boldsymbol{\theta}}_{E,n}-\boldsymbol{\theta}_{0})=\sqrt
{n}(\widehat{\boldsymbol{\theta}}_{E,n}-\boldsymbol{\theta}_{n})+\sqrt
{n}\left(  \boldsymbol{\theta}_{n}-\boldsymbol{\theta}_{0}\right)  =\sqrt
{n}(\widehat{\boldsymbol{\theta}}_{E,n}-\boldsymbol{\theta}_{n}%
)+\boldsymbol{f.}%
\]
Under $H_{1,n}$, we have
\[
\sqrt{n}(\widehat{\boldsymbol{\theta}}_{E,n}-\boldsymbol{\theta}%
_{n})\overset{\mathcal{L}}{\underset{n\rightarrow\infty}{\longrightarrow}%
}\mathcal{N}(\boldsymbol{0},\boldsymbol{V}\left(  \boldsymbol{\theta}%
_{0}\right)  )
\]
and
\[
\sqrt{n}(\widehat{\boldsymbol{\theta}}_{E,n}-\boldsymbol{\theta}%
_{0})\overset{\mathcal{L}}{\underset{n\rightarrow\infty}{\longrightarrow}%
}\mathcal{N}(\boldsymbol{f},\boldsymbol{V}(\boldsymbol{\theta}_{0}%
))\boldsymbol{.}%
\]
In Theorem \ref{Th1}, it has been shown that
\[
S_{n}^{\phi}(\widehat{\boldsymbol{\theta}}_{E,n},\boldsymbol{\theta}%
_{0})=\left(  \boldsymbol{V}(\boldsymbol{\theta}_{0})^{-1/2}\sqrt
{n}(\widehat{\boldsymbol{\theta}}_{E,n}-\boldsymbol{\theta}_{0})\right)
^{T}\boldsymbol{V}(\boldsymbol{\theta}_{0})^{-1/2}\sqrt{n}%
(\widehat{\boldsymbol{\theta}}_{E,n}-\boldsymbol{\theta}_{0})+o_{p}(1).
\]
On the other hand, we have%
\[
\boldsymbol{V}(\boldsymbol{\theta}_{0})^{-1/2}\sqrt{n}%
(\widehat{\boldsymbol{\theta}}_{E,n}-\boldsymbol{\theta}_{0}%
)\overset{\mathcal{L}}{\underset{n\rightarrow\infty}{\longrightarrow}%
}\mathcal{N}(\boldsymbol{V}(\boldsymbol{\theta}_{0})^{-1/2}\boldsymbol{f}%
,\boldsymbol{I}_{p}).
\]
We thus obtain%
\[
S(F_{n,\widehat{\boldsymbol{\theta}}_{E,n}},F_{n,\boldsymbol{\theta}_{0}%
})\overset{\mathcal{L}}{\underset{n\rightarrow\infty}{\longrightarrow}}%
\chi_{p}^{2}(\delta(\boldsymbol{\boldsymbol{\theta}_{0}})),
\]
with $\delta(\boldsymbol{\boldsymbol{\theta}_{0}})$ as in (\ref{ncp}).
\end{proof}
\end{theorem}

\begin{remark}
Notice that $D_{\phi}^{h}(F_{n,\widehat{\boldsymbol{\theta}}_{E,n}%
},F_{n,\boldsymbol{\theta}_{0}})$ and $h^{\prime}(0)D_{\phi}%
(F_{n,\widehat{\boldsymbol{\theta}}_{E,n}},F_{n,\boldsymbol{\theta}_{0}})$
have the same asymptotic distribution under the contiguous alternative
hypothesis in (\ref{cont}). Hence, it is straightforward to extend Theorem
\ref{Th1B} for $T_{n}^{\phi,h}(\widehat{\boldsymbol{\theta}}_{E,n}%
,\boldsymbol{\theta}_{0})$ and to conclude that both test statistics,
$T_{n}^{\phi,h}(\widehat{\boldsymbol{\theta}}_{E,n},\boldsymbol{\theta}_{0})$
and $S_{n}^{\phi}(\widehat{\boldsymbol{\theta}}_{E,n},\boldsymbol{\theta}%
_{0})$, have the same asymptotic distribution under the contiguous alternative hypothesis.
\end{remark}

\begin{remark}
\label{beta2}Using Theorem \ref{Th1B} and the preceding remark, we obtain a
second approximation to the power function, at $\boldsymbol{\theta}%
_{n}=\boldsymbol{\theta}_{0}+n^{-1/2}\boldsymbol{f}$, as
\begin{align}
\beta_{n}^{2}\left(  \boldsymbol{\theta}_{n}\right)   &  =\Pr(S_{n}^{\phi
}(\widehat{\boldsymbol{\theta}}_{E,n},\boldsymbol{\theta}_{0})>\chi_{p,\alpha
}^{2}|\boldsymbol{\theta}=\boldsymbol{\theta}_{n})\nonumber\\
&  =\Pr(T_{n}^{\phi,h}(\widehat{\boldsymbol{\theta}}_{E,n},\boldsymbol{\theta
}_{0})>\chi_{p,\alpha}^{2}|\boldsymbol{\theta}=\boldsymbol{\theta}%
_{n})=1-F_{\chi_{p}^{2}\left(  \delta(\boldsymbol{\boldsymbol{\theta}_{0}%
})\right)  }\left(  \chi_{p,\alpha}^{2}\right)  , \label{bet2}%
\end{align}
where $F_{\chi_{p}^{2}\left(  \delta(\boldsymbol{\boldsymbol{\theta}_{0}%
})\right)  }$ is the distribution function of $\chi_{p}^{2}\left(
\delta(\boldsymbol{\boldsymbol{\theta}_{0}})\right)  $.\newline If we want to
approximate the power at the alternative $\boldsymbol{\theta}\neq
\boldsymbol{\theta}_{0}$, then we can take $\boldsymbol{f}=\boldsymbol{f}%
(n,\boldsymbol{\theta},\boldsymbol{\theta}_{0})=n^{1/2}\left(
\boldsymbol{\theta}-\boldsymbol{\theta}_{0}\right)  $. It should be noted that
this approximation is independent of the function $\phi$ in $\Phi$, as well as
of the $h$ function.
\end{remark}

\section{Analysis of Newcomb's measurements on the passage time of
light\label{Example}}

In this section, we make use of a well-known data to illustrate the
application of the empirical phi-divergence test statistics for testing a
simple null hypothesis. In 1882, the astronomer and mathematician Simon
Newcomb, measured the time required for a light signal to pass from his
laboratory on the Potomac River to a mirror at the base of the Washington
Monument and back, a distance of $7443.73$ meters. Table \ref{tNewcomb}
contains these measurements from three samples, as deviations from $24,800$
nanoseconds. For example, for the first observation, $28$ represents the
$24,828$ nanoseconds that were spent for the light to travel the required
$7443.73$ meters. The data contain three samples, of sizes $20$, $20$ and
$26$, respectively, corresponding to three different days. They have been
analyzed previously by a number of authors including Stigler (1973) and Voinov
et al. (2013).

Our interest here is in obtaining confidence intervals for the mean, $\mu$, of
the passage time of light, $X$. This means that a unique estimating function,
$g(X,\mu)=X-\mu$, is required\ for a univariate random variable ($k=1$). Note
that $r=p=1$, and so $t(\widehat{\mu}_{E,n})=0$. Based on the results of
Theorem \ref{th3}, we shall use the empirical R\'{e}nyi test statistics%
\begin{equation}
S_{n,a}^{\text{\textrm{R\'{e}nyi}}}\left(  \mu_{0}\right)  =S_{n,a}%
^{\text{\textrm{R\'{e}nyi}}}\left(  F_{n,\widehat{\mu}_{E,n}},F_{n,\mu_{0}%
}\right)  =\left\{
\begin{array}
[c]{ll}%
\frac{2n}{a\left(  a-1\right)  }\log%
{\textstyle\sum\limits_{i=1}^{n}}
\frac{1}{n}\left(  1+t(\mu_{0})(X_{i}-\mu_{0})\right)  ^{a-1}, & a\neq0,1,\\
-2%
{\textstyle\sum\limits_{i=1}^{n}}
\dfrac{\log\left(  1+t(\mu_{0})(X_{i}-\mu_{0})\right)  }{1+t(\mu_{0}%
)(X_{i}-\mu_{0})}, & a=0,\\
2%
{\textstyle\sum\limits_{i=1}^{n}}
\log\left(  1+t(\mu_{0})(X_{i}-\mu_{0})\right)  & a=1,
\end{array}
\right.  \label{RenyiEx}%
\end{equation}
where $t(\mu_{0})$ is the solution of the equation $\frac{1}{n}\sum_{i=1}%
^{n}\frac{X_{i}-\mu_{0}}{1+t(\mu_{0})(X_{i}-\mu_{0})}=0$ subject to
$\min\{t(\mu_{0})(X_{i}-\mu_{0})\}_{i=1,...,n}>\frac{1-n}{n}$ and $\frac{1}%
{n}\sum_{i=1}^{n}\frac{1}{1+t(\mu_{0})(X_{i}-\mu_{0})}=1$.%

\begin{table}[htbp] \tabcolsep2.8pt  \centering
$%
\begin{tabular}
[c]{ccccccccccccccccccccccccccc}\cline{1-21}%
day 1 & $28$ & $26$ & $33$ & $24$ & $34$ & $-44$ & $27$ & $16$ & $40$ & $-2$ &
$29$ & $22$ & $24$ & $21$ & $25$ & $30$ & $23$ & $29$ & $31$ & $19$ &  &  &  &
&  & \\\cline{1-21}%
day 2 & $24$ & $20$ & $36$ & $32$ & $36$ & $28$ & $25$ & $21$ & $28$ & $29$ &
$37$ & $25$ & $28$ & $26$ & $30$ & $32$ & $36$ & $26$ & $30$ & $22$ &  &  &  &
&  & \\\hline
day 3 & $36$ & $23$ & $27$ & $27$ & $28$ & $27$ & $31$ & $27$ & $26$ & $33$ &
$26$ & $32$ & $32$ & $24$ & $39$ & $28$ & $24$ & $25$ & $32$ & $25$ & $29$ &
$27$ & $28$ & $29$ & $16$ & $23$\\\hline
\end{tabular}
\ \ $\caption{Newcomb's data on the passage time of light.\label{tNewcomb}}%
\end{table}%

A confidence interval, with confidence level $1-\alpha$, based on
(\ref{RenyiEx})\ is given by%
\[
CI_{1-\alpha}^{a}(\mu)=\{\mu\in%
\mathbb{R}
:S_{n,a}^{\text{\textrm{R\'{e}nyi}}}\left(  \mu\right)  \leq\chi_{1,\alpha
}^{2}\},
\]
with $\chi_{1,\alpha}^{2}$ being the (right hand) quantile of order $\alpha$
for the $\chi_{1}^{2}$\ distribution. By taking different choices for the
parameter $a$ as $\{-1,-0.5,0,0.5,1,1.5,2.5\}$, we obtain different confidence
intervals, and this is done for each of the three days. Table \ref{tCI}
summarizes these results obtained at $95\%$ confidence level. If the
confidence interval based on the Likelihood Ratio test statistic, i.e.,
$CI_{0.95}^{1}(\mu)$ are compared with $CI_{0.95}^{a}(\mu)$, we observe that
narrower confidence intervals are obtained for $a<1$, and wider confidence
intervals are obtained for $a>1$. This behavior is more evident for the first
day, due to the presence of some outliers in the sample; see for example,
Bernett and Lewis (1993). Next, we approximate the power function, assuming
that $\mu_{0}$ is a value very close to $\widehat{\mu}_{E,n}$. Despite the
fact that the exact significance may not be approximated well (since the true
value $\mu_{0}$ is unknown), the power values can provide meaningful
information on the quality of performance of the different R\'{e}nyi test statistics.%

\begin{table}[htbp] \tabcolsep2.8pt  \centering
$%
\begin{tabular}
[c]{clcccccc}\hline
& $CI_{0.95}^{-1}(\mu)$ & $CI_{0.95}^{-0.5}(\mu)$ & $CI_{0.95}^{0}(\mu)$ &
$CI_{0.95}^{0.5}(\mu)$ & $CI_{0.95}^{1}(\mu)$ & $CI_{0.95}^{1.5}(\mu)$ &
$CI_{0.95}^{2.5}(\mu)$\\\hline
day 1 & \multicolumn{1}{c}{$(14.22,27.24)$} & $(13.65,27.25)$ &
$(10.85,27.11)$ & $(12.00,27.24)$ & $(12.92,27.24)$ & $(9.46,26.97)$ &
$(6.90,26.60)$\\\hline
day 2 & \multicolumn{1}{c}{$(26.48,30.70)$} & $(26.47,30.70)$ &
$(26.40,30.78)$ & $(26.43,30.74)$ & $(26.45,30.71)$ & $(26.35,30.83)$ &
$(26.24,30.97)$\\\hline
day 3 & \multicolumn{1}{c}{$(26.24,29.42)$} & $(26.27,29.48)$ &
$(26.00,29.70)$ & $(26.10,29.61)$ & $(26.20,29.54)$ & $(25.88,29.79)$ &
$(25.62,30.01)$\\\hline
\end{tabular}
\ \ \ \ $%
\caption{Confidence intervals, based on the empirical R�nyi test statistics, for Newcomb's data for the three days.\label{tCI}}%
\end{table}%

In order to calculate the power function by using the first approximation%
\[
\beta_{n,a}^{1,\text{\textrm{R\'{e}nyi}}}\left(  \mu^{\ast}\right)
=1-\Phi\left(  \frac{\sqrt{n}}{\sigma_{\text{\textrm{R\'{e}nyi}}}^{a}\left(
\mu_{0},\mu^{\ast}\right)  }\left(  \frac{\chi_{1,\alpha}^{2}}{2n}%
-D_{\text{\textrm{R\'{e}nyi}}}^{a}(F_{n,\mu^{\ast}},F_{n,\mu_{0}})\right)
\right)  ,
\]
we need to obtain
\[
\sigma_{\text{\textrm{R\'{e}nyi}}}^{a}\left(  \mu_{0},\mu^{\ast}\right)
=\kappa_{\text{\textrm{R\'{e}nyi}}}^{a}\left(  \mu_{0},\mu^{\ast}\right)
\left\vert \tau_{\text{\textrm{R\'{e}nyi}}}^{a}\left(  \mu_{0},\mu^{\ast
}\right)  \right\vert V^{\frac{1}{2}}\left(  \mu_{0}\right)
\]
in terms of%
\begin{align*}
\kappa_{\text{\textrm{R\'{e}nyi}}}^{a}\left(  \mu_{0},\mu^{\ast}\right)   &
=\frac{\partial}{\partial x}\left.  h(x)\right\vert _{x=D_{\phi_{a}}\left(
F_{n,\mu^{\ast}},F_{n,\mu_{0}}\right)  }=\left\{
\begin{array}
[c]{ll}%
\frac{1}{a\left(  a-1\right)  D_{\phi_{a}}\left(  F_{n,\mu^{\ast}}%
,F_{n,\mu_{0}}\right)  +1}, & a\left(  a-1\right)  \neq0,\\
1, & a\left(  a-1\right)  =0,
\end{array}
\right. \\
&  =\left(  \sum\limits_{i=1}^{n}p_{i}^{a}(\mu^{\ast})p_{i}^{1-a}(\mu
_{0})\right)  ^{-1},
\end{align*}%
\[
p_{i}(\mu)=\frac{1}{n}\frac{1}{1+t(\mu)(X_{i}-\mu)},\quad\mu\in\{\mu_{0}%
,\mu^{\ast}\},
\]%
\[
\tau_{\text{\textrm{R\'{e}nyi}}}^{a}\left(  \mu_{0},\mu^{\ast}\right)
=\left\{
\begin{array}
[c]{ll}%
\dfrac{n}{1-a}%
{\displaystyle\sum\limits_{i=1}^{n}}
\left[  t^{\prime}(\mu^{\ast})(X_{i}-\mu^{\ast})-t(\mu^{\ast})\right]
p_{i}^{a+1}(\mu^{\ast})p_{i}^{1-a}(\mu_{0}), & a\neq1,\\
n%
{\displaystyle\sum\limits_{i=1}^{n}}
\left[  t^{\prime}(\mu^{\ast})(X_{i}-\mu^{\ast})-t(\mu^{\ast})\right]
p_{i}^{2}(\mu^{\ast})\left(  \log\frac{p_{i}(\mu_{0})}{p_{i}(\mu^{\ast}%
)}-1\right)  , & a=1,
\end{array}
\right.
\]%
\[
t^{\prime}(\mu^{\ast})=-\left(  \sum_{j=1}^{n}p_{j}^{2}(\mu^{\ast})\right)
\left(  \sum_{j=1}^{n}(X_{j}-\mu^{\ast})^{2}p_{j}^{2}(\mu^{\ast})\right)
^{-1},
\]
and%
\[
V\left(  \mu_{0}\right)  =\left(  S_{12}^{T}\left(  \mu_{0}\right)
S_{11}^{-1}\left(  \mu_{0}\right)  S_{12}\left(  \mu_{0}\right)  \right)
^{-1}=E\left[  (X-\mu_{0})^{2}\right]  .
\]
Since $V\left(  \mu_{0}\right)  $ is unknown, it can be replaced by a
consistent estimator%
\[
\widetilde{V}\left(  \mu_{0}\right)  =\sum_{j=1}^{n}(X_{j}-\mu_{0})^{2}%
p_{j}(\mu_{0}),
\]
so that%
\[
\widetilde{\beta}_{n,a}^{1}\left(  \mu^{\ast}\right)  =1-\Phi\left(
\frac{\sqrt{n}}{\widetilde{\sigma}_{\text{\textrm{R\'{e}nyi}}}^{a}\left(
\mu_{0},\mu^{\ast}\right)  }\left(  \frac{\chi_{1,\alpha}^{2}}{2n}%
-D_{\text{\textrm{R\'{e}nyi}}}^{a}(F_{n,\mu^{\ast}},F_{n,\mu_{0}})\right)
\right)  ,
\]
where%
\[
\widetilde{\sigma}_{\text{\textrm{R\'{e}nyi}}}^{a}\left(  \mu_{0},\mu^{\ast
}\right)  =\kappa_{\text{\textrm{R\'{e}nyi}}}^{a}\left(  \mu_{0},\mu^{\ast
}\right)  \left\vert \tau_{\text{\textrm{R\'{e}nyi}}}^{a}\left(  \mu_{0}%
,\mu^{\ast}\right)  \right\vert \widetilde{V}^{\frac{1}{2}}\left(  \mu
_{0}\right)  ,
\]%
\[
D_{\text{\textrm{R\'{e}nyi}}}^{a}\left(  F_{n,\mu^{\ast}},F_{n,\mu_{0}%
}\right)  =\left\{
\begin{array}
[c]{ll}%
-\frac{1}{a\left(  a-1\right)  }\log\kappa_{\text{\textrm{R\'{e}nyi}}}%
^{a}\left(  \mu_{0},\mu^{\ast}\right)  , & a\left(  a-1\right)  \neq0,\\
\sum_{i=1}^{n}p_{i}(\mu_{0})\log\frac{p_{i}(\mu_{0})}{p_{i}(\mu^{\ast})}, &
a=0,\\
\sum_{i=1}^{n}p_{i}(\mu^{\ast})\log\frac{p_{i}(\mu^{\ast})}{p_{i}(\mu_{0})}, &
a=1.
\end{array}
\right.
\]
In Figure \ref{fig000}, the power functions are plotted for the three days,
for three different values of the parameter $a\in\{-1,1,2.5\}$. We do observe
from these plots that the most powerful test statistics are for the case when
$a>1$, which is in fact the case when the confidence intervals are wider.

Notice that there are two outliers on Day 1 ($-44$ and $-2$) and one possible
outlier on Day 3 ($16$); see Stigler (1973) and Barnett and Lewis (1994). By
dropping these outliers, we reproduced the plots of approximated power
functions for Days 1 and 3, and we found that in the new plots the
approximated power functions are more similar to that of Day 2. In the
presence of outliers (Days 1 and 3), the comparison of shape of the power
function for different values provides an interesting information: the power
function is flatter for $S_{n,a}^{\text{\textrm{R\'{e}nyi}}}\left(
\mu\right)  $, with $a=-1$ in comparison with $a=1$, $2$. This tends to be
associated with higher capacity of the empirical likelihood ratio test
($S_{n,a}^{\text{\textrm{R\'{e}nyi}}}\left(  \mu\right)  $, with $a=1$) for
detecting samples not fulfilling the null hypothesis, but a lack of robustness
in the presence of outliers. In Section \ref{Simulation2}, a simulation study
is carried out to examine this robustness aspect.

\begin{center}%
\begin{figure}  \centering
$%
\begin{tabular}
[c]{cc}%
\raisebox{-0cm}{\includegraphics[
natheight=2.491500in,
natwidth=3.746400in,
height=4.815cm,
width=7.2071cm
]%
{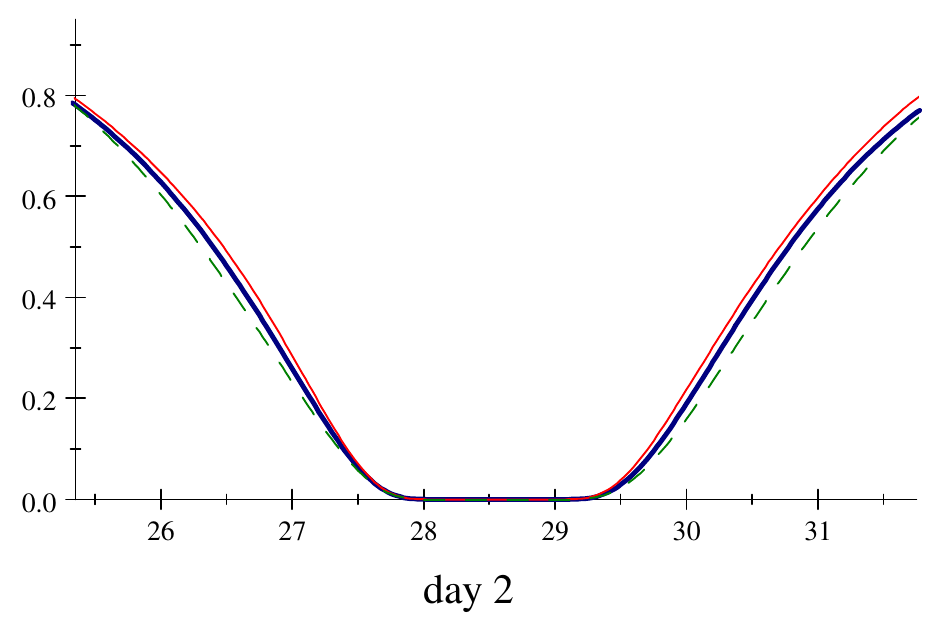}%
}
&
\raisebox{1.0032in}{\fbox{\includegraphics[
natheight=0.677200in,
natwidth=1.152800in,
height=0.7083in,
width=1.1865in
]%
{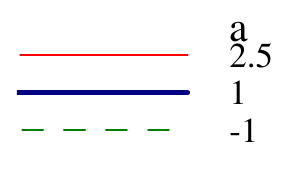}%
}}
\\%
{\includegraphics[
natheight=2.491500in,
natwidth=3.746400in,
height=1.8957in,
width=2.8374in
]%
{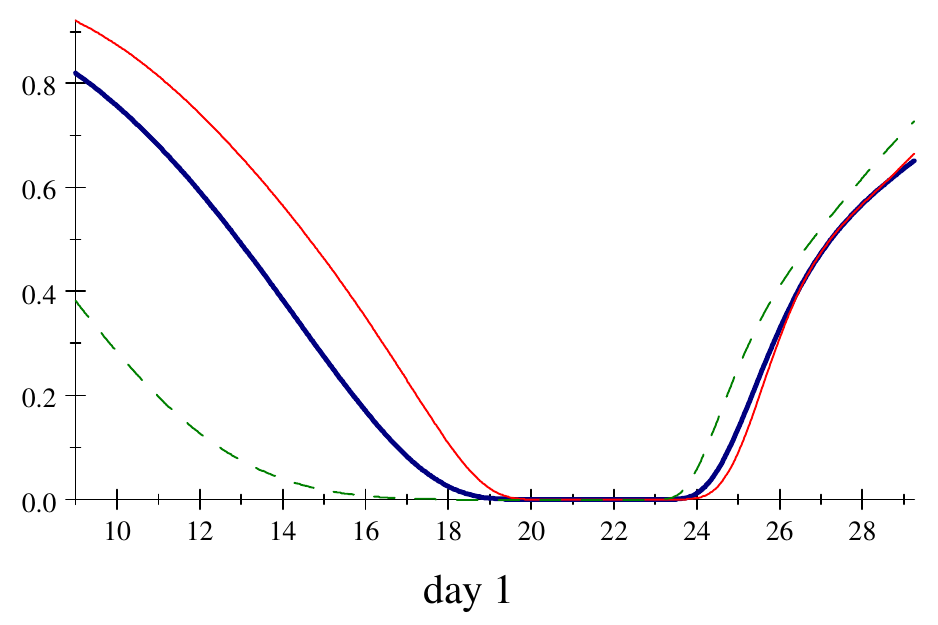}%
}
&
{\includegraphics[
natheight=2.491500in,
natwidth=3.746400in,
height=1.8957in,
width=2.8374in
]%
{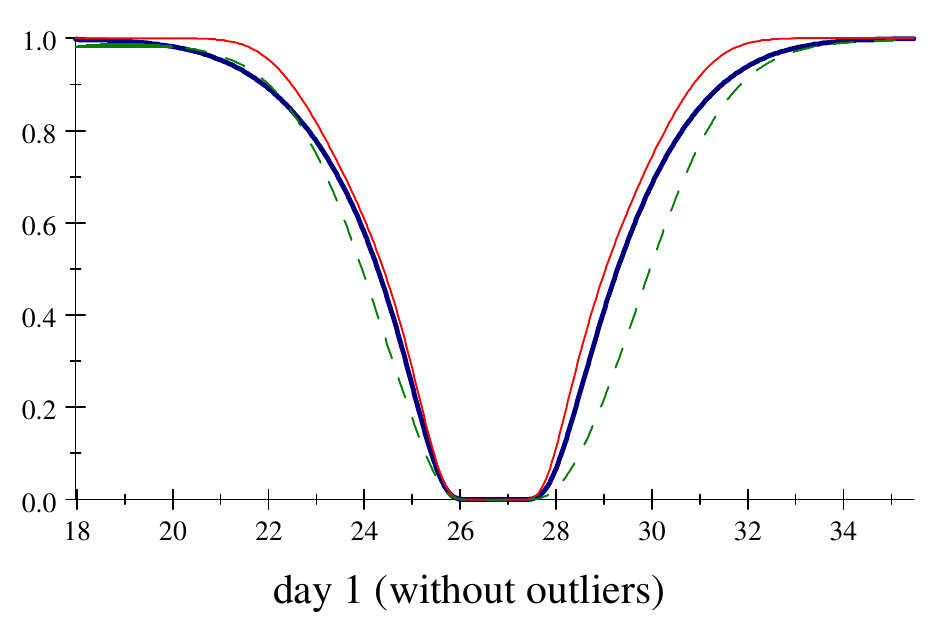}%
}
\\%
{\includegraphics[
natheight=2.491500in,
natwidth=3.746400in,
height=1.8957in,
width=2.8374in
]%
{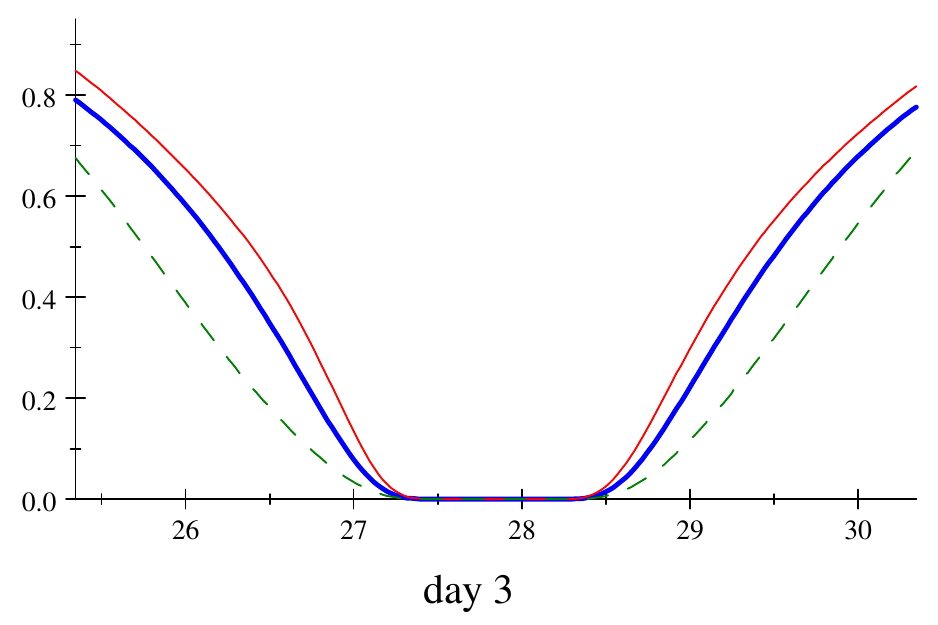}%
}
&
{\includegraphics[
natheight=2.491500in,
natwidth=3.746400in,
height=1.8957in,
width=2.8374in
]%
{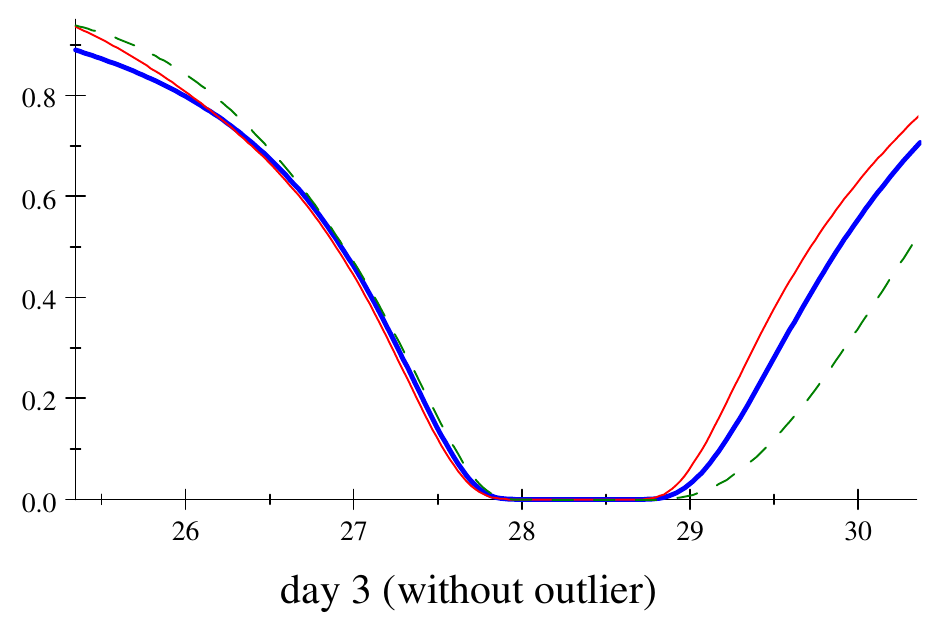}%
}
\end{tabular}
$%
\caption{Approximate power functions $\protect\widetilde{\beta}_{n,a}^{1}\left(  \mu^{\ast}\right)  $ for Newcomb's measurements for three days.\label{fig000}}%
\end{figure}%

\end{center}

\section{Simulation study\label{Simulation}}

\subsection{Evaluation of procedures under normality\label{Simulation1}}

In this section, we will pay special attention to a subfamily of $\phi
$-divergence measures in (\ref{eqDiv}), the so-called power divergence
measures (see Cressie and Read (1984)) for which $\phi(x)\equiv\phi_{\lambda
}(x)\in\Phi$ is given by%
\begin{align}
\phi_{\lambda}(x)  &  =\frac{1}{\lambda(1+\lambda)}\left[  x^{\lambda
+1}-x-\lambda(x-1)\right]  ,\quad\text{if }\lambda(1+\lambda)\neq
0,\label{PD0}\\
\phi_{0}(x)  &  =\lim_{\lambda\rightarrow0}\phi_{\lambda}(x)=x\log
x-x+1,\nonumber\\
\phi_{-1}(x)  &  =\lim_{\lambda\rightarrow-1}\phi_{\lambda}(x)=-\log
x+x-1.\nonumber
\end{align}
In this case, (\ref{F1}) can be expressed as%
\begin{align}
T_{n}^{\lambda}(\widehat{\boldsymbol{\theta}}_{E,n},\boldsymbol{\theta}_{0})
&  =\frac{2}{\lambda(1+\lambda)}\left(  \sum_{i=1}^{n}(1+\boldsymbol{t}%
(\boldsymbol{\theta}_{0})^{T}\boldsymbol{g}(\boldsymbol{X}_{i}%
,\boldsymbol{\theta}_{0}))^{\lambda}-\sum_{i=1}^{n}(1+\boldsymbol{t}%
(\widehat{\boldsymbol{\theta}}_{E,n})^{T}\boldsymbol{g}(\boldsymbol{X}%
_{i},\widehat{\boldsymbol{\theta}}_{E,n}))^{\lambda}\right)  \text{,}%
\quad\text{if }\lambda(1+\lambda)\neq0,\label{PD1}\\
&  =2\sum_{i=1}^{n}\log\left(  \frac{1+\boldsymbol{t}(\boldsymbol{\theta}%
_{0})^{T}\boldsymbol{g}(\boldsymbol{X}_{i},\boldsymbol{\theta}_{0}%
)}{1+\boldsymbol{t}(\widehat{\boldsymbol{\theta}}_{E,n})^{T}\boldsymbol{g}%
(\boldsymbol{X}_{i},\widehat{\boldsymbol{\theta}}_{E,n})}\right)
=L_{E,n}(\widehat{\boldsymbol{\theta}}_{E},\boldsymbol{\theta}_{0}%
),\quad\text{if }\lambda=0,\nonumber\\
&  =2\left(  \sum_{i=1}^{n}\frac{\log(1+\boldsymbol{t}%
(\widehat{\boldsymbol{\theta}}_{E,n})^{T}\boldsymbol{g}(\boldsymbol{X}%
_{i},\widehat{\boldsymbol{\theta}}_{E,n})))}{1+\boldsymbol{t}%
(\widehat{\boldsymbol{\theta}}_{E,n})^{T}\boldsymbol{g}(\boldsymbol{X}%
_{i},\widehat{\boldsymbol{\theta}}_{E,n})}-\sum_{i=1}^{n}\frac{\log
(1+\boldsymbol{t}(\boldsymbol{\theta}_{0})^{T}\boldsymbol{g}(\boldsymbol{X}%
_{i},\boldsymbol{\theta}_{0})))}{1+\boldsymbol{t}(\boldsymbol{\theta}_{0}%
)^{T}\boldsymbol{g}(\boldsymbol{X}_{i},\boldsymbol{\theta}_{0})}\right)
\text{,}\quad\text{if }\lambda=-1.\nonumber
\end{align}
This family of test statistics include several commonly used tests such as the
empirical modified Kullback-Leibler statistic ($T_{n}^{\lambda}%
(\widehat{\boldsymbol{\theta}}_{E,n},\boldsymbol{\theta}_{0})$, with
$\lambda=-1$), the empirical Freeman-Tukey statistic ($T_{n}^{\lambda
}(\widehat{\boldsymbol{\theta}}_{E,n},\boldsymbol{\theta}_{0})$, with
$\lambda=-\frac{1}{2}$), the empirical likelihood ratio test statistic
($T_{n}^{\lambda}(\widehat{\boldsymbol{\theta}}_{E,n},\boldsymbol{\theta}%
_{0})$, with $\lambda=0$), the empirical Cressie-Read statistic ($T_{n}%
^{\lambda}(\widehat{\boldsymbol{\theta}}_{E,n},\boldsymbol{\theta}_{0})$, with
$\frac{2}{3}$), and the empirical Pearson's chi-square statistic
($T_{n}^{\lambda}(\widehat{\boldsymbol{\theta}}_{E,n},\boldsymbol{\theta}%
_{0})$, with $\lambda=1$). Similarly, (\ref{F2}) can be expressed as
\begin{align}
S_{n}^{\lambda}(\widehat{\boldsymbol{\theta}}_{E,n},\boldsymbol{\theta}_{0})
&  =\frac{2}{\lambda(1+\lambda)}\left(  \sum_{i=1}^{n}\frac{(1+\boldsymbol{t}%
(\boldsymbol{\theta}_{0})^{T}\boldsymbol{g}(\boldsymbol{X}_{i}%
,\boldsymbol{\theta}_{0}))^{\lambda}}{(1+\boldsymbol{t}%
(\widehat{\boldsymbol{\theta}}_{E,n})^{T}\boldsymbol{g}(\boldsymbol{X}%
_{i},\widehat{\boldsymbol{\theta}}_{E,n}))^{\lambda+1}}-n\right)
\text{,}\quad\text{if }\lambda(1+\lambda)\neq0,\label{PD2}\\
&  =2\sum_{i=1}^{n}\frac{1}{1+\boldsymbol{t}(\widehat{\boldsymbol{\theta}%
}_{E,n})^{T}\boldsymbol{g}(\boldsymbol{X}_{i},\widehat{\boldsymbol{\theta}%
}_{E,n})}\log\left(  \frac{1+\boldsymbol{t}(\boldsymbol{\theta}_{0}%
)^{T}\boldsymbol{g}(\boldsymbol{X}_{i},\boldsymbol{\theta}_{0})}%
{1+\boldsymbol{t}(\widehat{\boldsymbol{\theta}}_{E,n})^{T}\boldsymbol{g}%
(\boldsymbol{X}_{i},\widehat{\boldsymbol{\theta}}_{E,n})}\right)
\text{,}\quad\text{if }\lambda=0,\nonumber\\
&  =2\sum_{i=1}^{n}\frac{1}{1+\boldsymbol{t}(\boldsymbol{\theta}_{0}%
)^{T}\boldsymbol{g}(\boldsymbol{X}_{i},\boldsymbol{\theta}_{0})}\log\left(
\frac{1+\boldsymbol{t}(\widehat{\boldsymbol{\theta}}_{E,n})^{T}\boldsymbol{g}%
(\boldsymbol{X}_{i},\widehat{\boldsymbol{\theta}}_{E,n})}{1+\boldsymbol{t}%
(\boldsymbol{\theta}_{0})^{T}\boldsymbol{g}(\boldsymbol{X}_{i}%
,\boldsymbol{\theta}_{0})}\right)  \text{,}\quad\text{if }\lambda=-1.\nonumber
\end{align}
The $r$-dimensional vectors $\boldsymbol{t}(\boldsymbol{\theta}_{0})$ and
$\boldsymbol{t}(\widehat{\boldsymbol{\theta}}_{E,n})$\ are obtained by solving
(\ref{ec}) with $\boldsymbol{\theta}_{0}$\ and $\widehat{\boldsymbol{\theta}%
}_{E,n}$\ being specified.

We now examine, through Monte Carlo simulations, the performance of the
confidence intervals (CIs) obtained through the empirical power divergence
test statistics in (\ref{PD1}) and (\ref{PD2}) for the choices $\lambda
\in\{-1,-\frac{1}{2},0,\frac{2}{3},1\}$, when the sample sizes are small,
$n=30,45,60$, and the nominal confidence levels are $90\%$ and $95\%$. Even
though the asymptotic distribution is equivalent for the different empirical
divergence based statistics, in practice we need to work with finite sample
sizes and it is therefore important to compare the CIs based on $T_{n}%
^{\lambda}(\widehat{\boldsymbol{\theta}}_{E,n},\boldsymbol{\theta}_{0})$ and
$S_{n}^{\lambda}(\widehat{\boldsymbol{\theta}}_{E,n},\boldsymbol{\theta}_{0}%
)$, in different scenarios. We focus on the coverage probability and average
width of the CIs of a mean based on $5000$ simulated samples under the same
setting as Example 1 of Qin and Lawless (1994). In this case, a sample of
univariate ($k=1$) i.i.d. random variables, $X_{1},...,X_{n}$, of size $n$ was
considered, with mean $\theta$ and variance $\theta^{2}+1$, i.e., $r=2>p=1$
and%
\[
\boldsymbol{g}(X,\theta)=(g_{1}(X,\theta),g_{2}(X,\theta))^{T}=(X-\theta
,X^{2}-2\theta^{2}-1)^{T}.
\]%
\begin{table}[htbp] \scriptsize\tabcolsep0.8pt  \centering
$%
\begin{tabular}
[c]{llllccccccc}
&  &  &  & \multicolumn{7}{c}{$\mathcal{N}(\theta,\theta^{2}+1)$, $\theta
_{0}=0$}\\\hline
& \hspace*{0.35cm} &  & \hspace*{0.35cm} & \multicolumn{3}{c}{$T_{n}^{\lambda
}(\widehat{\theta}_{E,n},\theta_{0})$} & \hspace*{0.35cm} &
\multicolumn{3}{c}{$S_{n}^{\lambda}(\widehat{\theta}_{E,n},\theta_{0})$%
}\\\cline{4-7}\cline{9-11}
&  &  &  & Cov (\%) & \hspace*{0.35cm} & Avw &  & Cov (\%) & \hspace*{0.35cm}
& Avw\\\hline
&  & \multicolumn{1}{r}{$\lambda$} &  & \multicolumn{7}{c}{$n=30$}\\
$90\%$ &  & \multicolumn{1}{r}{$-1$} &  & 86.36 &  & \textbf{0.524} &  &
\textbf{88.35} &  & \textbf{0.540}\\
&  & \multicolumn{1}{r}{$-\frac{1}{2}$} &  & 87.11 &  & 0.542 &  & 87.98 &  &
0.551\\
&  & \multicolumn{1}{r}{$0$} &  & \textbf{87.77} &  & 0.557 &  & 87.64 &  &
0.560\\
&  & \multicolumn{1}{r}{$\frac{2}{3}$} &  & 87.46 &  & 0.571 &  & 86.88 &  &
0.570\\
&  & \multicolumn{1}{r}{$1$} &  & 87.19 &  & 0.575 &  & 86.55 &  & 0.573\\
$95\%$ &  & \multicolumn{1}{r}{$-1$} &  & 91.97 &  & \textbf{0.617} &  &
\textbf{93.85} &  & \textbf{0.637}\\
&  & \multicolumn{1}{r}{$-\frac{1}{2}$} &  & 92.67 &  & 0.642 &  & 93.77 &  &
0.654\\
&  & \multicolumn{1}{r}{$0$} &  & \textbf{93.05} &  & 0.663 &  & 93.48 &  &
0.669\\
&  & \multicolumn{1}{r}{$\frac{2}{3}$} &  & \textbf{93.05} &  & 0.684 &  &
92.76 &  & 0.684\\
&  & \multicolumn{1}{r}{$1$} &  & 92.88 &  & 0.690 &  & 92.28 &  & 0.689\\
&  & \multicolumn{1}{r}{$\lambda$} &  & \multicolumn{7}{c}{$n=45$}\\
$90\%$ &  & \multicolumn{1}{r}{$-1$} &  & 87.62 &  & \textbf{0.445} &  &
\textbf{88.95} &  & \textbf{0.455}\\
&  & \multicolumn{1}{r}{$-\frac{1}{2}$} &  & 88.45 &  & 0.457 &  & 88.85 &  &
0.462\\
&  & \multicolumn{1}{r}{$0$} &  & \textbf{88.83} &  & 0.469 &  & 88.66 &  &
0.468\\
&  & \multicolumn{1}{r}{$\frac{2}{3}$} &  & 88.70 &  & 0.480 &  & 88.11 &  &
0.474\\
&  & \multicolumn{1}{r}{$1$} &  & 88.51 &  & 0.484 &  & 87.81 &  & 0.476\\
$95\%$ &  & \multicolumn{1}{r}{$-1$} &  & 93.09 &  & \textbf{0.526} &  &
\textbf{94.42} &  & \textbf{0.538}\\
&  & \multicolumn{1}{r}{$-\frac{1}{2}$} &  & 93.43 &  & 0.544 &  & 94.34 &  &
0.550\\
&  & \multicolumn{1}{r}{$0$} &  & 94.02 &  & 0.559 &  & 94.13 &  & 0.559\\
&  & \multicolumn{1}{r}{$\frac{2}{3}$} &  & \textbf{94.23} &  & 0.575 &  &
93.51 &  & 0.569\\
&  & \multicolumn{1}{r}{$1$} &  & 93.81 &  & 0.581 &  & 93.22 &  & 0.573\\
&  & \multicolumn{1}{r}{$\lambda$} &  & \multicolumn{7}{c}{$n=60$}\\
$90\%$ &  & \multicolumn{1}{r}{$-1$} &  & 87.19 &  & \textbf{0.394} &  &
\textbf{88.47} &  & \textbf{0.401}\\
&  & \multicolumn{1}{r}{$-\frac{1}{2}$} &  & 87.86 &  & 0.404 &  & 88.45 &  &
0.405\\
&  & \multicolumn{1}{r}{$0$} &  & \textbf{88.36} &  & 0.413 &  & 88.13 &  &
0.409\\
&  & \multicolumn{1}{r}{$\frac{2}{3}$} &  & 88.28 &  & 0.421 &  & 87.63 &  &
0.413\\
&  & \multicolumn{1}{r}{$1$} &  & 88.03 &  & 0.425 &  & 87.29 &  & 0.415\\
$95\%$ &  & \multicolumn{1}{r}{$-1$} &  & 93.22 &  & \textbf{0.467} &  &
\textbf{93.85} &  & \textbf{0.475}\\
&  & \multicolumn{1}{r}{$-\frac{1}{2}$} &  & 93.45 &  & 0.481 &  & 93.72 &  &
0.483\\
&  & \multicolumn{1}{r}{$0$} &  & \textbf{93.95} &  & 0.493 &  & 93.55 &  &
0.489\\
&  & \multicolumn{1}{r}{$\frac{2}{3}$} &  & 93.78 &  & 0.505 &  & 93.13 &  &
0.496\\
&  & \multicolumn{1}{r}{$1$} &  & 93.70 &  & 0.509 &  & 92.90 &  &
0.498\\\hline
\end{tabular}
\ \ \ \ \ \ \ \ \ \ \ \ \qquad%
\begin{tabular}
[c]{llllccccccc}
&  &  &  & \multicolumn{7}{c}{$\mathcal{N}(\theta,\theta^{2}+1)$, $\theta
_{0}=1$}\\\hline
& \hspace*{0.35cm} &  & \hspace*{0.35cm} & \multicolumn{3}{c}{$T_{n}^{\lambda
}(\widehat{\theta}_{E,n},\theta_{0})$} & \hspace*{0.35cm} &
\multicolumn{3}{c}{$S_{n}^{\lambda}(\widehat{\theta}_{E,n},\theta_{0})$%
}\\\cline{4-7}\cline{9-11}
&  &  &  & Cov (\%) & \hspace*{0.35cm} & Avw &  & Cov (\%) & \hspace*{0.35cm}
& Avw\\\hline
&  & \multicolumn{1}{r}{$\lambda$} &  & \multicolumn{7}{c}{$n=30$}\\
$90\%$ &  & \multicolumn{1}{r}{$-1$} &  & 83.59 &  & 0.627 &  & 79.37 &  &
\textbf{0.542}\\
&  & \multicolumn{1}{r}{$-\frac{1}{2}$} &  & 84.64 &  & 0.635 &  & 79.73 &  &
0.547\\
&  & \multicolumn{1}{r}{$0$} &  & \textbf{85.16} &  & 0.642 &  & 80.11 &  &
0.551\\
&  & \multicolumn{1}{r}{$\frac{2}{3}$} &  & 83.01 &  & 0.608 &  &
\textbf{80.51} &  & 0.555\\
&  & \multicolumn{1}{r}{$1$} &  & 82.81 &  & \textbf{0.605} &  & 80.39 &  &
0.556\\
$95\%$ &  & \multicolumn{1}{r}{$-1$} &  & 88.70 &  & 0.728 &  & 84.36 &  &
\textbf{0.645}\\
&  & \multicolumn{1}{r}{$-\frac{1}{2}$} &  & 89.83 &  & 0.740 &  & 85.08 &  &
0.652\\
&  & \multicolumn{1}{r}{$0$} &  & \textbf{90.39} &  & 0.749 &  & 85.52 &  &
0.657\\
&  & \multicolumn{1}{r}{$\frac{2}{3}$} &  & 89.08 &  & 0.721 &  &
\textbf{85.88} &  & 0.663\\
&  & \multicolumn{1}{r}{$1$} &  & 88.62 &  & \textbf{0.714} &  & 85.82 &  &
0.664\\
&  & \multicolumn{1}{r}{$\lambda$} &  & \multicolumn{7}{c}{$n=45$}\\
$90\%$ &  & \multicolumn{1}{r}{$-1$} &  & 85.13 &  & \textbf{0.476} &  &
84.44 &  & \textbf{0.460}\\
&  & \multicolumn{1}{r}{$-\frac{1}{2}$} &  & 86.03 &  & 0.482 &  & 85.19 &  &
0.464\\
&  & \multicolumn{1}{r}{$0$} &  & \textbf{86.46} &  & 0.488 &  & 85.73 &  &
0.468\\
&  & \multicolumn{1}{r}{$\frac{2}{3}$} &  & 86.20 &  & 0.485 &  &
\textbf{86.03} &  & 0.473\\
&  & \multicolumn{1}{r}{$1$} &  & 86.20 &  & 0.487 &  & 86.01 &  & 0.475\\
$95\%$ &  & \multicolumn{1}{r}{$-1$} &  & 90.37 &  & \textbf{0.563} &  &
89.73 &  & \textbf{0.546}\\
&  & \multicolumn{1}{r}{$-\frac{1}{2}$} &  & 91.12 &  & 0.571 &  & 90.18 &  &
0.553\\
&  & \multicolumn{1}{r}{$0$} &  & 91.45 &  & 0.579 &  & 90.94 &  & 0.559\\
&  & \multicolumn{1}{r}{$\frac{2}{3}$} &  & 91.53 &  & 0.580 &  & 91.02 &  &
0.567\\
&  & \multicolumn{1}{r}{$1$} &  & \textbf{91.62} &  & 0.582 &  &
\textbf{91.19} &  & 0.570\\
&  & \multicolumn{1}{r}{$\lambda$} &  & \multicolumn{7}{c}{$n=60$}\\
$90\%$ &  & \multicolumn{1}{r}{$-1$} &  & 85.46 &  & \textbf{0.413} &  &
84.91 &  & \textbf{0.403}\\
&  & \multicolumn{1}{r}{$-\frac{1}{2}$} &  & 85.99 &  & 0.417 &  & 85.46 &  &
0.406\\
&  & \multicolumn{1}{r}{$0$} &  & 86.28 &  & 0.421 &  & 85.80 &  & 0.410\\
&  & \multicolumn{1}{r}{$\frac{2}{3}$} &  & 86.37 &  & 0.423 &  & 85.90 &  &
0.415\\
&  & \multicolumn{1}{r}{$1$} &  & \textbf{86.39} &  & 0.425 &  &
\textbf{86.09} &  & 0.417\\
$95\%$ &  & \multicolumn{1}{r}{$-1$} &  & 91.02 &  & \textbf{0.489} &  &
90.83 &  & \textbf{0.479}\\
&  & \multicolumn{1}{r}{$-\frac{1}{2}$} &  & 92.08 &  & 0.496 &  & 91.80 &  &
0.485\\
&  & \multicolumn{1}{r}{$0$} &  & 92.58 &  & 0.502 &  & 92.25 &  & 0.490\\
&  & \multicolumn{1}{r}{$\frac{2}{3}$} &  & \textbf{92.92} &  & 0.506 &  &
\textbf{92.56} &  & 0.498\\
&  & \multicolumn{1}{r}{$1$} &  & 92.86 &  & 0.509 &  & 92.54 &  &
0.501\\\hline
\end{tabular}
\ \ \ \ \ \ \ \ \ \ \ \ $%
\caption{Coverage and average width of $90\%$ and $95\%$ confidence intervals.\label{t2}}%
\end{table}
\newline In the present study, we have considered an additional sample of size
$n=45$, compared to those of Qin and Lawless (1994). For constructing the test
statistic, an unknown distribution of $X_{i}$ is assumed, but for this
simulation study, normally distributed random variables were taken with
$\theta_{0}\in\{0,1\}$. It is important to mention here that Baggerly (1998)
studied the coverage probabilities for the normal distribution, but only for
the family of statistics in (\ref{PD1}). The results of the simulation study
are summarized in Table \ref{t2} and Figure \ref{fig1} for the empirical power
divergence test statistics in (\ref{PD1}) and (\ref{PD2}) for two different
samples when the nominal confidence level is $95\%$. In Figure \ref{fig1}, it
can be seen that the confidence intervals based on both statistics,
$T_{n}^{\lambda}(\widehat{\theta}_{E,n},\theta_{0})$\ and $S_{n}^{\lambda
}(\widehat{\theta}_{E,n},\theta_{0})$, with small values of $\lambda
\in\{-1,-\frac{1}{2},0,\frac{2}{3},1\}$ tend to be narrower in width, and this
is also seen in the simulation results of Table \ref{t2}. In fact, the
simulation study of Baggerly showed that $T_{n}^{-1}(\widehat{\theta}%
_{E,n},\theta_{0})$ had the worst coverage probability, and this is also seen
in the present simulation study. This is in accordance with the results
obtained for two specific samples analyzed in Figure \ref{fig1}. The smallest
average width of confidence intervals is obtained for $S_{n}^{-1}%
(\widehat{\theta}_{E,n},\theta_{0})$, which is the so-called empirical
modified likelihood ratio test statistic or empirical minimum discrimination
information statistic (see Gokhale and Kullback, 1978), and this might
explain, when $\theta_{0}=0$, the low coverage probability as well. Often, the
coverage probability closest to the nominal level is also obtained for the
same test statistics, $S_{n}^{-1}(\widehat{\theta}_{E,n},\theta_{0})$ when
$\theta_{0}=1$ but the empirical likelihood ratio test has a closer coverage
probability when $\theta_{0}=0$. When both characteristics of a CI, viz.,
average width and coverage probability, are taken into account, the empirical
modified likelihood ratio test statistic, $S_{n}^{-1}(\widehat{\theta}%
_{E,n},\theta_{0})$, is a good compromise and it turns out to be the best for
$\theta_{0}=0$ in all the cases considered. Finally, another important feature
seen in Figure \ref{fig1} is that even though the support of the asymptotic
distribution of the test statistics is strictly positive, it is possible to
find samples for which $T_{n}^{\lambda}(\widehat{\theta}_{E,n},\theta_{0})$ is
negative, while $S_{n}^{\lambda}(\widehat{\theta}_{E,n},\theta_{0})$ is always
strictly positive. This usually happens when $\widehat{\theta}_{E,n}$\ is very
close to $\theta_{0}$.%

\begin{figure}  \centering
$%
\begin{tabular}
[c]{cc}%
\multicolumn{2}{c}{%
{\fbox{\includegraphics[
natheight=1.218500in,
natwidth=1.540200in,
height=1.0637in,
width=1.337in
]%
{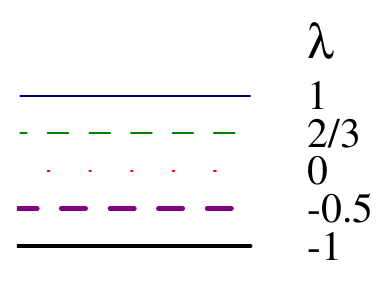}%
}}
}\\%
{\includegraphics[
natheight=2.965400in,
natwidth=4.483200in,
height=2.252in,
width=3.3909in
]%
{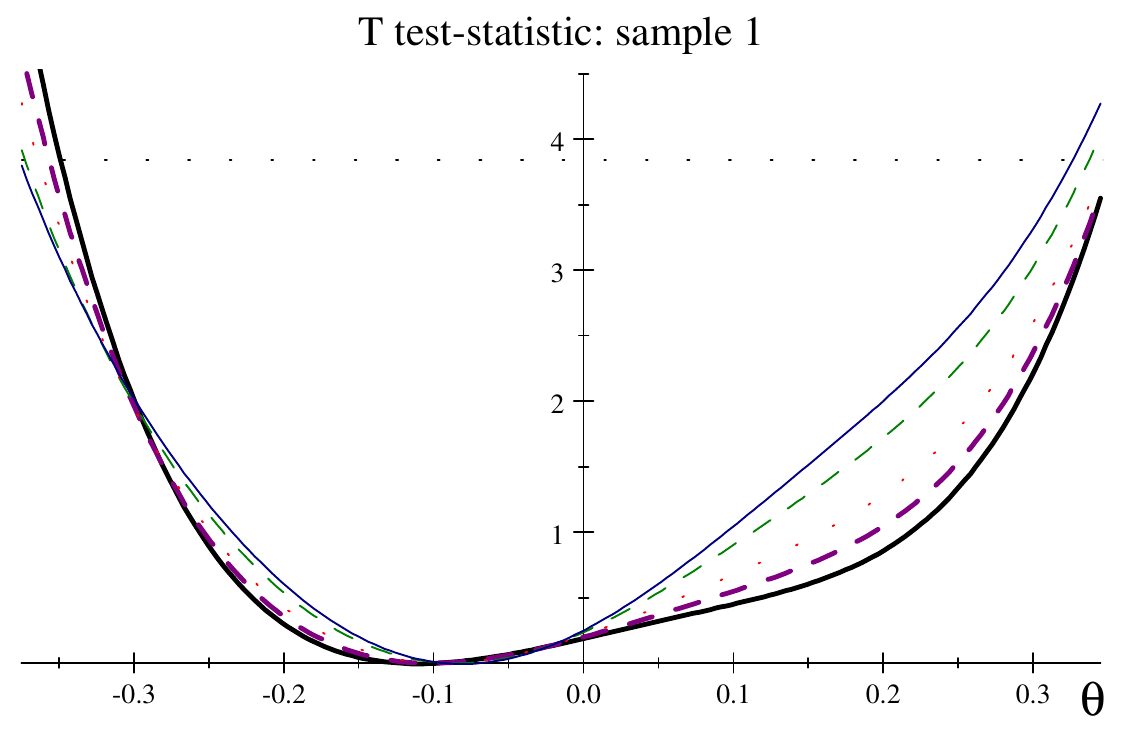}%
}
&
{\includegraphics[
natheight=2.965400in,
natwidth=4.483200in,
height=2.252in,
width=3.3909in
]%
{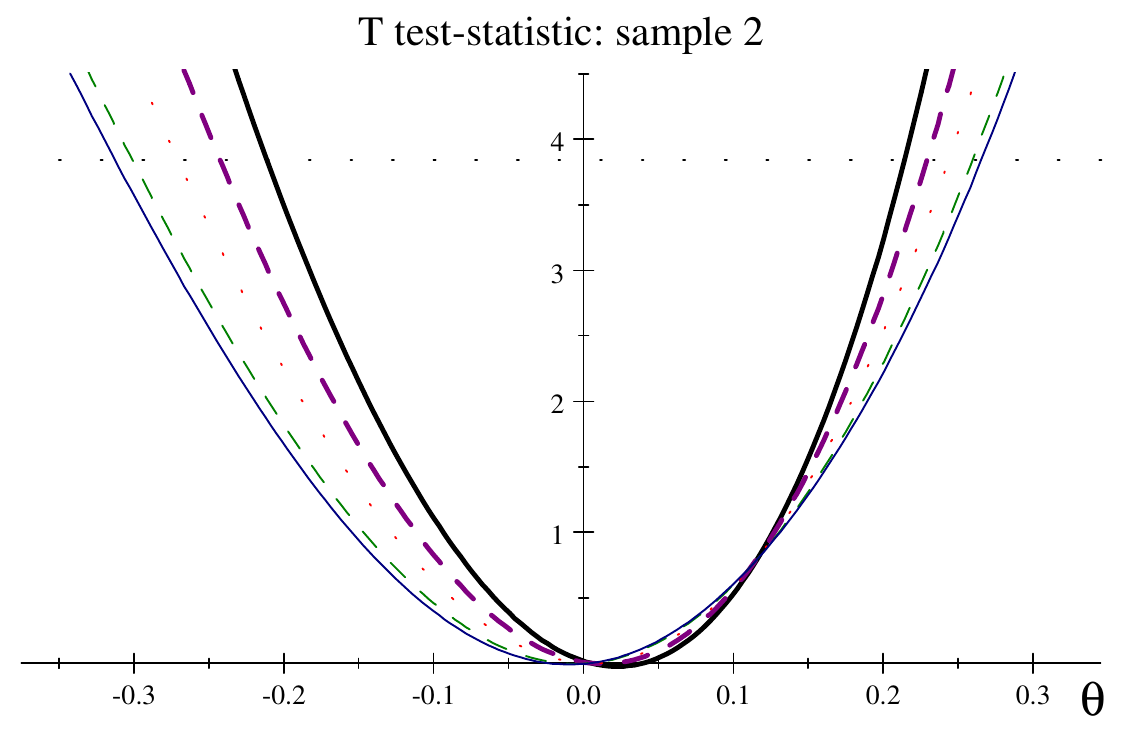}%
}
\\%
{\includegraphics[
natheight=2.965400in,
natwidth=4.483200in,
height=2.252in,
width=3.3909in
]%
{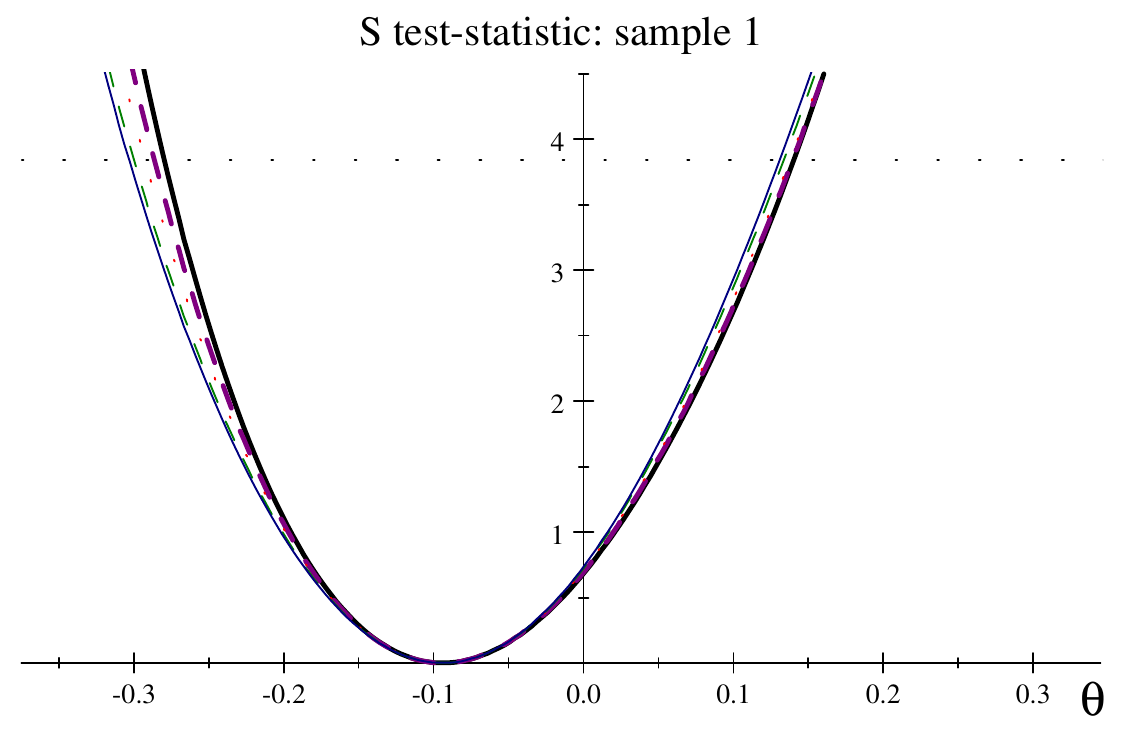}%
}
&
{\includegraphics[
natheight=2.965400in,
natwidth=4.483200in,
height=2.252in,
width=3.3909in
]%
{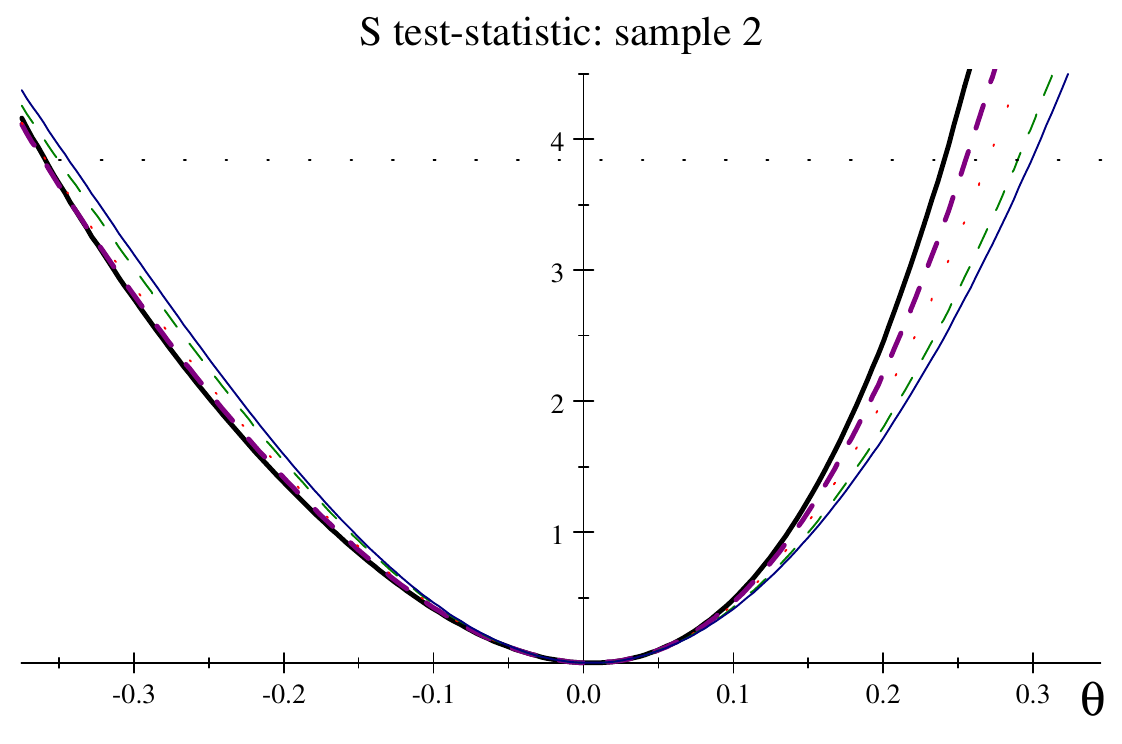}%
}
\end{tabular}
\ $%
\caption{Test statistics for samples 1 and 2 of size 30 from the standard normal population.\label{fig1}}%
\end{figure}%

\subsection{Robustness of procedures to contamination in the
data\label{Simulation2}}

We conducted a simulation study using the same design as in the preceding
subsection, but by considering $\frac{15}{100}\%$ of shifted observations in
each sample: $2$ shifted observations out of $30$ in the sample, $3$ shifted
observations out of $45$, $4$ shifted observations out of $60$. The underlying
distribution for shifted observations is taken to be $\mathcal{N}(\theta
+\sqrt{\theta^{2}+1},\theta^{2}+1)$, given that $\mathcal{N}(\theta,\theta
^{2}+1)$ is the true distribution. Therefore, the shifted observations follow
$\mathcal{N}(1,1)$ distribution when the true distribution is $\mathcal{N}%
(0,1)$, while they follow $\mathcal{N}(2.4142,2)$ distribution when the true
distribution is $\mathcal{N}(1,2)$. In Table \ref{t2b}, the results of
introducing $\frac{15}{100}\%$ of shifted observations is shown. Our interest
is focused on identifying the statistics that are less sensitive (robust
statistics with respect to shifted observations) as well as those that are
quite sensitive (non-robust statistics with respect to shifted observations).
As pointed out in Remark \ref{Remark}, all the test-statistics have the same
infinitesimal robustness, so the differences in robustness could be
atributable to the magnitude of the shift. When $\theta_{0}=0$, the best
statistics with shifted observations are exactly the same as ones without
shifted observations, but when the dispersion of the data is higher, i.e.
$\theta_{0}=1$, the closest coverage probability to the nominal level is not
for the empirical likelihood ratio test; we further observe that
$T_{n}^{\lambda}(\widehat{\theta}_{E,n},\theta_{0})$ and $S_{n}^{\lambda
}(\widehat{\theta}_{E,n},\theta_{0})$ not only have the narrowest CIs but also
the highest coverage probabilities. This means that the empirical modified
likelihood ratio test statistic, $S_{n}^{-1}(\widehat{\theta}_{E,n},\theta
_{0})$, is a robust statistic with respect to shifted observations in
comparison to the empirical likelihood ratio test, with respect to the
coverage probability and width of CIs. This also agrees with the conclusion
obtained from Table \ref{tCI}.%

\begin{table}[htbp] \scriptsize\tabcolsep0.8pt  \centering
$%
\begin{tabular}
[c]{llllccccccc}
&  &  &  & \multicolumn{7}{c}{$\mathcal{N}(\theta,\theta^{2}+1)$, $\theta
_{0}=0$}\\\hline
& \hspace*{0.35cm} &  & \hspace*{0.35cm} & \multicolumn{3}{c}{$T_{n}^{\lambda
}(\widehat{\theta}_{E,n},\theta_{0})$} & \hspace*{0.35cm} &
\multicolumn{3}{c}{$S_{n}^{\lambda}(\widehat{\theta}_{E,n},\theta_{0})$%
}\\\cline{4-7}\cline{9-11}
&  &  &  & Cov (\%) & \hspace*{0.35cm} & Avw &  & Cov (\%) & \hspace*{0.35cm}
& Avw\\\hline
&  & \multicolumn{1}{r}{$\lambda$} &  & \multicolumn{7}{c}{$n=30$}\\
$90\%$ &  & \multicolumn{1}{r}{$-1$} &  & 83.54 &  & \textbf{0.553} &  &
\textbf{84.47} &  & \textbf{0.542}\\
&  & \multicolumn{1}{r}{$-\frac{1}{2}$} &  & 84.39 &  & 0.568 &  & 84.15 &  &
0.552\\
&  & \multicolumn{1}{r}{$0$} &  & \textbf{84.51} &  & 0.580 &  & 83.84 &  &
0.560\\
&  & \multicolumn{1}{r}{$\frac{2}{3}$} &  & 84.17 &  & 0.590 &  & 83.15 &  &
0.569\\
&  & \multicolumn{1}{r}{$1$} &  & 83.80 &  & 0.592 &  & 82.61 &  & 0.572\\
$95\%$ &  & \multicolumn{1}{r}{$-1$} &  & 90.27 &  & \textbf{0.652} &  &
\textbf{90.70} &  & \textbf{0.639}\\
&  & \multicolumn{1}{r}{$-\frac{1}{2}$} &  & 91.00 &  & 0.673 &  & 90.65 &  &
0.655\\
&  & \multicolumn{1}{r}{$0$} &  & \textbf{91.40} &  & 0.691 &  & 90.05 &  &
0.670\\
&  & \multicolumn{1}{r}{$\frac{2}{3}$} &  & 90.90 &  & 0.707 &  & 89.07 &  &
0.684\\
&  & \multicolumn{1}{r}{$1$} &  & 90.33 &  & 0.711 &  & 88.65 &  & 0.689\\
&  & \multicolumn{1}{r}{$\lambda$} &  & \multicolumn{7}{c}{$n=45$}\\
$90\%$ &  & \multicolumn{1}{r}{$-1$} &  & 83.59 &  & \textbf{0.471} &  &
\textbf{83.08} &  & \textbf{0.457}\\
&  & \multicolumn{1}{r}{$-\frac{1}{2}$} &  & 84.10 &  & 0.479 &  &
\textbf{83.08} &  & 0.463\\
&  & \multicolumn{1}{r}{$0$} &  & \textbf{84.24} &  & 0.486 &  & 82.75 &  &
0.469\\
&  & \multicolumn{1}{r}{$\frac{2}{3}$} &  & 83.63 &  & 0.492 &  & 82.03 &  &
0.474\\
&  & \multicolumn{1}{r}{$1$} &  & 83.32 &  & 0.493 &  & 81.62 &  & 0.476\\
$95\%$ &  & \multicolumn{1}{r}{$-1$} &  & 90.30 &  & \textbf{0.557} &  &
\textbf{90.59} &  & \textbf{0.541}\\
&  & \multicolumn{1}{r}{$-\frac{1}{2}$} &  & 91.10 &  & 0.570 &  & 90.18 &  &
0.551\\
&  & \multicolumn{1}{r}{$0$} &  & \textbf{91.41} &  & 0.581 &  & 89.81 &  &
0.560\\
&  & \multicolumn{1}{r}{$\frac{2}{3}$} &  & 91.06 &  & 0.590 &  & 89.22 &  &
0.570\\
&  & \multicolumn{1}{r}{$1$} &  & 90.55 &  & 0.593 &  & 88.56 &  & 0.573\\
&  & \multicolumn{1}{r}{$\lambda$} &  & \multicolumn{7}{c}{$n=60$}\\
$90\%$ &  & \multicolumn{1}{r}{$-1$} &  & 83.15 &  & \textbf{0.418} &  &
\textbf{82.33} &  & \textbf{0.402}\\
&  & \multicolumn{1}{r}{$-\frac{1}{2}$} &  & 83.56 &  & 0.424 &  & 82.14 &  &
0.406\\
&  & \multicolumn{1}{r}{$0$} &  & \textbf{83.65} &  & 0.428 &  & 81.73 &  &
0.410\\
&  & \multicolumn{1}{r}{$\frac{2}{3}$} &  & 83.34 &  & 0.431 &  & 81.28 &  &
0.414\\
&  & \multicolumn{1}{r}{$1$} &  & 82.86 &  & 0.431 &  & 80.93 &  & 0.415\\
$95\%$ &  & \multicolumn{1}{r}{$-1$} &  & 90.10 &  & \textbf{0.495} &  &
\textbf{89.30} &  & \textbf{0.477}\\
&  & \multicolumn{1}{r}{$-\frac{1}{2}$} &  & \textbf{90.39} &  & 0.505 &  &
89.12 &  & 0.484\\
&  & \multicolumn{1}{r}{$0$} &  & 90.17 &  & 0.512 &  & 88.89 &  & 0.490\\
&  & \multicolumn{1}{r}{$\frac{2}{3}$} &  & 89.55 &  & 0.517 &  & 88.34 &  &
0.497\\
&  & \multicolumn{1}{r}{$1$} &  & 89.43 &  & 0.518 &  & 87.92 &  &
0.499\\\hline
\end{tabular}
\ \ \ \ \ \ \ \ \ \ \ \ \ \qquad%
\begin{tabular}
[c]{llllccccccc}
&  &  &  & \multicolumn{7}{c}{$\mathcal{N}(\theta,\theta^{2}+1)$, $\theta
_{0}=1$}\\\hline
& \hspace*{0.35cm} &  & \hspace*{0.35cm} & \multicolumn{3}{c}{$T_{n}^{\lambda
}(\widehat{\theta}_{E,n},\theta_{0})$} & \hspace*{0.35cm} &
\multicolumn{3}{c}{$S_{n}^{\lambda}(\widehat{\theta}_{E,n},\theta_{0})$%
}\\\cline{4-7}\cline{9-11}
&  &  &  & Cov (\%) & \hspace*{0.35cm} & Avw &  & Cov (\%) & \hspace*{0.35cm}
& Avw\\\hline
&  & \multicolumn{1}{r}{$\lambda$} &  & \multicolumn{7}{c}{$n=30$}\\
$90\%$ &  & \multicolumn{1}{r}{$-1$} &  & 85.09 &  & \textbf{0.564} &  &
\textbf{85.58} &  & \textbf{0.575}\\
&  & \multicolumn{1}{r}{$-\frac{1}{2}$} &  & \textbf{85.13} &  & 0.571 &  &
85.37 &  & 0.580\\
&  & \multicolumn{1}{r}{$0$} &  & 84.43 &  & 0.577 &  & 84.72 &  & 0.584\\
&  & \multicolumn{1}{r}{$\frac{2}{3}$} &  & 83.11 &  & 0.582 &  & 83.62 &  &
0.589\\
&  & \multicolumn{1}{r}{$1$} &  & 82.31 &  & 0.583 &  & 83.19 &  & 0.590\\
$95\%$ &  & \multicolumn{1}{r}{$-1$} &  & 91.13 &  & \textbf{0.669} &  &
\textbf{91.50} &  & \textbf{0.683}\\
&  & \multicolumn{1}{r}{$-\frac{1}{2}$} &  & \textbf{91.28} &  & 0.680 &  &
91.42 &  & 0.691\\
&  & \multicolumn{1}{r}{$0$} &  & 90.97 &  & 0.688 &  & 90.89 &  & 0.697\\
&  & \multicolumn{1}{r}{$\frac{2}{3}$} &  & 89.87 &  & 0.695 &  & 90.13 &  &
0.702\\
&  & \multicolumn{1}{r}{$1$} &  & 89.13 &  & 0.696 &  & 89.56 &  & 0.704\\
&  & \multicolumn{1}{r}{$\lambda$} &  & \multicolumn{7}{c}{$n=45$}\\
$90\%$ &  & \multicolumn{1}{r}{$-1$} &  & \textbf{86.25} &  & \textbf{0.473} &
& \textbf{86.50} &  & \textbf{0.479}\\
&  & \multicolumn{1}{r}{$-\frac{1}{2}$} &  & 85.96 &  & 0.479 &  & 85.90 &  &
0.483\\
&  & \multicolumn{1}{r}{$0$} &  & 85.14 &  & 0.484 &  & 85.57 &  & 0.488\\
&  & \multicolumn{1}{r}{$\frac{2}{3}$} &  & 84.05 &  & 0.490 &  & 84.30 &  &
0.493\\
&  & \multicolumn{1}{r}{$1$} &  & 83.15 &  & 0.493 &  & 83.56 &  & 0.496\\
$95\%$ &  & \multicolumn{1}{r}{$-1$} &  & \textbf{92.48} &  & \textbf{0.562} &
& \textbf{92.25} &  & \textbf{0.569}\\
&  & \multicolumn{1}{r}{$-\frac{1}{2}$} &  & 92.09 &  & 0.570 &  & 92.25 &  &
0.576\\
&  & \multicolumn{1}{r}{$0$} &  & 91.66 &  & 0.578 &  & 91.72 &  & 0.583\\
&  & \multicolumn{1}{r}{$\frac{2}{3}$} &  & 90.44 &  & 0.588 &  & 90.63 &  &
0.592\\
&  & \multicolumn{1}{r}{$1$} &  & 89.46 &  & 0.592 &  & 89.85 &  & 0.595\\
&  & \multicolumn{1}{r}{$\lambda$} &  & \multicolumn{7}{c}{$n=60$}\\
$90\%$ &  & \multicolumn{1}{r}{$-1$} &  & \textbf{85.22} &  & \textbf{0.416} &
& \textbf{85.24} &  & \textbf{0.419}\\
&  & \multicolumn{1}{r}{$-\frac{1}{2}$} &  & 84.61 &  & 0.420 &  & 84.67 &  &
0.423\\
&  & \multicolumn{1}{r}{$0$} &  & 83.95 &  & 0.425 &  & 83.91 &  & 0.427\\
&  & \multicolumn{1}{r}{$\frac{2}{3}$} &  & 82.49 &  & 0.431 &  & 82.66 &  &
0.432\\
&  & \multicolumn{1}{r}{$1$} &  & 81.70 &  & 0.434 &  & 82.03 &  & 0.435\\
$95\%$ &  & \multicolumn{1}{r}{$-1$} &  & \textbf{91.93} &  & \textbf{0.494} &
& \textbf{91.62} &  & \textbf{0.498}\\
&  & \multicolumn{1}{r}{$-\frac{1}{2}$} &  & 91.41 &  & 0.501 &  & 91.27 &  &
0.504\\
&  & \multicolumn{1}{r}{$0$} &  & 90.29 &  & 0.508 &  & 90.35 &  & 0.510\\
&  & \multicolumn{1}{r}{$\frac{2}{3}$} &  & 89.23 &  & 0.517 &  & 89.35 &  &
0.519\\
&  & \multicolumn{1}{r}{$1$} &  & 88.42 &  & 0.522 &  & 88.96 &  &
0.523\\\hline
\end{tabular}
\ $%
\caption{Coverage probability and average width of $90\%$ and $95\%$ confidence intervals with data contaminated by shifted observations.\label{t2b}}%
\end{table}%

\section{Testing for a composite null hypothesis\label{secComp}}

Now, let us consider the problem of testing the composite null hypothesis%
\begin{equation}
H_{0}:\boldsymbol{c}(\boldsymbol{\theta})\boldsymbol{=0}_{q}. \label{2.1.3}%
\end{equation}
The rest of this section, proceeds as follows. In Section \ref{secComp1}, we
introduce the empirical phi-divergence test statistic for the composite null
hypothesis in (\ref{2.1.3}). Section \ref{secComp2} is devoted to the
asymptotic results. In Section \ref{secComp3}, we present some results
regarding the power function of the family of empirical test statistics
proposed here. In Section \ref{secComp4}, a simulation study is carried out to
evaluate the performance of the proposed test procedure in comparison to some
other tests.

\subsection{Empirical phi-divergence test statistics\label{secComp1}}

In the following, we shall assume as in Qin and Lawless (1995), that the
unknown parameter vector $\boldsymbol{\theta}$ is defined through $r=p$
estimating functions, given in (\ref{estF}). With regard to (\ref{2.1.3}), we
shall assume that $\boldsymbol{c}:$ $\mathbb{R}^{p}\mathbb{\rightarrow R}^{q}$
is a vector-valued function such that the $q\times p$ matrix $\boldsymbol{C}%
$$\left(  \boldsymbol{\theta}\right)  =\frac{\partial\boldsymbol{c(\theta)}%
}{\partial\boldsymbol{\theta}^{T}}$ exists and is continuous at
$\boldsymbol{\theta}$ and that $\mathrm{rank}($$\boldsymbol{C}$$\left(
\boldsymbol{\theta}\right)  )=q$ ($q\leq p$). We denote by
$\widetilde{\boldsymbol{\theta}}_{E,n}$ the empirical restricted maximum
likelihood estimator of $\boldsymbol{\theta}$, obtained by minimizing
$\ell_{E,n}(\boldsymbol{\theta})=-\sum_{i=1}^{n}\log\left[  1+\boldsymbol{t}%
(\boldsymbol{\theta})^{T}\boldsymbol{g}(\boldsymbol{X}_{i},\boldsymbol{\theta
})\right]  $, subject to $\sum_{i=1}^{n}p_{i}\left(  \boldsymbol{\theta
}\right)  \boldsymbol{g}(\boldsymbol{X}_{i},\boldsymbol{\theta}%
)=\boldsymbol{0}_{r}$ and $\boldsymbol{c}$$\boldsymbol{(\theta)=0}_{q}$, i.e.,
(\ref{ec}) or%
\begin{equation}
\boldsymbol{\xi}_{1n}\left(  \boldsymbol{\theta},\boldsymbol{\boldsymbol{t}%
(\boldsymbol{\theta})},\boldsymbol{\nu}\right)  =\boldsymbol{0}_{p}\text{,}
\label{seq2}%
\end{equation}
where
\[
\boldsymbol{\xi}_{1n}\left(  \boldsymbol{\theta},\boldsymbol{\boldsymbol{t}%
(\boldsymbol{\theta})},\boldsymbol{\nu}\right)  =\boldsymbol{\xi}_{1n}\left(
\boldsymbol{\theta},\boldsymbol{\boldsymbol{t}(\boldsymbol{\theta})}\right)
=\frac{1}{n}\sum\limits_{i=1}^{n}\frac{1}{1+\boldsymbol{t}(\boldsymbol{\theta
})^{T}\boldsymbol{g}(\boldsymbol{X}_{i},\boldsymbol{\theta})}\boldsymbol{g}%
(\boldsymbol{X}_{i},\boldsymbol{\theta})
\]
provides $\boldsymbol{t}(\boldsymbol{\theta})$ in terms of $\boldsymbol{\theta
}$. Upon using the Lagrange multiplier method, once again, we have%
\[
\zeta_{E,n}(\boldsymbol{\theta},\boldsymbol{\boldsymbol{t}(\boldsymbol{\theta
})},\boldsymbol{\nu})=\frac{1}{n}\ell_{E,n}(\boldsymbol{\theta)}%
+\boldsymbol{\nu}^{T}\boldsymbol{c(\theta)},
\]
where $\boldsymbol{\nu}$ is a $q$-dimensional vector of Lagrange multipliers,
and differentiating $\zeta_{E,n}(\boldsymbol{\theta}%
,\boldsymbol{\boldsymbol{t}(\boldsymbol{\theta})},\boldsymbol{\nu})$ with
respect to $\boldsymbol{\theta}$ and $\boldsymbol{\nu}$, we obtain%
\begin{align*}
\boldsymbol{\xi}_{2n}\left(  \boldsymbol{\theta},\boldsymbol{\boldsymbol{t}%
(\boldsymbol{\theta})},\boldsymbol{\nu}\right)   &  =\frac{1}{n}%
\sum\limits_{i=1}^{n}\frac{1}{1+\boldsymbol{t}(\boldsymbol{\theta}%
)^{T}\boldsymbol{g}(\boldsymbol{X}_{i},\boldsymbol{\theta})}\boldsymbol{G}%
_{\boldsymbol{X}_{i}}(\boldsymbol{\theta})^{T}\boldsymbol{t}%
(\boldsymbol{\theta})+\mathbf{C}(\boldsymbol{\theta})\boldsymbol{^{T}\nu},\\
\boldsymbol{\xi}_{3n}\left(  \boldsymbol{\theta},\boldsymbol{\boldsymbol{t}%
(\boldsymbol{\theta})},\boldsymbol{\nu}\right)   &  =\boldsymbol{\xi}%
_{3n}\left(  \boldsymbol{\theta}\right)  =\boldsymbol{c(\theta)}.
\end{align*}
Therefore, $(\widetilde{\boldsymbol{\theta}}_{E,n}^{T}%
,\widetilde{\boldsymbol{\nu}}_{E,n}^{T})^{T}$ is obtained as the solution of
(\ref{seq2}) and
\begin{equation}
\boldsymbol{\xi}_{2n}\left(  \boldsymbol{\theta,\boldsymbol{t}%
(\boldsymbol{\theta}),\nu}\right)  =\boldsymbol{0}_{p}\text{,}%
\boldsymbol{\qquad\xi}_{3n}\left(  \boldsymbol{\theta,\boldsymbol{t}%
(\boldsymbol{\theta}),\nu}\right)  =\boldsymbol{0}_{q}\text{,} \label{seq1}%
\end{equation}
and conditions (\ref{cond}) and (\ref{empF2}) must be also satisfied.

The empirical likelihood ratio test for testing (\ref{2.1.3}) has the
expression
\begin{equation}
L_{E,n}(\widehat{\boldsymbol{\theta}}_{E,n},\widetilde{\boldsymbol{\theta}%
}_{E,n})=2\ell_{E,n}(\widehat{\boldsymbol{\theta}}_{E,n})-2\ell_{E,n}%
(\widetilde{\boldsymbol{\theta}}_{E,n}), \label{LRT2}%
\end{equation}
where $\widehat{\boldsymbol{\theta}}_{E,n}$ is the empirical maximum
likelihood estimator of the parameter $\boldsymbol{\theta}$, defined in
Section \ref{sec1}\ (case $r=p$), for which $\ell_{E,n}%
(\widehat{\boldsymbol{\theta}}_{E})=0$. Taking into account the results stated
in Section \ref{sec1} for $r=p$, we have $F_{n}%
=F_{n,\widehat{\boldsymbol{\theta}}_{E,n}}$, and so it is easy to show that
the expression in (\ref{LRT2}) can be written as
\[
L_{E,n}(\widehat{\boldsymbol{\theta}}_{E,n},\widetilde{\boldsymbol{\theta}%
}_{E,n})=2n\left(  D_{Kullback}(F_{n},F_{n,\widetilde{\boldsymbol{\theta}%
}_{E,n}})-D_{Kullback}(F_{n},F_{n,\widehat{\boldsymbol{\theta}}_{E,n}%
})\right)  =2nD_{Kullback}(F_{n,\widehat{\boldsymbol{\theta}}_{E,n}%
},F_{n,\widetilde{\boldsymbol{\theta}}_{E,n}}).
\]
Defining a family of empirical phi-divergence test statistics for testing the
hypothesis in (\ref{2.1.3}), we have%
\begin{align*}
T_{n}^{\phi}(\widehat{\boldsymbol{\theta}}_{E,n},\widetilde{\boldsymbol{\theta
}}_{E,n})  &  =\frac{2n}{\phi^{\prime\prime}(1)}\left(  D_{\phi}%
(F_{n},F_{n,\widetilde{\boldsymbol{\theta}}_{E,n}})-D_{\phi}(F_{n}%
,F_{n,\widehat{\boldsymbol{\theta}}_{E,n}})\right)  ,\\
S_{n}^{\phi}(\widehat{\boldsymbol{\theta}}_{E,n},\widetilde{\boldsymbol{\theta
}}_{E,n})  &  =\frac{2n}{\phi^{\prime\prime}(1)}D_{\phi}%
(F_{n,\widehat{\boldsymbol{\theta}}_{E,n}},F_{n,\widetilde{\boldsymbol{\theta
}}_{E,n}})
\end{align*}
where the $\phi$-divergence measures are as defined in (\ref{eqDiv}). Since
$F_{n}=F_{n,\widehat{\boldsymbol{\theta}}_{E,n}}$, both families of empirical
phi-divergence test statistics are equivalent and taking into account
(\ref{F2}) and $\boldsymbol{t}(\widehat{\boldsymbol{\theta}}_{E,n}%
)=\boldsymbol{0}_{p}$, we have%
\begin{align}
T_{n}^{\phi}(\widehat{\boldsymbol{\theta}}_{E,n},\widetilde{\boldsymbol{\theta
}}_{E,n})  &  =S_{n}^{\phi}(\widehat{\boldsymbol{\theta}}_{E,n}%
,\widetilde{\boldsymbol{\theta}}_{E,n})\nonumber\\
&  =\frac{2}{\phi^{\prime\prime}(1)}\sum\limits_{i=1}^{n}\frac{\phi\left(
1+\boldsymbol{t}(\widetilde{\boldsymbol{\theta}}_{E,n})^{T}\boldsymbol{g}%
(\boldsymbol{x}_{i},\widetilde{\boldsymbol{\theta}}_{E,n})\right)
}{1+\boldsymbol{t}(\widetilde{\boldsymbol{\theta}}_{E,n})^{T}\boldsymbol{g}%
(\boldsymbol{x}_{i},\widetilde{\boldsymbol{\theta}}_{E,n})}.
\label{phi-div-test}%
\end{align}

\subsection{Asymptotic null distributions\label{secComp2}}

Under the same conditions of Theorem \ref{Th1}\ and the conditions given in
Section \ref{secComp1} for $\boldsymbol{c}(\boldsymbol{\theta})$, it is
assumed that there exists a neighbourhood of $\boldsymbol{\theta}_{0}$ in
which $\left\Vert \frac{\partial\boldsymbol{C}(\boldsymbol{\theta})}%
{\partial\boldsymbol{\theta}}\right\Vert $ is bounded by some some constant,
we have
\begin{equation}
\sqrt{n}\left(
\begin{array}
[c]{c}%
\widetilde{\boldsymbol{\theta}}_{E,n}-\boldsymbol{\theta}_{0}\\
\widetilde{\boldsymbol{\nu}}_{E,n}%
\end{array}
\right)  \underset{n\rightarrow\infty}{\overset{\mathcal{L}}{\rightarrow}%
}\mathcal{N}\left(  \boldsymbol{0}_{p+q},\left(
\begin{array}
[c]{cc}%
\boldsymbol{P}\left(  \boldsymbol{\theta}_{0}\right)  & \boldsymbol{0}%
_{p\times q}\\
\boldsymbol{0}_{q\times p} & \boldsymbol{Q}\left(  \boldsymbol{\theta}%
_{0}\right)
\end{array}
\right)  \right)  , \label{2.5}%
\end{equation}
where $\boldsymbol{\theta}_{0}$ is the true value of the parameter
$\boldsymbol{\theta}$ and%
\begin{align}
\boldsymbol{P\left(  \boldsymbol{\theta}_{0}\right)  }  &
\boldsymbol{=V\left(  \boldsymbol{\theta}_{0}\right)  }-\boldsymbol{V\left(
\boldsymbol{\theta}_{0}\right)  C}\left(  \boldsymbol{\theta}_{0}\right)
^{T}\boldsymbol{Q\left(  \boldsymbol{\theta}_{0}\right)  C\left(
\boldsymbol{\theta}_{0}\right)  V\left(  \boldsymbol{\theta}_{0}\right)
},\label{P}\\
\boldsymbol{Q}\left(  \boldsymbol{\theta}_{0}\right)   &
=(\boldsymbol{C\left(  \boldsymbol{\theta}_{0}\right)  V\left(
\boldsymbol{\theta}_{0}\right)  C}\left(  \boldsymbol{\theta}_{0}\right)
^{T})^{-1}. \label{BM}%
\end{align}
This result is derived by taking into account, the facts that%
\begin{align}
\sqrt{n}(\widetilde{\boldsymbol{\theta}}_{E,n}-\boldsymbol{\theta}_{0})  &
=-\boldsymbol{S}_{12}^{-1}\left(  \boldsymbol{\theta}_{0}\right)  \sqrt
{n}\boldsymbol{\bar{g}}_{n}(\boldsymbol{\theta}_{0})-\boldsymbol{V}\left(
\boldsymbol{\theta}_{0}\right)  \boldsymbol{C}\left(  \boldsymbol{\theta}%
_{0}\right)  ^{T}\sqrt{n}\widetilde{\boldsymbol{\nu}}_{E,n}+o_{p}%
(1),\label{o1}\\
\sqrt{n}\widetilde{\boldsymbol{\nu}}_{E,n}  &  =-\boldsymbol{Q\left(
\boldsymbol{\theta}_{0}\right)  C\left(  \boldsymbol{\theta}_{0}\right)
S}_{12}^{-1}\left(  \boldsymbol{\theta}_{0}\right)  \sqrt{n}\boldsymbol{\bar
{g}}_{n}(\boldsymbol{\theta}_{0})+o_{p}(1). \label{o2}%
\end{align}
For a complete proof, one may refer to Qin and Lawless (1995).

\begin{theorem}
\label{Th2}Under $H_{0}$ in (\ref{2.1.3}) and the assumptions above, we have%
\begin{equation}
T_{n}^{\phi}(\widehat{\boldsymbol{\theta}}_{E,n},\widetilde{\boldsymbol{\theta
}}_{E,n})\overset{\mathcal{L}}{\underset{n\rightarrow\infty}{\longrightarrow}%
}\chi_{q}^{2}. \label{th2R}%
\end{equation}

\begin{proof}
Let us consider the function%
\[
\ell_{\phi}\left(  \boldsymbol{t}\right)  =\frac{1}{n}\sum\limits_{i=1}%
^{n}\frac{\phi(1+\boldsymbol{t}^{T}\boldsymbol{g}(\boldsymbol{x}%
_{i},\widetilde{\boldsymbol{\theta}}_{E,n}))}{1+\boldsymbol{t}^{T}%
\boldsymbol{g}(\boldsymbol{x}_{i},\widetilde{\boldsymbol{\theta}}_{E,n})}.
\]
A second-order Taylor expansion of $\ell_{\phi}(\widetilde{\boldsymbol{t}})$
around $\boldsymbol{0}_{p}$, with $\widetilde{\boldsymbol{t}}=\boldsymbol{t}%
(\widetilde{\boldsymbol{\theta}}_{E,n})$, gives%
\[
T_{n}^{\phi}(\widehat{\boldsymbol{\theta}}_{E,n},\widetilde{\boldsymbol{\theta
}}_{E,n})=\frac{2n\ell_{\phi}(\widetilde{\boldsymbol{t}})}{\phi^{\prime\prime
}\left(  1\right)  }=n\boldsymbol{t}(\widetilde{\boldsymbol{\theta}}%
_{E,n})^{T}\boldsymbol{S}_{11}(\widetilde{\boldsymbol{\theta}}_{E,n}%
)\boldsymbol{t}(\widetilde{\boldsymbol{\theta}}_{E,n})+o\left(  ||\sqrt
{n}\boldsymbol{t}(\widetilde{\boldsymbol{\theta}}_{E,n})||^{2}\right)  .
\]
But,
\[
\sqrt{n}\boldsymbol{t}(\widetilde{\boldsymbol{\theta}}_{E,n})=\sqrt
{n}\boldsymbol{S}_{11}^{-1}\left(  \boldsymbol{\theta}_{0}\right)
\boldsymbol{\bar{g}}_{n}(\widetilde{\boldsymbol{\theta}}_{E,n})+o_{p}%
(\boldsymbol{1}_{r}),
\]
and
\begin{align*}
\boldsymbol{\bar{g}}_{n}(\widetilde{\boldsymbol{\theta}}_{E,n})  &
=\boldsymbol{\bar{g}}_{n}(\boldsymbol{\theta}_{0})+\frac{\partial}%
{\partial\boldsymbol{\theta}^{T}}\left.  \boldsymbol{\bar{g}}_{n}\left(
\boldsymbol{\theta}\right)  \right\vert _{\boldsymbol{\theta}%
=\boldsymbol{\theta}_{0}}(\widetilde{\boldsymbol{\theta}}_{E,n}%
-\boldsymbol{\theta}_{0})+o_{p}(n^{-1/2}\boldsymbol{1}_{r})\\
&  =\boldsymbol{\bar{g}}_{n}(\boldsymbol{\theta}_{0})+\boldsymbol{S}%
_{12}\left(  \boldsymbol{\theta}_{0}\right)  (\widetilde{\boldsymbol{\theta}%
}_{E,n}-\boldsymbol{\theta}_{0})+o_{p}(n^{-1/2}\boldsymbol{1}_{r})\\
&  =\boldsymbol{S}_{12}\left(  \boldsymbol{\theta}_{0}\right)  \boldsymbol{V}%
\left(  \boldsymbol{\theta}_{0}\right)  \boldsymbol{C}\left(
\boldsymbol{\theta}_{0}\right)  ^{T}\widetilde{\boldsymbol{\nu}}_{E,n}%
+o_{p}(n^{-1/2}\boldsymbol{1}_{r}).
\end{align*}
By the strong law of large numbers, $\frac{\partial}{\partial
\boldsymbol{\theta}^{T}}\left.  \boldsymbol{\bar{g}}_{n}\left(
\boldsymbol{\theta}\right)  \right\vert _{\boldsymbol{\theta}%
=\boldsymbol{\theta}_{0}}\overset{a.s.}{\underset{n\rightarrow\infty
}{\longrightarrow}}\boldsymbol{S}_{12}\left(  \boldsymbol{\theta}_{0}\right)
$, and so (\ref{o1}) can be rewritten as%
\[
\boldsymbol{S}_{12}\left(  \boldsymbol{\theta}_{0}\right)
(\widetilde{\boldsymbol{\theta}}_{E,n}-\boldsymbol{\theta}_{0}%
)=-\boldsymbol{\bar{g}}_{n}(\boldsymbol{\theta}_{0})-\boldsymbol{S}%
_{12}\left(  \boldsymbol{\theta}_{0}\right)  \boldsymbol{V}\left(
\boldsymbol{\theta}_{0}\right)  \boldsymbol{C}\left(  \boldsymbol{\theta}%
_{0}\right)  ^{T}\widetilde{\boldsymbol{\nu}}_{E,n}+o_{p}(n^{-1/2}%
\boldsymbol{1}_{r}).
\]
Consequently,%
\begin{equation}
T_{n}^{\phi}(\widehat{\boldsymbol{\theta}}_{E,n},\widetilde{\boldsymbol{\theta
}}_{E,n})=\left(  \sqrt{n}\widetilde{\boldsymbol{\nu}}_{E,n}\right)
^{T}\boldsymbol{Q}^{-1}\left(  \boldsymbol{\theta}_{0}\right)  \sqrt
{n}\widetilde{\boldsymbol{\nu}}_{E,n}+o_{p}(1), \label{est2}%
\end{equation}
with $\sqrt{n}\widetilde{\boldsymbol{\nu}}_{E,n}\underset{n\rightarrow
\infty}{\overset{\mathcal{L}}{\rightarrow}}\mathcal{N}\left(  \boldsymbol{0}%
_{q},\boldsymbol{Q}\left(  \boldsymbol{\theta}_{0}\right)  \right)  $, which
means that (\ref{th2R}) holds.
\end{proof}
\end{theorem}

\subsection{Approximations of the power function\label{secComp3}}

Assume that $\boldsymbol{\theta}^{\ast}\notin\Theta_{0}$ is the true value of
the unknown parameter so that $\widehat{\boldsymbol{\theta}}_{E,n}%
\underset{n\rightarrow\infty}{\overset{a.s.}{\rightarrow}}\boldsymbol{\theta
}^{\ast}$, and that there exists a $\boldsymbol{\theta}_{0}\in\Theta_{0}$ such
that the restricted empirical maximum likelihood estimator satisfies
$\widetilde{\boldsymbol{\theta}}_{E,n}\underset{n\rightarrow\infty
}{\overset{a.s.}{\rightarrow}}\boldsymbol{\theta}_{0}$. Then, we have
\begin{equation}
\sqrt{n}\left(  \left(
\begin{array}
[c]{c}%
\widehat{\boldsymbol{\theta}}_{E,n}\\
\widetilde{\boldsymbol{\theta}}_{E,n}%
\end{array}
\right)  -\left(
\begin{array}
[c]{c}%
\boldsymbol{\theta}^{\ast}\\
\boldsymbol{\theta}_{0}%
\end{array}
\right)  \right)  \underset{n\rightarrow\infty}{\overset{\mathcal{L}%
}{\rightarrow}}\mathcal{N}\left(  \left(
\begin{array}
[c]{c}%
\boldsymbol{0}_{p}\\
\boldsymbol{0}_{p}%
\end{array}
\right)  ,\left(
\begin{array}
[c]{cc}%
\boldsymbol{V}\left(  \boldsymbol{\theta}_{0}\right)  & \boldsymbol{A}%
_{12}\left(  \boldsymbol{\theta}^{\ast},\boldsymbol{\theta}_{0}\right) \\
\boldsymbol{A}_{12}\left(  \boldsymbol{\theta}^{\ast},\boldsymbol{\theta}%
_{0}\right)  ^{T} & \boldsymbol{P(\theta}_{0})
\end{array}
\right)  \right)  , \label{A2}%
\end{equation}
where $\boldsymbol{A}_{12}\left(  \boldsymbol{\theta}^{\ast}%
,\boldsymbol{\theta}_{0}\right)  $ is\ a $q\times q$ matrix.

\begin{theorem}
\label{ThTh}Under the conditions stated above, we have
\[
\sqrt{n}\left(  D_{\phi}(\widehat{\boldsymbol{\theta}}_{E,n}%
,\widetilde{\boldsymbol{\theta}}_{E,n})-D_{\phi}(\boldsymbol{\theta}^{\ast
},\boldsymbol{\theta}_{0})\right)  \underset{n\rightarrow\infty
}{\overset{\mathcal{L}}{\rightarrow}}\mathcal{N}\left(  0,\varsigma_{\phi}%
^{2}(\boldsymbol{\theta}^{\ast},\boldsymbol{\theta}_{0})\right)  ,
\]
where
\[
\varsigma_{\phi}^{2}\left(  \boldsymbol{\theta}^{\ast},\boldsymbol{\theta}%
_{0}\right)  =\boldsymbol{\tau}_{\phi}^{T}\boldsymbol{\left(
\boldsymbol{\theta}^{\ast},\boldsymbol{\theta}_{0}\right)  V}\left(
\boldsymbol{\theta}_{0}\right)  \boldsymbol{\tau}_{\phi}\boldsymbol{\left(
\boldsymbol{\theta}^{\ast},\boldsymbol{\theta}_{0}\right)  +}%
2\boldsymbol{\boldsymbol{\tau}_{\phi}\boldsymbol{\left(  \boldsymbol{\theta
}^{\ast},\boldsymbol{\theta}_{0}\right)  }}^{T}\boldsymbol{A}_{12}\left(
\boldsymbol{\theta}^{\ast},\boldsymbol{\theta}_{0}\right)  \boldsymbol{\iota
}_{\phi}\boldsymbol{\left(  \boldsymbol{\theta}^{\ast},\boldsymbol{\theta}%
_{0}\right)  }+\boldsymbol{\iota}_{\phi}\boldsymbol{\left(  \boldsymbol{\theta
}^{\ast},\boldsymbol{\theta}_{0}\right)  }^{T}\boldsymbol{P(\theta}%
_{0})\boldsymbol{\iota}_{\phi}\boldsymbol{\left(  \boldsymbol{\theta}^{\ast
},\boldsymbol{\theta}_{0}\right)  ,}%
\]
with $\boldsymbol{\tau}_{\phi}\boldsymbol{\left(  \boldsymbol{\theta}^{\ast
},\boldsymbol{\theta}_{0}\right)  }$ given by (\ref{tao}) and%
\[
\boldsymbol{\iota}_{\phi}\boldsymbol{\left(  \boldsymbol{\theta}^{\ast
},\boldsymbol{\theta}_{0}\right)  =}(\iota_{1}^{\phi}\left(
\boldsymbol{\theta}^{\ast},\boldsymbol{\boldsymbol{\theta}_{0}}\right)
,...,\iota_{p}^{\phi}\left(  \boldsymbol{\theta}^{\ast}%
,\boldsymbol{\boldsymbol{\theta}_{0}}\right)  )^{T}=\left.  \frac{\partial
D_{\phi}\left(  \boldsymbol{\theta}^{\ast},\boldsymbol{\vartheta}\right)
}{\partial\boldsymbol{\vartheta}}\right\vert _{\boldsymbol{\vartheta
}=\boldsymbol{\theta}_{0}}.
\]

\end{theorem}

\begin{remark}
Based on Theorem \ref{ThTh}, we get an approximation of the power function at
$\boldsymbol{\theta}^{\ast}$ as
\begin{align*}
\beta_{n,\phi}^{1}(\boldsymbol{\theta}^{\ast},\boldsymbol{\boldsymbol{\theta
}_{0}})  &  =P_{\boldsymbol{\theta}^{\ast}}(T_{n}^{\phi}%
(\widehat{\boldsymbol{\theta}}_{E,n},\widetilde{\boldsymbol{\theta}}%
_{E,n})>\chi_{q,\alpha}^{2})\\
&  =1-\Phi\left(  \frac{\sqrt{n}}{\varsigma_{\phi}\left(  \boldsymbol{\theta
}^{\ast},\boldsymbol{\theta}_{0}\right)  }\left(  \frac{\phi^{\prime\prime
}\left(  1\right)  }{2n}\chi_{q,\alpha}^{2}-D_{\phi}\left(
F_{n,\boldsymbol{\theta}^{\ast}},F_{n,\boldsymbol{\theta}_{0}}\right)
\right)  \right)  .
\end{align*}
If some alternative $\boldsymbol{\theta}^{\ast}\neq\boldsymbol{\theta}_{0}$ is
the true parameter, then the probability of rejecting $\boldsymbol{\theta}%
_{0}$ with the rejection rule $T_{n}^{\phi}(\widehat{\boldsymbol{\theta}%
}_{E,n},\widetilde{\boldsymbol{\theta}}_{E,n})>\chi_{q,\alpha}^{2}$, for a
fixed significance level $\alpha$, tends to one as $n\rightarrow\infty.$ Thus,
the test is consistent in the sense of Fraser (1957).
\end{remark}

We may also find an approximation to the power of $T_{n}^{\phi}%
(\widehat{\boldsymbol{\theta}}_{E,n},\widetilde{\boldsymbol{\theta}}_{E,n})$
at an alternative hypothesis close to the null hypothesis. Let
$\boldsymbol{\theta}_{n}\in\Theta-\Theta_{0}$ be a given alternative, and let
$\boldsymbol{\theta}_{0}$ be the element in $\Theta_{0}$ closest to
$\boldsymbol{\theta}_{n}$ in terms of Euclidean distance. In order to
introduce contiguous alternative hypotheses, we may consider a fixed
$\boldsymbol{f}$ $\in\mathbb{R}^{p}$ and allow $\boldsymbol{\theta}_{n}$ to
tend to $\boldsymbol{\theta}_{0}$ as $n$ increases in the following manner:
\begin{equation}
H_{1,n}:\boldsymbol{\theta}_{n}=\boldsymbol{\theta}_{0}+n^{-1/2}%
\boldsymbol{f.} \label{C1}%
\end{equation}

\begin{theorem}
\label{Th2B}Under the assumptions of Theorem \ref{Th2} and $H_{1,n}$ given in
(\ref{C1}), we have%
\begin{equation}
T_{n}^{\phi}(\widehat{\boldsymbol{\theta}}_{E,n},\widetilde{\boldsymbol{\theta
}}_{E,n})\overset{\mathcal{L}}{\underset{n\rightarrow\infty}{\longrightarrow}%
}\chi_{q}^{2}(\varrho(\boldsymbol{\theta}_{0})), \label{NC2}%
\end{equation}
where $\varrho(\boldsymbol{\theta}_{0})=\boldsymbol{f}^{T}\boldsymbol{C\left(
\boldsymbol{\theta}_{0}\right)  }^{T}\boldsymbol{Q\left(  \boldsymbol{\theta
}_{0}\right)  C\left(  \boldsymbol{\theta}_{0}\right)  f}$.
\end{theorem}

\begin{proof}
If we substitute (\ref{o2}) into (\ref{est2}), we obtain%
\[
T_{n}^{\phi}(\widehat{\boldsymbol{\theta}}_{E,n},\widetilde{\boldsymbol{\theta
}}_{E,n})=\sqrt{n}\boldsymbol{\bar{g}}_{n}^{T}\left(  \boldsymbol{\theta}%
_{0}\right)  \boldsymbol{S}_{12}^{-1}\left(  \boldsymbol{\theta}_{0}\right)
\boldsymbol{C}(\boldsymbol{\boldsymbol{\theta}_{0}})\boldsymbol{^{T}%
Q}(\boldsymbol{\boldsymbol{\theta}_{0}})\boldsymbol{C}%
(\boldsymbol{\boldsymbol{\theta}_{0}})\boldsymbol{S}_{12}^{-1}\left(
\boldsymbol{\theta}_{0}\right)  \sqrt{n}\boldsymbol{\bar{g}}_{n}%
(\boldsymbol{\theta}_{0})+o_{p}(1).
\]
Since $\boldsymbol{\bar{g}}_{n}(\widehat{\boldsymbol{\theta}}_{E,n}%
)=\boldsymbol{0}_{p}$, the Taylor expansion $\boldsymbol{\bar{g}}%
_{n}(\widehat{\boldsymbol{\theta}}_{E,n})=\boldsymbol{\bar{g}}_{n}%
(\boldsymbol{\theta}_{0})+\frac{\partial}{\partial\boldsymbol{\theta}^{T}%
}\left.  \boldsymbol{\bar{g}}_{n}\left(  \boldsymbol{\theta}\right)
\right\vert _{\boldsymbol{\theta}=\boldsymbol{\theta}_{0}}%
(\widehat{\boldsymbol{\theta}}_{E,n}-\boldsymbol{\theta}_{0})+o_{p}%
(n^{-1/2}\boldsymbol{1}_{r})$ yields
\[
-\boldsymbol{S}_{12}^{-1}\left(  \boldsymbol{\theta}_{0}\right)  \sqrt
{n}\boldsymbol{\bar{g}}_{n}(\boldsymbol{\theta}_{0}%
)=\widehat{\boldsymbol{\theta}}_{E,n}-\boldsymbol{\theta}_{0}+o_{p}%
(\boldsymbol{1}_{r}).
\]
By following the same steps as in the proof of Theorem \ref{Th1B}, we have
that under $H_{1,n}$, $\sqrt{n}(\widehat{\boldsymbol{\theta}}_{E,n}%
-\boldsymbol{\theta}_{0})\overset{\mathcal{L}}{\underset{n\rightarrow
\infty}{\longrightarrow}}\mathcal{N}(\boldsymbol{f},\boldsymbol{V}%
(\boldsymbol{\theta}_{0}))$. Hence, we obtain
\begin{align*}
T_{n}^{\phi}(\widehat{\boldsymbol{\theta}}_{E,n},\widetilde{\boldsymbol{\theta
}}_{E,n})  &  =(\widehat{\boldsymbol{\theta}}_{E,n}-\boldsymbol{\theta}%
_{0})^{T}\boldsymbol{C}(\boldsymbol{\boldsymbol{\theta}_{0}})\boldsymbol{^{T}%
Q}(\boldsymbol{\boldsymbol{\theta}_{0}})\boldsymbol{C}%
(\boldsymbol{\boldsymbol{\theta}_{0}})(\widehat{\boldsymbol{\theta}}%
_{E,n}-\boldsymbol{\theta}_{0})+o_{p}(1)\\
&  =\boldsymbol{Z}^{T}\boldsymbol{Z}+o_{p}(1),
\end{align*}
where $\boldsymbol{Z}=\boldsymbol{Q\left(  \boldsymbol{\theta}_{0}\right)
}^{1/2}\boldsymbol{C\left(  \boldsymbol{\theta}_{0}\right)  }\sqrt
{n}(\widehat{\boldsymbol{\theta}}_{E,n}-\boldsymbol{\theta}_{0}%
)\underset{n\rightarrow\infty}{\overset{\mathcal{L}}{\rightarrow}}%
\mathcal{N}(\boldsymbol{Q\left(  \boldsymbol{\theta}_{0}\right)  }%
^{1/2}\boldsymbol{C\left(  \boldsymbol{\theta}_{0}\right)  f},\boldsymbol{I}%
_{q})$. We thus obtain%
\[
T_{n}^{\phi}(\widehat{\boldsymbol{\theta}}_{E,n},\widetilde{\boldsymbol{\theta
}}_{E,n})\overset{\mathcal{L}}{\underset{n\rightarrow\infty}{\longrightarrow}%
}\chi_{q}^{2}(\boldsymbol{f}^{T}\boldsymbol{C\left(  \boldsymbol{\theta}%
_{0}\right)  }^{T}\boldsymbol{Q\left(  \boldsymbol{\theta}_{0}\right)
C\left(  \boldsymbol{\theta}_{0}\right)  f}),
\]
which completes the proof.
\end{proof}

A second way to consider contiguous alternative hypotheses is to relax the
condition $\boldsymbol{c\left(  \boldsymbol{\theta}\right)  =0}_{q}$ defining
the null hypothesis $\Theta_{0}$. Let $\boldsymbol{f\in}\mathbb{R}^{d}$ be
such that $\boldsymbol{f\neq0}_{p}$. Consider the following sequence
$\boldsymbol{\theta}_{n}$ of parameters approaching $\Theta_{0}$:%
\begin{equation}
H_{1,n}^{\ast}:\boldsymbol{c}(\boldsymbol{\theta}_{n})=n^{-1/2}%
\boldsymbol{\bar{f}}. \label{C2}%
\end{equation}
A Taylor expansion of $\boldsymbol{c}$$\left(  \boldsymbol{\theta}_{n}\right)
$ around $\boldsymbol{\theta}_{0}$ yields
\[
\boldsymbol{c}\left(  \boldsymbol{\theta}_{n}\right)  =\boldsymbol{c}\left(
\boldsymbol{\theta}_{0}\right)  +\boldsymbol{C}(\boldsymbol{\boldsymbol{\theta
}_{0}})\left(  \boldsymbol{\theta}_{n}-\boldsymbol{\theta}_{0}\right)
+o\left(  \left\Vert \boldsymbol{\theta}_{n}-\boldsymbol{\theta}%
_{0}\right\Vert \right)  .
\]
Upon substituting $\boldsymbol{\theta}_{n}=\boldsymbol{\theta}_{0}%
+n^{-1/2}\boldsymbol{f}$ in the previous formula and taking into account that
$\boldsymbol{c}$$\left(  \boldsymbol{\theta}_{0}\right)  =\boldsymbol{0}_{r}$,
we have%
\[
\boldsymbol{c}\left(  \boldsymbol{\theta}_{n}\right)  =n^{-1/2}\boldsymbol{C}%
(\boldsymbol{\boldsymbol{\theta}_{0}})\boldsymbol{f.}%
\]
Then, the equivalence between $H_{1,n}^{\ast}$ and $H_{1,n}$ is obtained for
$\boldsymbol{\bar{f}=C\left(  \boldsymbol{\theta}_{0}\right)  f}$. We thus
have the following result.

\begin{theorem}
Under the contiguous alternative hypothesis in (\ref{C2}), we have
\[
T_{n}^{\phi}(\widetilde{\boldsymbol{\theta}}_{E,n})\underset{n\rightarrow
\infty}{\overset{L}{\rightarrow}}\chi_{q}^{2}(\varrho(\boldsymbol{\theta}%
_{0})),
\]
with $\varrho(\boldsymbol{\theta}_{0})=\boldsymbol{\bar{f}}^{T}%
\boldsymbol{Q\left(  \boldsymbol{\theta}_{0}\right)  \bar{f}}$ and
$\boldsymbol{Q\left(  \boldsymbol{\theta}_{0}\right)  }$ as defined in
(\ref{BM}).
\end{theorem}

\subsection{Simulation results\label{secComp4}}

From a sample $X_{1},...,X_{n}$ of i.i.d. random variables with $E[X_{i}]=\mu$
and $Var[X_{i}]=\sigma^{2},$ we wish to test that the coefficient of variation
is $1$, i.e., $H_{0}:\mu=\sigma$\ against$\ H_{0}:\mu\neq\sigma$. In order to
make a decision, we need to obtain the maximum likelihood estimator under the
restriction $c(\boldsymbol{\vartheta})=\sigma-\mu=0$, with
$\boldsymbol{\vartheta}=(\mu,\sigma)^{T}$. The estimating equations in this
case are $g_{1}(x_{i},\boldsymbol{\vartheta})=x_{i}-\mu$, $g_{2}%
(x_{i},\boldsymbol{\vartheta})=x_{i}^{2}-(\sigma^{2}+\mu^{2})$. If we
establish a bijective transformation between $\boldsymbol{\vartheta}%
=(\mu,\sigma)^{T}$\ and $\boldsymbol{\theta}=(u,v)^{T}$, and are able to
obtain the empirical restricted maximum likelihood estimator of
$\boldsymbol{\theta}$, $\widetilde{\boldsymbol{\theta}}_{E,n}$, due to the
invariance property, we can obtain the empirical restricted maximum likelihood
estimator of $\boldsymbol{\vartheta}$, $\widetilde{\boldsymbol{\vartheta}}%
_{R}$, by taking the inverse of the transformation. Let $\mu=u$ and\ $\sigma
^{2}=v-u^{2}$. Then%
\[
H_{0}:c(\boldsymbol{\theta})=0\quad\text{vs.}\quad H_{0}:c(\boldsymbol{\theta
})\neq0
\]
with $c(\boldsymbol{\theta})=c(u,v)=v-2u^{2}=0$. The estimating equations
under the new parameterization are
\begin{align*}
g_{1}(X_{i},\boldsymbol{\theta})  &  =g_{1}(X_{i},u,v)=X_{i}-u,\\
g_{2}(X_{i},\boldsymbol{\theta})  &  =g_{2}(X_{i},u,v)=X_{i}^{2}-v.
\end{align*}
In general, we have the system of $2p+r=5$ equations as follows:%
\begin{align*}
\xi_{1,j,n}(u,v,t_{1},t_{2},\gamma)  &  =\frac{1}{n}%
{\textstyle\sum\limits_{i=1}^{n}}
\frac{g_{j}(X_{i},u,v)}{1+t_{1}g_{1}(X_{i},u,v)+t_{2}g_{2}(X_{i},u,v)}=0,\quad
j=1,2,\\
\xi_{2,1,n}(u,v,t_{1},t_{2},\gamma)  &  =\frac{1}{n}%
{\textstyle\sum\limits_{i=1}^{n}}
\frac{t_{1}\frac{\partial}{\partial u}g_{1}(X_{i},u,v)+t_{2}\frac{\partial
}{\partial u}g_{2}(X_{i},u,v)}{1+t_{1}g_{1}(X_{i},u,v)+t_{2}g_{2}(X_{i}%
,u,v)}+\frac{\partial}{\partial u}c(u,v)\gamma=0,\\
\xi_{3,1,n}(u,v,t_{1},t_{2},\gamma)  &  =\frac{1}{n}%
{\textstyle\sum\limits_{i=1}^{n}}
\frac{t_{1}\frac{\partial}{\partial v}g_{1}(X_{i},u,v)+t_{2}\frac{\partial
}{\partial v}g_{2}(X_{i},u,v)}{1+t_{1}g_{1}(X_{i},u,v)+t_{2}g_{2}(X_{i}%
,u,v)}+\frac{\partial}{\partial v}c(u,v)\gamma=0,\\
c(u,v)  &  =v-2u^{2}=0,
\end{align*}
which are equivalent to
\begin{align*}
\frac{1}{n}%
{\textstyle\sum\limits_{i=1}^{n}}
\frac{X_{i}-u}{1+t_{1}(X_{i}-u)+t_{2}(X_{i}^{2}-v)}  &  =0,\\
\frac{1}{n}%
{\textstyle\sum\limits_{i=1}^{n}}
\frac{X_{i}^{2}-v}{1+t_{1}(X_{i}-u)+t_{2}(X_{i}^{2}-v)}  &  =0,\\
-\frac{1}{n}%
{\textstyle\sum\limits_{i=1}^{n}}
\frac{t_{1}}{1+t_{1}g_{1}(X_{i},u,v)+t_{2}g_{2}(X_{i},u,v)}-4u\gamma &  =0,\\
-\frac{1}{n}%
{\textstyle\sum\limits_{i=1}^{n}}
\frac{t_{2}}{1+t_{1}g_{1}(X_{i},u,v)+t_{2}g_{2}(X_{i},u,v)}+\gamma &  =0,\\
v-2u^{2}  &  =0,
\end{align*}
for obtaining $\widetilde{\boldsymbol{\theta}}_{E,n}=(\widetilde{u}%
,\widetilde{v})^{T}$. Observe that%
\[
\frac{1}{n}\sum_{i=1}^{n}\frac{1}{1+t_{1}g_{1}(X_{i},u,v)+t_{2}g_{2}%
(X_{i},u,v)}=1,
\]
since the sum of the probabilities is 1, and consequently the third and fourth
equations become $t_{1}=-4u\gamma$, $t_{2}=\gamma$, and we also know that
$v=2u^{2}$. So, the optimization problem is reduced to simply to%
\begin{align*}
f_{1}(u,\gamma)  &  =%
{\textstyle\sum\limits_{i=1}^{n}}
\frac{X_{i}-u}{1+\gamma(X_{i}^{2}-4X_{i}u+2u^{2})}=0,\\
f_{2}(u,\gamma)  &  =%
{\textstyle\sum\limits_{i=1}^{n}}
\frac{X_{i}^{2}-2u^{2}}{1+\gamma(X_{i}^{2}-4X_{i}u+2u^{2})}=0.
\end{align*}
The solution $\widetilde{\boldsymbol{\theta}}_{E,n}$\ must satisfy that the
$n$ probabilities are not less than zero, i.e.,
\begin{align*}
\frac{1}{n(1+t_{1}g_{1}(X_{i},u,v)+t_{2}g_{2}(X_{i},u,v))}  &  =\frac
{1}{n(1+\gamma(X_{i}^{2}-4x_{i}u+2u^{2}))}\geq0,\\
\gamma(X_{i}^{2}-4X_{i}u+2u^{2})  &  \geq\frac{1-n}{n},\quad i=1,...,n,
\end{align*}
and if such a solution exists, it is unique; see Qin and Lawless (1994). In
our empirical study, we solved the system of these two equations by using the
NAG subroutine in Fortran, \texttt{C05PBF}.

To compare the exact coverage probabilities of the confidence intervals for
the coefficient of variation based\ on some empirical phi-divergence test
statistics, a simulation study was conducted separately for continuous and
discrete distributions since we found that the rate of convergence to the
asymptotic distribution\ is much faster for discrete distributions (Poisson)
than for continuous distributions (normal and Student $t$). Based on $100,000$
samples of sizes $30$, $45$, $60$, $75$, $90$, $105$ from $\mathcal{N}(1,1)$
and $1+\sqrt{0.6}t_{5}$ distributions and sample sizes $15$, $20$, $25$, $30$
from Poisson, $\mathcal{P}(1)$, the so-called power-divergence measures (see
Cressie and Read (1984)) were considered to construct the phi-divergence test
statistics,
\begin{align*}
T_{n}^{\lambda}(\widehat{\boldsymbol{\theta}}_{E,n}%
,\widetilde{\boldsymbol{\theta}}_{E,n})  &  =\frac{2}{\lambda(1+\lambda
)}\left(
{\textstyle\sum\limits_{i=1}^{n}}
(1+\widetilde{\gamma}(X_{i}^{2}-4X_{i}\widetilde{u}+2\widetilde{u}%
^{2}))^{\lambda}-n\right)  ,\quad\text{if }\lambda(1+\lambda)\neq0,\\
&  =2%
{\textstyle\sum\limits_{i=1}^{n}}
\log\left(  1+\widetilde{\gamma}(X_{i}^{2}-4X_{i}\widetilde{u}+2\widetilde{u}%
^{2})\right)  ,\quad\text{if }\lambda=0,\\
&  =-2%
{\textstyle\sum\limits_{i=1}^{n}}
\frac{\log\left(  1+\widetilde{\gamma}(X_{i}^{2}-4X_{i}\widetilde{u}%
+2\widetilde{u}^{2})\right)  }{1+\widetilde{\gamma}(X_{i}^{2}-4X_{i}%
\widetilde{u}+2\widetilde{u}^{2})},\quad\text{if }\lambda=-1,
\end{align*}
that is, for each $\lambda\in%
\mathbb{R}
$, we have a different divergence measure by taking $\phi=\phi_{\lambda}$,
where $\phi_{\lambda}(x)$\ is as defined earlier in Section \ref{Simulation}.
In addition, the empirical generalized Wald test statistic, the empirical
generalized score test statistic, and the empirical Lagrange multiplier test
statistic were also obtained. For this purpose, the following auxiliary
matrices are necessary:%
\begin{align*}
\boldsymbol{V}\left(  \boldsymbol{\theta}\right)   &  =\left(  \boldsymbol{S}%
_{12}\left(  \boldsymbol{\theta}\right)  ^{T}\boldsymbol{S}_{11}^{-1}\left(
\boldsymbol{\theta}\right)  \boldsymbol{S}_{12}\left(  \boldsymbol{\theta
}\right)  \right)  ^{-1},\\
\boldsymbol{C}(\boldsymbol{\theta})  &  =\frac{\partial}{\partial
\boldsymbol{\theta}^{T}}\boldsymbol{c}(\boldsymbol{\theta})=%
\begin{pmatrix}
\frac{\partial}{\partial u}(v-2u^{2}), & \frac{\partial}{\partial v}(v-2u^{2})
\end{pmatrix}
=%
\begin{pmatrix}
-4u, & 1
\end{pmatrix}
,\\
\boldsymbol{Q}(\boldsymbol{\theta})  &  =\left(  \boldsymbol{C}%
(\boldsymbol{\theta})\boldsymbol{V}(\boldsymbol{\theta})\boldsymbol{C}%
(\boldsymbol{\theta})^{T}\right)  ^{-1},
\end{align*}
where $\boldsymbol{S}_{11}\left(  \boldsymbol{\theta}\right)  =E\left[
\boldsymbol{g}\left(  \boldsymbol{X},\boldsymbol{\theta}\right)
\boldsymbol{g}\left(  \boldsymbol{X},\boldsymbol{\theta}\right)  ^{T}\right]
$ and $\boldsymbol{S}_{12}\left(  \boldsymbol{\theta}\right)  =E\left[
\boldsymbol{G}_{\boldsymbol{X}}(\boldsymbol{\theta})\right]  $ are unknown.
The matrices $\boldsymbol{S}_{11}\left(  \boldsymbol{\theta}\right)  $\ and
$\boldsymbol{S}_{12}\left(  \boldsymbol{\theta}\right)  $\ can be replaced by
any consistent estimator
\begin{align*}
\widehat{\boldsymbol{S}}_{12}\left(  \boldsymbol{\theta}\right)   &  =\frac
{1}{n}\sum_{i=1}^{n}\boldsymbol{g}(X_{i},\boldsymbol{\theta})\boldsymbol{g}%
(X_{i},\boldsymbol{\theta})^{T}=%
\begin{pmatrix}
\frac{1}{n}\sum\limits_{i=1}^{n}(X_{i}-u)^{2} & \frac{1}{n}\sum\limits_{i=1}%
^{n}(X_{i}-u)(X_{i}^{2}-v)\\
\frac{1}{n}\sum\limits_{i=1}^{n}(X_{i}-u)(X_{i}^{2}-v) & \frac{1}{n}%
\sum\limits_{i=1}^{n}(X_{i}^{2}-v)^{2}%
\end{pmatrix}
,\\
\widehat{\boldsymbol{S}}_{11}\left(  \boldsymbol{\theta}\right)   &
=\frac{\partial}{\partial\boldsymbol{\theta}^{T}}\boldsymbol{\bar{g}}%
_{n}(\boldsymbol{\theta})=-\boldsymbol{I}_{2},
\end{align*}
and also $\boldsymbol{V}\left(  \boldsymbol{\theta}\right)  $ and
$\boldsymbol{Q}(\boldsymbol{\theta})$ by%
\begin{align*}
\widehat{\boldsymbol{V}}(\boldsymbol{\theta})  &  =\widehat{\boldsymbol{S}%
}_{11}^{-1}\left(  \boldsymbol{\theta}\right)  \widehat{\boldsymbol{S}}%
_{12}\left(  \boldsymbol{\theta}\right)  \widehat{\boldsymbol{S}}_{11}%
^{-1}\left(  \boldsymbol{\theta}\right)  =\widehat{\boldsymbol{S}}_{12}\left(
\boldsymbol{\theta}\right)  ,\\
\widehat{\boldsymbol{Q}}(\boldsymbol{\theta})  &  =\left(  \boldsymbol{C}%
(\boldsymbol{\theta})\widehat{\boldsymbol{V}}(\boldsymbol{\theta
})\boldsymbol{C}(\boldsymbol{\theta})^{T}\right)  ^{-1}\\
&  =\frac{n}{\sum_{i=1}^{n}\left(  X_{i}^{4}-2vX_{i}^{2}+v^{2}\right)
-8u\sum_{i=1}^{n}\left(  -vX_{i}+X_{i}^{3}-uX_{i}^{2}+uv\right)  +16u^{2}%
\sum_{i=1}^{n}\left(  -2uX_{i}+X_{i}^{2}+u^{2}\right)  }.
\end{align*}
The unconstrained maximum likelihood estimators were calculated as the
solution of the system of $p=2$ equations:%
\begin{align*}
\frac{1}{n}\sum_{i=1}^{n}g_{1}(X_{i},u,v)  &  =\frac{1}{n}\sum_{i=1}^{n}%
(X_{i}-u)=0,\\
\frac{1}{n}\sum_{i=1}^{n}g_{2}(X_{i},u,v)  &  =\frac{1}{n}\sum_{i=1}^{n}%
(X_{i}^{2}-v)=0,
\end{align*}
that is,
\[
\widehat{u}=\frac{1}{n}\sum_{i=1}^{n}X_{i},\qquad\widehat{v}=\frac{1}{n}%
\sum_{i=1}^{n}X_{i}^{2}.
\]
First, the expression of the empirical generalized Wald test statistic is
obtained as
\[
W_{n}(\widehat{\boldsymbol{\theta}}_{E,n})=nc(\widehat{\boldsymbol{\theta}%
}_{E,n})\widehat{\boldsymbol{Q}}(\widehat{\boldsymbol{\theta}}_{E,n}%
)c(\widehat{\boldsymbol{\theta}}_{E,n})=\frac{n\left(  2\widehat{u}%
^{2}-\widehat{v}\right)  ^{2}}{\frac{1}{n}\sum_{i=1}^{n}X_{i}^{4}%
-8\widehat{u}\frac{1}{n}\sum_{i=1}^{n}X_{i}^{3}-\widehat{v}^{2}+24\widehat{u}%
^{2}\widehat{v}-16\widehat{u}^{4}};
\]
next, the empirical generalized score test statistic is obtained as
\begin{align*}
S_{n}(\widetilde{\boldsymbol{\theta}}_{E,n})  &  =n\boldsymbol{\bar{g}}%
_{n}(\widetilde{\boldsymbol{\theta}}_{E,n})^{T}\widehat{\boldsymbol{V}%
}(\widetilde{\boldsymbol{\theta}}_{E,n})\boldsymbol{C}%
(\widetilde{\boldsymbol{\theta}}_{E,n})^{T}\widehat{\boldsymbol{Q}%
}(\widetilde{\boldsymbol{\theta}}_{E,n})\boldsymbol{C}%
(\widetilde{\boldsymbol{\theta}}_{E,n})\widehat{\boldsymbol{V}}%
(\widetilde{\boldsymbol{\theta}}_{E,n})\boldsymbol{\bar{g}}_{n}%
(\widetilde{\boldsymbol{\theta}}_{E,n})\\
&  =\frac{n\left[  4\widetilde{u}(\widetilde{u}-\widehat{u})+(\widehat{v}%
-\widetilde{v})\right]  ^{2}}{\frac{1}{n}\sum_{i=1}^{n}X_{i}^{4}%
-8\widetilde{u}\frac{1}{n}\sum_{i=1}^{n}X_{i}^{3}-2\widehat{v}\widetilde{v}%
+\widetilde{v}^{2}+8\widehat{u}\widetilde{u}\widetilde{v}+24\widetilde{u}%
^{2}\widehat{v}-8\widetilde{u}^{2}\widetilde{v}-32\widehat{u}\widetilde{u}%
^{3}+16\widetilde{u}^{4}}\\
&  =\frac{n\left[  \widehat{v}+\widetilde{v}-4\widetilde{u}\widehat{u}\right]
^{2}}{\frac{1}{n}\sum_{i=1}^{n}X_{i}^{4}-8\widetilde{u}\frac{1}{n}\sum
_{i=1}^{n}X_{i}^{3}-8\widehat{u}\widetilde{u}\widetilde{v}+10\widehat{v}%
\widetilde{v}+\widetilde{v}^{2}},
\end{align*}
where the second equality is obtained by taking into account that
$\widetilde{v}=2\widetilde{u}^{2}$; finally, the empirical Lagrange multiplier
test statistic is obtained as
\[
\lambda_{n}(\widetilde{\boldsymbol{\theta}}_{E,n})=n\widetilde{\gamma
}\widehat{\boldsymbol{Q}}^{-1}(\widetilde{\boldsymbol{\theta}}_{E,n}%
)\widetilde{\gamma}=n\widetilde{\gamma}^{2}\left(  \frac{1}{n}\sum_{i=1}%
^{n}X_{i}^{4}-8\widetilde{u}\frac{1}{n}\sum_{i=1}^{n}X_{i}^{3}-8\widehat{u}%
\widetilde{u}\widetilde{v}+10\widehat{v}\widetilde{v}+\widetilde{v}%
^{2}\right)  .
\]
%

\begin{table}[htbp] \scriptsize\centering
\begin{tabular}
[c]{ccccccccccc}\hline
$n$ & Nom. level & $T_{n}^{-1}$ & $T_{n}^{-1/2}$ & $T_{n}^{0}$ & $T_{n}^{2/3}$
& $T_{n}^{1}$ & $T_{n}^{2}$ & $W_{n}$ & $S_{n}$ & $\lambda_{n}$\\\hline
30 & 0.90 & 0.8537 & 0.8615 & 0.8672 & 0.8717 & \textbf{0.8723} & 0.8683 &
0.8676 & 0.8671 & 0.8689\\
45 & 0.90 & 0.8682 & 0.8745 & 0.8789 & 0.8828 & 0.8836 & 0.8811 & 0.8778 &
0.8780 & \textbf{0.8842}\\
60 & 0.90 & 0.8775 & 0.8824 & 0.8860 & 0.8895 & 0.8908 & 0.8899 & 0.8852 &
0.8854 & \textbf{0.8930}\\
75 & 0.90 & 0.8815 & 0.8856 & 0.8892 & 0.8923 & 0.8932 & 0.8922 & 0.8870 &
0.8880 & \textbf{0.8962}\\
90 & 0.90 & 0.8851 & 0.8892 & 0.8919 & 0.8947 & 0.8955 & 0.8951 & 0.8900 &
0.8900 & \textbf{0.8996}\\
105 & 0.90 & 0.8870 & 0.8905 & 0.8938 & 0.8965 & 0.8976 & 0.8976 & 0.8916 &
0.8917 & \textbf{0.9009}\\\hline
30 & 0.95 & 0.9062 & 0.9143 & 0.9210 & 0.9251 & \textbf{0.9258} & 0.9211 &
0.9228 & 0.9169 & 0.9149\\
45 & 0.95 & 0.9199 & 0.9265 & 0.9312 & 0.9352 & \textbf{0.9358} & 0.9321 &
0.9308 & 0.9275 & 0.9285\\
60 & 0.95 & 0.9286 & 0.9342 & 0.9392 & 0.9425 & \textbf{0.9426} & 0.9401 &
0.9374 & 0.9336 & 0.9360\\
75 & 0.95 & 0.9322 & 0.9379 & 0.9417 & \textbf{0.9446} & \textbf{0.9446} &
0.9427 & 0.9399 & 0.9367 & 0.9398\\
90 & 0.95 & 0.9350 & 0.9398 & 0.9432 & 0.9457 & \textbf{0.9460} & 0.9443 &
0.9415 & 0.9387 & 0.9419\\
105 & 0.95 & 0.9372 & 0.9417 & 0.9449 & 0.9477 & \textbf{0.9480} & 0.9464 &
0.9428 & 0.9403 & 0.9447\\\hline
\end{tabular}
$\ \ $%
\caption{Simulated values of interval coverage probabilities with underlying $\mathcal{N}(1,1)$
distribution.\label{t2.1}}%
\end{table}%
%

\begin{table}[htbp] \scriptsize\centering
\begin{tabular}
[c]{ccccccccccc}\hline
$n$ & Nom. level & $T_{n}^{-1}$ & $T_{n}^{-1/2}$ & $T_{n}^{0}$ & $T_{n}^{2/3}$
& $T_{n}^{1}$ & $T_{n}^{2}$ & $W_{n}$ & $S_{n}$ & $\lambda_{n}$\\\hline
30 & 0.90 & 0.8003 & 0.8086 & 0.8145 & 0.8175 & 0.8174 & 0.8130 &
\textbf{0.8302} & 0.8148 & 0.8208\\
45 & 0.90 & 0.8232 & 0.8306 & 0.8356 & 0.8387 & 0.8384 & 0.8343 &
\textbf{0.8465} & 0.8333 & 0.8417\\
60 & 0.90 & 0.8354 & 0.8424 & 0.8473 & 0.8489 & 0.8488 & 0.8446 &
\textbf{0.8561} & 0.8446 & 0.8526\\
75 & 0.90 & 0.8440 & 0.8500 & 0.8541 & 0.8568 & 0.8573 & 0.8535 &
\textbf{0.8621} & 0.8531 & 0.8597\\
90 & 0.90 & 0.8517 & 0.8571 & 0.8610 & 0.8631 & 0.8629 & 0.8586 &
\textbf{0.8683} & 0.8603 & 0.8645\\
105 & 0.90 & 0.8577 & 0.8627 & 0.8657 & 0.8668 & 0.8665 & 0.8621 &
\textbf{0.8737} & 0.8654 & 0.8665\\\hline
30 & 0.95 & 0.8584 & 0.8696 & 0.8776 & 0.8825 & 0.8822 & 0.8740 &
\textbf{0.8897} & 0.8665 & 0.8709\\
45 & 0.95 & 0.8807 & 0.8906 & 0.8976 & 0.9012 & 0.9008 & 0.8938 &
\textbf{0.9037} & 0.8847 & 0.8923\\
60 & 0.95 & 0.8925 & 0.9007 & 0.9074 & 0.9103 & 0.9098 & 0.9026 &
\textbf{0.9107} & 0.8958 & 0.9004\\
75 & 0.95 & 0.9013 & 0.9088 & 0.9140 & 0.9165 & 0.9159 & 0.9087 &
\textbf{0.9169} & 0.9040 & 0.9072\\
90 & 0.95 & 0.9080 & 0.9150 & 0.9199 & 0.9218 & 0.9209 & 0.9146 &
\textbf{0.9220} & 0.9108 & 0.9112\\
105 & 0.95 & 0.9135 & 0.9201 & 0.9234 & 0.9244 & 0.9235 & 0.9160 &
\textbf{0.9265} & 0.9163 & 0.9125\\\hline
\end{tabular}
$\ \ $%
\caption{Simulated values of interval coverage probabilities with underlying $1+\sqrt{0.6}t_{5}$
distribution.\label{t2.2}}%
\end{table}%
%

\begin{table}[htbp] \scriptsize\centering
\begin{tabular}
[c]{ccccccccccc}\hline
$n$ & Nom. level & $T_{n}^{-1}$ & $T_{n}^{-1/2}$ & $T_{n}^{0}$ & $T_{n}^{2/3}$
& $T_{n}^{1}$ & $T_{n}^{2}$ & $W_{n}$ & $S_{n}$ & $\lambda_{n}$\\\hline
15 & 0.90 & 0.8856 & 0.8878 & 0.9010 & 0.9089 & 0.9089 & 0.9071 & 0.8962 &
0.9075 & \textbf{0.9099}\\
20 & 0.90 & 0.8814 & 0.8849 & 0.8875 & 0.8904 & 0.8925 & 0.8923 &
\textbf{0.9074} & 0.8962 & 0.9013\\
25 & 0.90 & 0.8758 & 0.8793 & 0.8874 & 0.8934 & 0.8930 & 0.8923 &
\textbf{0.9048} & 0.8882 & 0.8974\\
30 & 0.90 & 0.8790 & 0.8848 & 0.8861 & 0.8916 & 0.8923 & 0.8901 &
\textbf{0.9027} & 0.8933 & 0.8928\\\hline
15 & 0.95 & 0.9329 & 0.9392 & 0.9529 & 0.9606 & 0.9613 & 0.9639 & 0.9546 &
0.9440 & \textbf{0.9670}\\
20 & 0.95 & 0.9359 & 0.9452 & 0.9486 & 0.9522 & 0.9539 & \textbf{0.9544} &
0.9501 & 0.9505 & 0.9519\\
25 & 0.95 & 0.9336 & 0.9388 & 0.9409 & 0.9434 & 0.9438 & 0.9442 &
\textbf{0.9502} & 0.9468 & 0.9472\\
30 & 0.95 & 0.9318 & 0.9336 & 0.9406 & 0.9442 & 0.9449 & 0.9446 &
\textbf{0.9519} & 0.9408 & 0.9466\\\hline
\end{tabular}
$\ \ $%
\caption{Simulated values  of interval coverage probabilities with underlying Poisson $\mathcal{P}(1)$
distribution.\label{t2.3}}%
\end{table}%

Focusing on the coefficients of variation of the two continuous distributions,
the exact coverage probabilities of the confidence intervals based on the
empirical power divergence test statistics with $\lambda\in\{-1,\frac{1}%
{2},0,\frac{2}{3},1,2\}$ are presented in Tables \ref{t2.1} and \ref{t2.2}
when the nominal coverage probability based on the asymptotic distribution is
either $90\%$ or $95\%$. In addition, the empirical generalized Wald test
statistic, the empirical generalized score test statistic and the empirical
Lagrange multiplier test statistic are also considered for comparative
purposes. From these results, we note that the empirical likelihood ratio test
is not satisfactory and that among the empirical power divergence test
statistics there is a good choice, for the underlying normal distribution, in
the empirical chi-squared test statistic ($\lambda=1$), and for the non-normal
underlying distribution in the empirical Cressie-Read test statistic
($\lambda=\frac{2}{3}$), but there is not much difference between their
performance. If we consider other test statistics, for the normally
distributed observations, with theoretical asymptotic coverage as $95\%$, the
empirical chi-squared test statistic ($\lambda=1$) while for $90\%$ coverage
probability, the empirical Lagrange multiplier test statistic are seen to be
the best ones but there is very little difference with respect to the
empirical chi-squared test statistic ($\lambda=1$). For the non-normal
distribution, the empirical Wald test statistic is slightly superior than the
empirical Cressie-Read test statistic,\ while the opposite seems to be the
case for the normal distribution. Since in practice we do not know the form of
the underlying distribution, based on this simulation study, we would
recommend the use of either the empirical Cressie-Read test statistic\ or the
empirical Wald test statistic. In Table \ref{t2.3},\ the same study has been
carried out but for the case of Poisson distribution. The empirical likelihood
ratio test is once again unsatisfactory and that the empirical Cressie-Read
test statistic and the empirical chi-squared test statistic have a better
coverage probability close to the nominal level. Further, the empirical Wald
test statistic seems to be slightly superior than the empirical Cressie Read
test statistic since it has a greater coverage probability in 5 of the 8 cases.

\section{Concluding remarks\label{LastSec}}

In a non-parametric setting, we have proposed here a broad family of empirical
test statistics based on $\phi$-divergence measures, first for a simple null
hypothesis and then for a composite null hypothesis. Through numerical
examples and simulations, it has been shown that confidence intervals
constructed thought the empirical $\phi$-divergence based statistics improve
the coverage probability of the empirical likelihood ratio test slightly.
However, the most promising advantage of this new family of test statistics is
that some members outperform the empirical likelihood ratio test in the
presence of some shifted observations in the data. The approximation of the
power function based on a specific sample provides an insight about the most
appropriate robust $\phi$-divergence statistic. These robust statistics tend
to yield narrower confidence intervals in comparison to the empirical
likelihood ratio test.

The development of these empirical $\phi$-divergence tests in the two-sample
and multi-sample situations will be of great interest. We are currently
working in this direction and hope to report these findings in a future
paper.


\newpage

\begin{thebibliography}{99}                                                                                               %


\bibitem {baggerly}Baggerly, K. A. (1998). Empirical likelihood as a
goodness-of-fit measure. \emph{Biometrika}, \textbf{85}, 535--547.

\bibitem {barnet}Barnett, V. and Lewis, T. (1994). \emph{Outliers in
Statistical Data}. Third Edition. John Wiley \& Sons, Chichester, England.

\bibitem {basu}Basu, A., Shioya, H. and Park, C. (2011). \emph{Statistical
Inference: The Minimum Distance Approach}\textit{. }Chapman \& Hall/CRC Press,
Boca Raton, Florida.

\bibitem {bata}Bhattacharyya, A. (1943). On a measure of divergence between
two statistical populations defined by their probability distributions.
\emph{Bulletin of the Calcutta Mathematical Society}, 35, 99--109.

\bibitem {broni}Broniatowski, M. and Keziou, A. (2012). Divergences and
duality for estimating and test under moment condition models. \emph{Journal
of Statistical Planning and Inference, }\textbf{142}, 2554-2573.

\bibitem {cressie}Cressie, N. and Read, T. R. C. (1984). Multinomial
goodness-of-fit tests. \emph{Journal of the Royal Statistical Society, }Series
B, \textbf{46}, 440--464.

\bibitem {Fer}Ferguson, T. S. (1996).\emph{ A Course in Large Sample Theory}.
Chapman and Hall, New York.

\bibitem {Fraser}Fraser, D. A. S. (1957). \textit{Nonparametric Methods in
Statistics. }John Wiley \& Sons, New York.

\bibitem {gokhale}Gokhale, D. V. and Kullback, S. (1978). \emph{The
Information in Contingency Tables}. Marcel Dekker, New York.

\bibitem {ha}H\'{a}jek, J. and Sid\'{a}k, Z. (1967). \emph{Theory of Rank
Tests. }Academic Press, New York.

\bibitem {Heritier}Heritier, S. and Ronchetti, E. (1994). Robust
Bounded-Influence Tests in General Parametric Models. \emph{Journal of the
American Statistical Association}, 89, 897--904.

\bibitem {Lecam}Le Cam, L. (1960). \emph{Locally Asymptotic Normal Families of
Distributions. } Universality of California Press, Berkeley, California.

\bibitem {Menendez1}Men\'{e}ndez, M. L., Morales, D., Pardo, L. and
Salicr\'{u}, M. (1995). Asymptotic behavior and statistical applications of
divergence measures in multinomial populations: A unified study.
\emph{Statistical Papers, }\textbf{36, }1-29.

\bibitem {Menendez2}Men\'{e}ndez, M. L., Pardo, J. A., Pardo, L. and Pardo, M.
C. (1997). Asymptotic approximations for the distributions of the $(h,\phi
)$-divergence goodness-of-fit statistics: Applications to R\'{e}nyi's
statistic. \emph{Kybernetes, }\textbf{26, }442-452.

\bibitem {M}Morales, D. and Pardo, L. (2001). Some approximations to power
functions of $\phi$-divergence tests in parametric models. \emph{TEST,
}\textbf{10}, 249-269.

\bibitem {fr}Owen, A. B. (1988). Empirical likelihood ratio confidence
interval for a single functional. \emph{Biometrika}, \textbf{75}, 308-313.

\bibitem {o2}Owen, A. B. (1990). Empirical likelihood confidence regions.
\emph{The Annals of Statistics}, \textbf{18}, 90-120.

\bibitem {Pardo}Pardo, L. (2006). \emph{Statistical Inference Based on
Divergence Measures}. Chapman \& Hall/ CRC Press, Boca Raton, Florida.

\bibitem {quin}Qin, J. and Lawless, J. (1994). Empirical likelihood and
general estimating equations. \emph{The Annals of Statistics}, \textbf{22}, 300-325.

\bibitem {raguete}Ragusa, G. (2011). Minimum Divergence, Generalized Empirical
likelihoods, and Higher Order Expansions. \emph{Econometric Reviews},
\textbf{30}, 4, 406-456.

\bibitem {Renyi}R\'{e}nyi, A. (1961). On measures of entropy and information.
\textit{Proceedings of the Fourth Berkeley Symposium on Mathematical
Statistics and Probability, }\textbf{1}, 547-561\textit{.}

\bibitem {scheena}Schennach, S. M. (2007). Point Estimation with Exponentially
Tilted Empirical Likelihood. \emph{The Annals of Statistics}, \textbf{35}, 634--672.

\bibitem {stigler}Stigler, S. M. (1973). Simon Newcomb, Percy Daniell, and the
history of robust estimation, 1885-1920. \emph{Journal of the American
Statistical Association}, \textbf{68}, 872-879.

\bibitem {shar}Sharma, B. D. and Mittal, D. P. (1997). New non-additive
measures of relative information. \emph{Journal of Combinatorics, Information
\& Systems Science}\textit{, }\textbf{2}\textit{, }122-133\textit{.}

\bibitem {toma0}Toma, A. (2009) \ Optimal robust M-estimators using
divergences. \emph{Statistics \& Probability Letters}, \textbf{79}, 1--5.

\bibitem {Toma}Toma, A. (2013) Robustness of dual divergence estimators for
models satisfying linear constraints. \emph{C. R. Acad. Sci. Paris}, Ser. I,
\textbf{351}, 311-316.

\bibitem {vanderVaart}van der Vaart, A. W. (2000). \emph{Asymptotic
Statistics}. Cambridge University Press, Cambridge.

\bibitem {bala}Voinov, V., Nikulin, M. S. and Balakrishnan, N. (2013).
\emph{Chi-Squared Goodness of Fit Tests with Applications}. Academic Press, Boston.
\end{thebibliography}
\end{document}